\documentclass[11pt]{amsart}

\usepackage{amsmath,
amssymb,
amscd,
amsfonts,
color,
enumitem,
epsfig,
float,
graphicx,
hyperref,
latexsym,
leftindex,
mathrsfs,
mathtools,
psfrag,
transparent,
url,
wasysym,
wrapfig,
tikz,
tikz-cd,
xr
}

\usepackage[normalem]{ulem}

\usepackage{pdflscape}
\usepackage{adjustbox}
\usepackage{tikz}
\usepackage{tikz-cd}
\usepackage{placeins}
\usepackage{afterpage}
\usepackage{lmodern}

 \usepackage[
textwidth=2.25cm,
textsize=tiny,
colorinlistoftodos]
{todonotes}

\usepackage{fullpage}
\usepackage[all]{xy}
\usepackage[T1]{fontenc}

\newtheorem{theorem}{Theorem}[section]

\newtheorem{proposition}[theorem]{Proposition}
\newtheorem{lemma}[theorem]{Lemma}

\theoremstyle{definition}
\newtheorem{definition}[theorem]{Definition}

\theoremstyle{remark}
\newtheorem{remark}[theorem]{Remark}

\theoremstyle{plain}
\newcommand{\thistheoremname}{}
\newtheorem{genericthm}[theorem]{\thistheoremname}

\newtheorem*{genericthm*}{\thistheoremname}
\newenvironment{namedthm*}[1]
  {\renewcommand{\thistheoremname}{#1}%
   \begin{genericthm*}}
  {\end{genericthm*}}
  
\newcommand\cC{\mathcal{C}}

\newcommand\cE{\mathcal{E}}

\newcommand\cL{\mathcal{L}}
\newcommand\cM{\mathcal{M}}

\newcommand\cX{\mathcal{X}}

\newcommand{\bN}{\mathbb{N}}
\newcommand{\bP}{\mathbb{P}}

\newcommand{\bR}{\mathbb{R}}
\newcommand{\bX}{\mathbb{X}}

\newcommand{\bZ}{\mathbb{Z}}

\newcommand{\bCP}{\mathbb{CP}}

\newcommand\bm{\mathbf{m}} 
\newcommand\bn{\mathbf{n}}

\newcommand\bmu{\boldsymbol{\mu}}

\newcommand{\sC}{\mathscr{C}}

\newcommand{\sO}{\mathscr{O}}

\newcommand{\on}{\operatorname}

\newcommand{\codim}{\on{codim}}

\newcommand\pt{\on{pt}}

\newcommand\id{\on{id}}

\newcommand{\im}{\on{im}}

\newcommand{\Fuk}{\on{Fuk}}

\newcommand{\Symp}{\mathsf{Symp}}

\newcommand{\Top}{\textsf{Top}}

\newcommand{\Ob}{\on{Ob}}
\newcommand{\Mor}{\on{Mor}}

\newcommand{\cell}{{\on{cell}}}

\newcommand{\lin}{{\on{lin}}}

\newcommand{\ul}{\underline}

\newcommand{\sr}{\stackrel}

\newcommand{\wt}{\widetilde}

\newcommand{\smotimes}{{\scriptstyle\otimes}}

\def\rT{{\rm T}}
\def\rd{{\rm d}}

\newcommand\quotient[2]{
        \mathchoice
            {
                \text{\raise1ex\hbox{$#1$}\Big/\lower1ex\hbox{$#2$}}
            }
            {
                 \text{\raise1ex\hbox{$#1$}\Big/\lower1ex\hbox{$\scriptstyle #2$}}
            }
            {
                #1\,/\,#2
            }
            {
                #1\,/\,#2
            }
    }

\makeatletter
\let\@cleartopmattertags\relax
\newcommand\articleend{%
  \enddoc@text
  \let\authors\@empty
  \let\shortauthors\@empty
  \let\contribs\@empty
  \let\xcontribs\@empty
  \let\toccontribs\@empty
  \let\addresses\@empty
  \let\thankses\@empty
  \newpage
 }
\let\@wraptoccontribs\wraptoccontribs
\makeatother

\newenvironment{itemlist}
   {
      \begin{list}
         {$\bullet$}
         {
            \setlength{\itemsep}{0.5ex}
            \setlength{\parsep}{0ex}
            \setlength{\leftmargin}{1em}
            \setlength{\parskip}{0ex}
            \setlength{\topsep}{0.5ex}
         }
   }
   {
      \end{list}
   }

\begin{document}

\begin{abstract}
This paper provides a blueprint for the construction of a symplectic $(A_\infty,2)$-category, \textsf{Symp}.
We develop two ways of encoding the information in \textsf{Symp} -- one topological, one algebraic.
The topological encoding is as an \emph{$(A_\infty,2)$-flow category}, which we define here.
The algebraic encoding is as a linear $(A_\infty,2)$-category, which we extract from the topological encoding.
In upcoming work, we plan to use the adiabatic Fredholm theory developed in \cite{bottman_wehrheim:adiabatic} to construct $\Symp$ as an $(A_\infty,2)$-flow category, which thus induces a linear $(A_\infty,2)$-category.

The notion of a linear $(A_\infty,2)$-category developed here goes beyond
the proposal of \cite{bottman_carmeli}.
The recursive structure of the 2-associahedra identifies faces with fiber products of 2-associahedra over associahedra, which led \cite{bottman_carmeli}
to associate operations to singular chains on 2-associahedra.
The innovation in our new definition of linear $(A_\infty,2)$-category is to extend the family of 2-associahedra to include all fiber products of 2-associahedra over associahedra.
This allows us to associate operations to cellular chains, which in particular enables us to produce a definition that involves only one operation in each arity, governed by a collection of \emph{$(A_\infty,2)$-equations}.
\end{abstract}

\title{Foundations of $(A_\infty,2)$-categories: from flow to linear}
\author{Nathaniel Bottman, Katrin Wehrheim}

\maketitle

\section{Introduction}
\label{sec:intro}

A linear $A_\infty$-category is a category over the operad $C_*^\cell(K) \coloneqq (C_*^\cell(K_r))_{r\geq 1}$, where $K_r$ is the $(r-2)$-dimensional associahedron.
In particular, there is an $r$-ary operation on morphisms associated to each cellular chain in $K_r$, and these operations satisfy coherences expressed by the commutativity of the squares
\begin{align}
\label{eq:A-infty_coherence}
\xymatrix{
C_*^\cell(K_{r-s+1})\otimes C_*^\cell(K_s)\otimes\Mor(X_0,X_1)\otimes\cdots\otimes\Mor(X_{r-1},X_r) \ar[r] \ar[d] & C_*^\cell(K_r)\otimes\cdots \ar[d] \\
C_*^\cell(K_{r-s+1})\otimes\cdots\otimes\Mor(X_i,X_{i+s})\otimes\cdots \ar[r] & \Mor(X_0,X_r).
}
\end{align}
The left vertical arrow indicates performing the $s$-ary composition
\begin{align}
C_*^\cell(K_s)\otimes\Mor(X_i,X_{i+1})\otimes\cdots\otimes\Mor(X_{i+s-1},X_{i+s})
\to
\Mor(X_i,X_{i+s}),
\end{align}
while the upper horizontal arrow indicates performing operadic composition in $C_*^\cell(K)$.
Specifically, the upper horizontal arrow uses the composition
\begin{align}
C_*^\cell(K_{r-s+1})\otimes C_*^\cell(K_s)
\to
C_*^\cell(K_{r-s+1}\times K_s)
\to
C_*^\cell(K_r),
\end{align}
where the first arrow is the Eilenberg--Zilber map and the second map uses the identification of the faces of the associahedra with 
Cartesian products of smaller associahedra.

In \cite{bottman_carmeli}, the first author and Carmeli proposed a definition of linear $(A_\infty,2)$-categories, with an eye toward defining $\Symp$, the Symplectic $(A_\infty,2)$-Category\footnote{See \S\ref{ssec:symp} or \cite[\S4]{abouzaid_bottman} for an overview of $\Symp$.}.
The idea of $\Symp$ is that objects are symplectic manifolds, that $\Mor(M_1,M_2)$ is defined to be the Fukaya category $M_1^-\times M_2$, and that there are composition maps on 2-morphisms associated to chains in 2-associahedra.
Therefore \cite{bottman_carmeli} aimed to define linear $(A_\infty,2)$-categories analogously to linear $A_\infty$-categories, with 2-associahedra -- introduced in \cite{bottman_2-associahedra,bottman_realizations} -- playing the role of associahedra.

As established in \cite[Thm.\ 4.1]{bottman_2-associahedra}, faces of 2-associahedra can be identified with
Cartesian products of \emph{fiber products} of 2-associahedra, with respect to the forgetful maps from 2-associahedra to associahedra.
Therefore \cite{bottman_carmeli} was unable to adapt the square \eqref{eq:A-infty_coherence} to the 2-associahedra; specifically, the issue is with the top arrow.
Instead, they used a version of \eqref{eq:A-infty_coherence} that used the Alexander--Whitney map, rather than the Eilenberg--Zilber map.
Furthermore, they were forced to associate operations to singular chains in 2-associahedra, rather than cellular chains.
For several reason, this is undesirable:
\begin{itemize}
\item
$\Symp$ would be difficult to implement using the definition of \cite{bottman_carmeli}, 
because it would require regularizing uncountably-infinitely-many moduli spaces compatibly.

\smallskip

\item
In a linear $A_\infty$-category, there is one operation for every finite sequence of objects, and the $A_\infty$-equations express the coherences satisfied by these operations.
The \cite{bottman_carmeli} definition of a linear $(A_\infty,2)$-category is much less concise of an algebraic structure -- besides there being uncountably-infinitely-many operations of each arity, there are no ``$(A_\infty,2)$-equations''.
\end{itemize}

\subsection{$(A_\infty,2)$-flow categories and linear $(A_\infty,2)$-categories}

The contribution of this work is to define the notion of an \emph{$(A_\infty,2)$-flow category} in Definition~\ref{def:A_infty-2-flow_category}
and to produce a new definition of a linear $(A_\infty,2)$-category in Definition~\ref{def:A_infty-2-category}.
The notion of an $(A_\infty,2)$-flow category is analogous to that of a flow category \cite{cohen-jones}.
(Also see \cite{porcelli_smith} for a recent example of flow categories being used in the context of Floer homotopy theory.)
Using the toolbox of adiabatic Fredholm theory, as developed by the authors in \cite{bottman_wehrheim:adiabatic}, we plan to construct $\Symp$ as an $(A_\infty,2)$-flow category.
Our new definition of a linear $(A_\infty,2)$-category ameliorates the two issues mentioned in bullet points at the end of the previous subsection as follows.
Recall the structural properties of 2-associahedra:

\medskip

\noindent
{\bf \textsc{(forgetful)} and \textsc{(recursive)} parts of Theorem 4.1, \cite{bottman_2-associahedra}, paraphrased.}
{\it The 2-associahedra $W_\bn$, for $r \geq 1$ and $\bn \in \bZ_{\geq0}^r$,
satisfy the following properties:
\begin{itemize}
\item[] \textsc{(forgetful)} $W_\bn$ is equipped with forgetful maps $\pi\colon W_\bn \to K_r$ to the $r$-th associahedron, which are surjective maps of posets.

\smallskip

\item[] \textsc{(recursive)} The faces of $W_\bn$ decompose canonically as Cartesian 
products of fiber products of smaller 2-associahedra, where the fiber products are with respect to the forgetful maps to the associahedra.
\end{itemize}
}

\begin{figure} \label{fig:degens of fiber product}
\centering
\def\svgwidth{0.8\columnwidth}
\begingroup%
  \makeatletter%
  \providecommand\color[2][]{%
    \errmessage{(Inkscape) Color is used for the text in Inkscape, but the package 'color.sty' is not loaded}%
    \renewcommand\color[2][]{}%
  }%
  \providecommand\transparent[1]{%
    \errmessage{(Inkscape) Transparency is used (non-zero) for the text in Inkscape, but the package 'transparent.sty' is not loaded}%
    \renewcommand\transparent[1]{}%
  }%
  \providecommand\rotatebox[2]{#2}%
  \newcommand*\fsize{\dimexpr\f@size pt\relax}%
  \newcommand*\lineheight[1]{\fontsize{\fsize}{#1\fsize}\selectfont}%
  \ifx\svgwidth\undefined%
    \setlength{\unitlength}{1805.05406652bp}%
    \ifx\svgscale\undefined%
      \relax%
    \else%
      \setlength{\unitlength}{\unitlength * \real{\svgscale}}%
    \fi%
  \else%
    \setlength{\unitlength}{\svgwidth}%
  \fi%
  \global\let\svgwidth\undefined%
  \global\let\svgscale\undefined%
  \makeatother%
  \begin{picture}(1,0.84681718)%
    \lineheight{1}%
    \setlength\tabcolsep{0pt}%
    \put(0.01288946,0.47105434){\color[rgb]{0,0,0}\makebox(0,0)[lt]{\lineheight{0}\smash{\begin{tabular}[t]{l} \end{tabular}}}}%
    \put(0,0){\includegraphics[width=\unitlength,page=1]{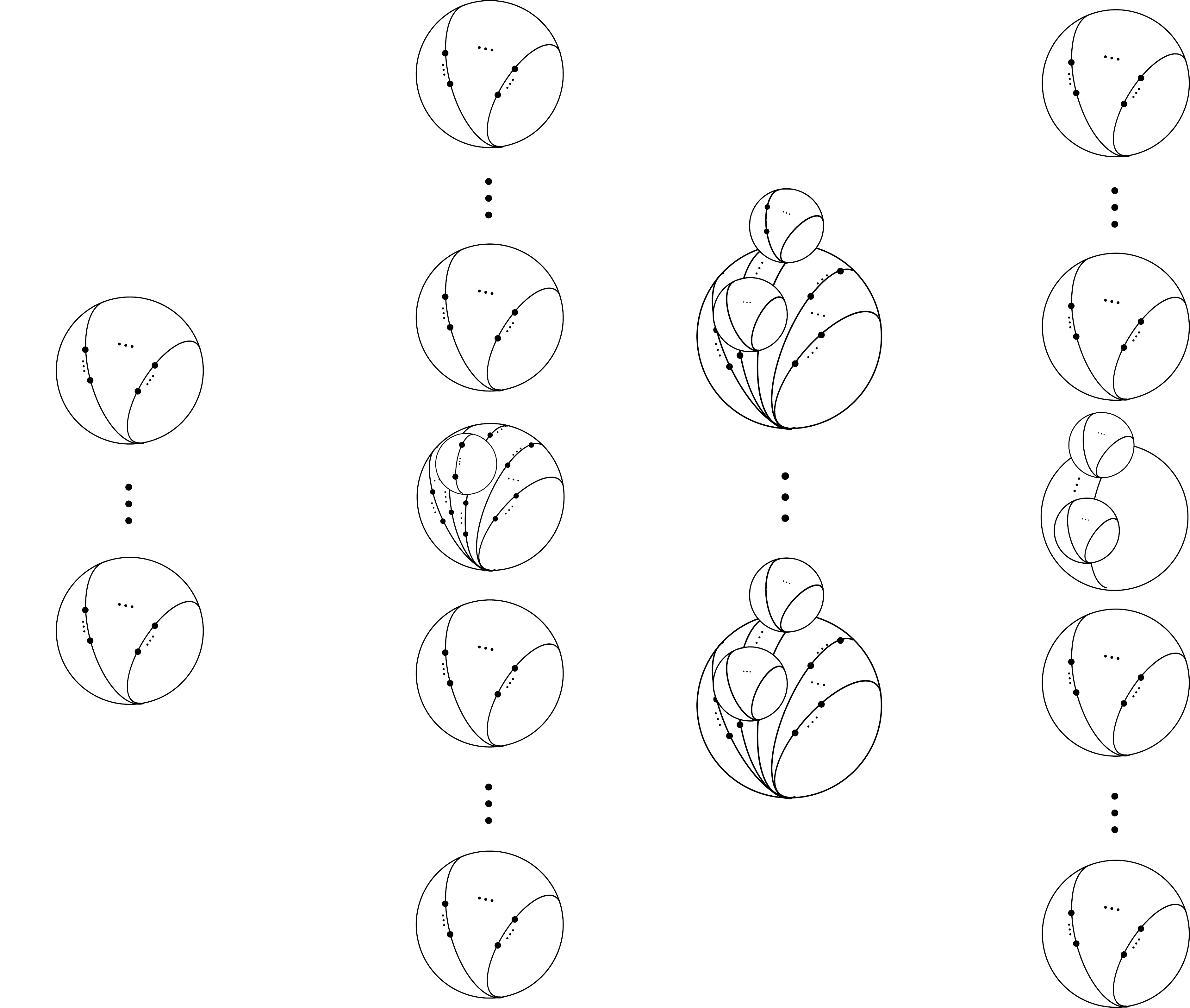}}%
    \put(-0.00217469,0.41603549){\color[rgb]{0,0,0}\makebox(0,0)[lt]{\lineheight{0}\smash{\begin{tabular}[t]{l}$\partial$\end{tabular}}}}%
    \put(0.22270194,0.41603549){\color[rgb]{0,0,0}\makebox(0,0)[lt]{\lineheight{0}\smash{\begin{tabular}[t]{l}$=$\end{tabular}}}}%
    \put(0.53412057,0.41603549){\color[rgb]{0,0,0}\makebox(0,0)[lt]{\lineheight{0}\smash{\begin{tabular}[t]{l}$\cup$\end{tabular}}}}%
    \put(0.77048657,0.41603549){\color[rgb]{0,0,0}\makebox(0,0)[lt]{\lineheight{0}\smash{\begin{tabular}[t]{l}$\cup$\end{tabular}}}}%
    \put(0,0){\includegraphics[width=\unitlength,page=2]{degens_of_fiber_product.pdf}}%
  \end{picture}%
\endgroup%

\caption{Here we depict the codimension-1 degenerations in an element of $\sO$ in a heuristic fashion.
On the left is a representative element of a fiber product of 2-associahedra.
There are three types of codimension-1 degenerations, which are depicted from left to right on the right-hand side: marked points on a single seam on a single sphere can collide; a proper subset of seams on all spheres can collide; or marked points on a single sphere can diverge to infinity -- which is equivalent to all seams on a single sphere colliding.
For details see Remark~\ref{rmk:three_types_of_comp_maps}.
}
\end{figure}

The \textsc{(recursive)} property was the reason that \cite{bottman_carmeli} was 
unable to adapt \eqref{eq:A-infty_coherence} to the context of $(A_\infty,2)$-categories.
However, there is a remedy for this problem: 
We expand the collection $(W_\bn)$ to the larger collection
\begin{align} \label{eq:O}
\sO \coloneqq \bigl(
W_{\bn^1} \times_{K_r} \cdots \times_{K_r} W_{\bn^a}
\bigr)_{\substack{r\geq 1, a\geq 1,\\\bn^1,\ldots,\bn^a \in 
\bZ_{\geq0}^a}}.
\end{align}
Again appealing to \cite[Theorem 4.1]{bottman_2-associahedra}, $\sO$ has the property that its
(codimension $1$) faces canonically decompose as Cartesian products of pairs of
elements of $\sO$ as depicted in Figure~\ref{fig:degens of fiber product}.
This enables our new Definition~\ref{def:A_infty-2-category} of linear $(A_\infty,2)$-categories, in which we associate operations to cellular chains in fiber products of 2-associahedra.
The two new definitions are related in that the chain complexes on an $(A_\infty,2)$-flow category form a linear $(A_\infty,2)$-category as follows.

\medskip

\noindent
{\bf Theorem \ref{thm:A-infinity_2_flow_to_linear}, paraphrased.}
{\it 
An $(A_\infty,2)$-flow category gives rise to a linear $(A_\infty,2)$-category.}

\begin{remark}
Our motivating example of an $(A_\infty,2)$-flow category is $\Symp$.
This is not, however, the only example of its kind:
\begin{itemize}
\item Recall that the original example of an $A_\infty$-space is any based loopspace $\Omega X$ \cite{stasheff1963I,stasheff1963II}.
In a similar way, Bottman--Carmeli showed that any continuous map $f\colon A \to X$ of pointed topological spaces gives rise to what they called a ``$\widetilde{(A_\infty,2)}$-space'', which, approximately, is an $(A_\infty,2)$-category in $\Top$ with a single object.

\item
In the forthcoming work \cite{manion_portwood}, Andrew Manion and Chris Portwood construct an uncurved linear $(A_\infty,2)$-category in which the objects, morphisms, and 2-morphisms are, respectively, uncurved $A_\infty$-categories, uncurved $A_\infty$-functors, and prenatural transformations.
This fulfills a longstanding expectation, and provides evidence that the definition of linear $(A_\infty,2)$-categories that we provide here is a/the ``right'' one.

Manion--Portwood's construction is a 2-categorical analogue of the fact that chain complexes form a dg-category, and in particular a linear $A_\infty$-category.
Their construction is also analogous to the fact that all (small) $(\infty,1)$-categories form an $(\infty,2)$-category \cite[Introduction to the Appendix, \S2.3]{gaitsgory2019study}.
\end{itemize}

\noindent
Continuing the line of thought in the latter bullet, we expect that our definition of $(A_\infty,2)$-flow categories in this paper can be extended to a definition of $(A_\infty,n)$-flow categories; we plan to construct this definition in a future paper.
This definition will be similar in spirit to Backman--Bottman--Poliakova's definition of $n$-associahedra in \cite{backman_bottman_poliakova}, following Bottman's definition of 2-associahedra in \cite{bottman_2-associahedra}.
We expect that Manion--Portwood's construction will extend to show that $(A_\infty,n-1)$-categories form an $(A_\infty,n)$-category.

Finally, we note that the $n$-associahedra underlie this future construction not just as a combinatorial scaffold, but as an operadic one.
Bottman--Carmeli showed in \cite{bottman_carmeli} that the 2-associahedra form an operad relative to the associahedra, and defined linear $(A_\infty,2)$-categories as categories over this relative 2-operad.
Similarly, we expect that the $n$-associahedra will form an operad relative to the $(n-1)$-associahedra, and that linear $(A_\infty,n)$-categories can be interpreted as categories over this relative $n$-operad.
\null\hfill$\triangle$
\end{remark}

\medskip

Before we end this introduction with an overview of the symplectic applications, we will specify some notational conventions and standing assumptions: 
We incorporate the bubbling phenomena in symplectic applications by allowing for unstable domain curves and thus including 2-associahedra whose realization in terms of witch curves \cite{bottman_realizations} allows for nondiscrete stabilizer subgroups.
We then interpret their algebraic impact as curvature generalizing \cite{fooo_12}.

\medskip
\begin{center}
\fbox{\parbox{0.8\columnwidth}{
\centering
All notions of $A_\infty$- and $(A_\infty,2)$-categories allow for nontrivial curvature.
}}
\end{center}
\medskip

\noindent
On the other hand, we simplify the algebraic exposition by the following assumption about chain complexes -- which goes along with the simplifying assumption of spaces being manifolds.
\medskip
\begin{center}
\fbox{\parbox{0.6\columnwidth}{
\centering
The coefficient ring (or field) $\Lambda$ has characteristic 2.
}}
\end{center}
\medskip

\subsection{Applications in Symplectic Geometry} \label{ssec:symp}

\subsubsection{From Weinstein's vision via Floer theory to Fukaya $A_\infty$-categories and the monotone symplectic 2-category}

In the early 1980s Weinstein \cite{weinstein_symplectic_category} challenged the nascent field of Symplectic Geometry/Topology to complete the construction of a {\bf symplectic ``category''} based on his observation that ``everything is Lagrangian'' and ought to form a category along the following lines:  
\begin{itemlist}
\item 
objects are symplectic manifolds, 
\item
morphisms $M_0 \overset{L_{01}}{\to} M_1$ are Lagrangian relations, i.e.\ Lagrangian submanifolds $L_{01}\subset M_0^-\times M_1$, 
\item
the natural geometric composition of general relations 
\begin{align} \label{eq:composition}
L_{01} \circ L_{12} &\; \coloneqq\; 
\bigl\{ (p_0, p_2 ) \,|\, \exists \, p_1\in M_1 : (p_0,p_1)\in L_{01}, (p_1,p_2)\in L_{12} \bigr\}  \\
&\;=\; 
 \pi_{02} \bigl( L_{01} \times L_{12}   \cap M_0 \times \Delta_{M_1} \times M_2 \bigr)
\nonumber
\end{align}
yields non-singular Lagrangians if the intersection with the diagonal $\Delta_M\coloneqq\{(p,p) \,|\, p\in M\} \subset  M\times M$ is transverse, and the restriction of the projection $\pi_{02}:M_0\times \Delta_{M_1}\times M_2  \to M_0\times M_2$ to the intersection $L_{01} \times_{M_1} L_{12}:=L_{01} \times L_{12}   \cap M_0 \times \Delta_{M_1} \times M_2$ is injective.
\end{itemlist}
A few years later, Gromov \cite{gromov} introduced the tool of {\bf pseudoholomorphic curves} into the field to establish the rigidity effects that make it a `Geometry' (e.g.\ a ball symplectically embeds into a cylinder only if it does so trivially) -- in contrast to the flexibility effects that make it a `Topology' (e.g.\ by Darboux theorem all symplectic structures in a fixed dimension are locally equivalent).
Since then, the study of algebraic structures derived from moduli spaces of pseudoholomorphic curves -- from Gromov-Witten invariants to Symplectic Field Theory \cite{eliashberg-givental-hofer-sft} and beyond, from Floer theory \cite{floer_lag_int} to Fukaya categories \cite{fooo_12} and beyond -- has grown Symplectic Geometry into a major subject with deep connections to low dimensional topology, algebraic geometry, and mathematical physics.
Again, Lagrangians appear prominently as the natural boundary condition for pseudoholomorphic curves, which allows to compactify and regularize the moduli spaces.

However, the only progress on Weinstein's challenge was his observation, published in \cite{guillemin_sternberg}, that Lagrangian subspaces of symplectic vector spaces always have a well defined geometric composition \eqref{eq:composition}, and thus do form a well defined linear symplectic category.
The next step towards Weinstein's challenge -- in ``monotone'' symplectic manifolds in which pseudoholomorphic curves are particularly well behaved -- was Wehrheim--Woodward's 2006 discovery \cite{wehrheim_woodward_2010functoriality} of the {\bf extended monotone symplectic 2-category}, which is roughly given as follows: 
\begin{itemlist}
\item 
objects are monotone symplectic manifolds, 
\item
$1$-morphisms $M_0 \overset{\cdots}{\to} \ldots \overset{\cdots}{\to} M_1$ are finite composable sequences of monotone Lagrangian relations, 
\item
composition of $1$-morphisms is given by concatenation of finite sequences, 
\item
$2$-morphisms are quilted Floer homology classes \cite{wehrheim_woodward_quilted}, 
\item
the higher structure maps are induced by moduli spaces of pseudholomorphic quilts with Lagrangian boundary and seam conditions \cite{wehrheim_woodward_jhol_quilts} -- based on the observation that Lagrangian relations $L_{01}\subset M_0^-\times M_1$ allow for a coupling between pseudoholomorphic curves in $M_0$ and $M_1$, 
\item
there are $2$-equivalences
$({\scriptstyle  \ldots \xrightarrow{L_{01}} M_1 \xrightarrow{L_{12}} \ldots } ) \simeq ( {\scriptstyle \ldots \xrightarrow{L_{01}\circ L_{12}} \ldots } )$ 
when the geometric composition \eqref{eq:composition} is ``embedded'' (i.e.\ $L_{01} \times_{M_1} L_{12}$ is cut out transversely and $\pi_{02}|_{L_{01} \times_{M_1} L_{12}}$ is injective) \cite{wehrheim_woodward_geometric_composition}.
\end{itemlist}
This construction contains a full resolution of Weinstein's challenge by allowing for non-monotone objects and morphisms, and taking equivalence classes induced by $({\scriptstyle  \ldots \xrightarrow{L_{01}} M_1 \xrightarrow{L_{12}} \ldots } ) \simeq ( {\scriptstyle \ldots \xrightarrow{L_{01}\circ L_{12}} \ldots } )$ whenever the geometric composition is well defined.\footnote{This construction in \cite{wehrheim_woodward_2010functoriality} can nowadays be understood as the categorification of a partially defined simplicial set.
Indeed, the categorification construction requires only $1$- and $2$-simplices.}
From the perspective of higher category theory, however, one ought to replace the quotient construction by an $(\infty,2)$-category whose $2$-morphisms encode the equivalence.
And from the perspective of modern Symplectic Geometry after Gromov, Floer, and Fukaya, the $2$-compositions in this higher categorical structure ought to arise from moduli spaces of pseudoholomorphic curves.
That is exactly what \cite{wehrheim_woodward_2010functoriality} did -- at the expense of restricting to monotone objects and morphisms to exclude a novel singularity formation in \cite{wehrheim_woodward_geometric_composition}: the novel figure eight bubbling phenomenon illustrated in \cite[Figure 15]{abouzaid_bottman}, which must be expected at the boundary of compactified moduli spaces.
Under the same assumptions, Ma'u--Wehrheim--Woodward \cite{mww} then lifted the main part of this construction\footnote{
The 2-category structure induces Yoneda embeddings 
$\Mor(M_1,M_2) \to {\rm Fun}(\Mor(\pt,M_1), \Mor(\pt,M_2) )$ 
that are functorial in the sense that they intertwine the composition of morphisms with the composition of functors.
} to the Floer chain level, resulting in $A_\infty$-functors $\Phi_{L_{12}} : \Fuk^\# (M_1) \to  \Fuk^\# (M_2)$ between extended Fukaya categories, which intertwine the geometric composition of Lagrangians with the composition of functors $\Phi_{L_{01}}\circ \Phi_{L_{12}} \simeq \Phi_{L_{01}\circ L_{12}}$ -- when \eqref{eq:composition} is embedded.
Here the (non-extended) {\bf Fukaya $\mathbf{A_\infty}$-category} ${\rm\mathbf{Fuk}}\mathbf{(M)}$ of a symplectic manifold $M$ is roughly defined as follows:
\begin{itemize}
\item 
Objects are Lagrangian submanifolds $L\subset M$, 
\item
morphisms are Floer chains, i.e.\ $\Mor(L,L')$ is generated by intersection points $L\cap L'$, 
\item
the $A_\infty$-structure maps are induced by moduli spaces of pseudholomorphic curves with Lagrangian boundary conditions, whose domains are given by the realizations of the associahedra in terms of disks with marked points on the boundary.
\end{itemize}
While these constructions initiated the development of categorical symplectic geometry as surveyed in \cite{abouzaid_bottman}, a general and rigorous foundation for its wide-ranging applications within Symplectic Geometry and connections to Algebraic Geometry and Low Dimensional Topology requires dropping the highly restrictive monotonicity assumptions, removing the extension construction -- i.e.\ defining composition directly by \eqref{eq:composition} -- and making the construction at chain-level.
This required a detailed understanding of the novel figure eight bubbling phenomenon, which was achieved analytically in \cite{bottman_wehrheim,bottman_figure_eight_singularity,bottman_wehrheim:adiabatic} and algebraically in \cite{bottman_2-associahedra,bottman_realizations,bottman_oblomkov,bottman_carmeli} and the present paper.
The resulting first prediction of the algebraic impact is \cite[Conj.4.11]{bottman_wehrheim}, which proposed that the geometric composition $(L_{01},L_{12})\mapsto L_{01}\circ L_{12}$ in \eqref{eq:composition} should lift to an $A_\infty$-bifunctor on non-extended Fukaya categories of cleanly-immersed Lagrangians, 
\begin{align}
\label{eq:bifunctor}
\bigl(\Fuk(M_0^-\times M_1),\Fuk(M_1^-\times M_2)\bigr)
\to
\Fuk(M_0^-\times M_2).
\end{align}
Restricting to $M_0=\pt$, this generalizes the \cite{mww} constructions to $A_\infty$-functors on non-extended Fukaya categories $\Fuk(M_1^-\times M_2) \to {\rm Fun}(\Fuk(M_1), \Fuk(M_2) )$, which are compatible with geometric composition \eqref{eq:composition} with
figure eight bubbling contributing to bounding cochains (see \S\ref{sss:MC}).

\subsubsection{Symplectic Examples of $(A_\infty,2)$-categorical structures}

Bottman's work on the combinatorics of the 2-associahedra -- a new family of polytopes motivated by the bubbling analysis in \cite{bottman_wehrheim} -- led to the notion of $(A_\infty,2)$-categories \cite{bottman_carmeli} and a {\bf blueprint for the Symplectic $\mathbf{(A_\infty,2)}$-category} $\rm\mathbf{Symp}$ in \cite[\S4]{abouzaid_bottman}, which contains \eqref{eq:bifunctor} and thus explains the chain-level functoriality \cite{mww} as a Yoneda construction.
This higher symplectic category $\Symp$ is now under construction by the present authors to fully resolve Weinstein's challenge of constructing a symplectic ``category'' -- in the natural context of Floer theory.
It generalizes and un-extends the monotone symplectic 2-category roughly as follows: 
\begin{itemlist}
\item 
objects are symplectic manifolds, 
\item
$1$-morphisms $M_0 \overset{L_{01}}{\to} M_1$ are cleanly immersed Lagrangian relations $L_{01} \looparrowright M_0^-\times M_1$, 
\item
composition of $1$-morphisms is given geometrically -- generalizing \eqref{eq:composition} by replacing the diagonal with the graph of a Hamiltonian symplectomorphism (which is Floer-theoretically inessential), 
\item
$2$-morphisms are Floer chains, i.e.\ $\leftindex^1\Mor(L_{01},L'_{01})$ is generated by the intersection points $L_{01}\cap L'_{01}$, 
\item
a bi-infinite collection of structure maps --  starting with \eqref{eq:bifunctor} -- is induced by moduli spaces of pseudholomorphic quilts 
whose domains are given by the witch curve realizations of the 2-associahedra.
\end{itemlist}
This blueprint for the construction of $\Symp$ is the motivating example for the notions of $(A_\infty,2)$-type categories that are developed in  \cite{bottman_carmeli} and this paper.
First, the natural construction of Gromov-compactified moduli spaces of pseudoholomorphic quilts yields compact topological spaces, which form the 2-composition data of an {\bf $\mathbf{(A_\infty,2)}$-category in spaces} as defined in \cite[Definition 3.6]{bottman_carmeli}.	
In the spirit of Definition~\ref{def:general A_infty-2-flow_category}, this structure might also be called an {\bf ``unregularized $\mathbf{(A_\infty,2)}$-flow category''}\footnote{
A formal definition of this notion would replace the ``regularization framework'' ($(*)$spaces, $(*)$maps) by (compact spaces, continuous maps) in Definition~\ref{def:general A_infty-2-flow_category}.
This makes sense of the structure and properties except for (iii) -- which could be dropped or replaced by a notion of ``expected boundary'' for the unregularized moduli spaces.
}
It is roughly given as follows:
\begin{itemlist}
\item
$(\Ob, \leftindex^1\Mor)=(\text{symplectic manifolds}, \text{Lagrangian relations})$ with composition \eqref{eq:composition} for appropriate choices of Hamiltonian symplectomorphisms forms a semi-simplicial set -- a generalized category.

\smallskip

\item
For every $L, L' \in  \leftindex^1\Mor$, there is an intersection space $\leftindex^2\Mor(L,L')=L\cap L'$.

\smallskip

\item
For any collection of objects $M_0,\ldots ,M_r\in\Ob$ and 1-morphisms
$\cL=\bigl(L_{(i-1)i}^{j,k} \in  \leftindex^1\Mor(M_{i-1},M_i) \bigr)$
there is a compactified moduli space $\overline\cM(\cL) $ with continuous evaluation maps $\alpha^{\cdots}_\cL, \beta^{\cdots}_\cL$ to intersection spaces and a forgetful map $p_\cL([(\ul \Sigma,\ul I, \ul u)])=[(\ul \Sigma,\ul I)]$ to a fiber product of 2-associahedra as specified in \eqref{eq:2-a-b-maps}.
Roughly speaking, such a compactified moduli space
$$
\overline\cM(\cL) \coloneqq \bigl\{ 
(\ul \Sigma,\ul I, \ul u) \,\big|\, [(\ul \Sigma,\ul I)]\in W_\bn,  u_i : \Sigma_i \to M_i , \overline\partial_{J_i} u_i =0,  (u_{i-1},u_i)(I_{ik})\subset L^k_{(i-1)i} 
\bigr\} / \sim
$$
consists of equivalence classes -- under reparameterization with biholomorphisms -- of 
\begin{itemize}
\item[$-$]
a collection of $r+1$ ``patches'' given by nodal Riemann surfaces $\ul\Sigma=(\Sigma_i)_{0\leq i \leq r}$ with boundary, 
\item[$-$]
``seams'' providing pairwise identifications of the boundary components $\partial\Sigma_{i-1} \hookleftarrow I_{ik}\hookrightarrow \partial \Sigma_i$, 
\item[$-$] 
maps $u_i:\Sigma_i \to M_i$ for each patch, which satisfy the Cauchy-Riemann PDEs $\overline\partial_{J_i} u_i =0$ and have finite symplectic area,\footnote{Finite area/energy $\int u_i^*\omega_{M_i}=\frac 12\int |\rd u_i|^2$ guarantees the existence of limits, which in turn define the maps $\alpha^{\cdots}_\cL, \beta^{\cdots}_\cL$.} 
\item[$-$]
lifts $I_{ik} \to L^k_{(i-1)i}$ of $(u_{i-1},u_i)|_{I_{ik}}: I_{ik}\to M_{i-1}\times M_i$ for each seam, which generalize the seam condition ``$(u_{i-1},u_i)(I_{ik})\subset L^k_{(i-1)i}$'' to immersions.
\end{itemize}
\end{itemlist}
As a second example, we expect to obtain a {\bf regularized $\mathbf{(A_\infty,2)}$-flow category} in the sense of Definition~\ref{def:A_infty-2-flow_category} by applying to the unregularized moduli spaces an appropriate choice of regularization framework as described in Definition~\ref{def:framework}.
This will replace each compactified moduli space $\overline\cM(\cL)$ (and its evaluation and forgetful maps) by a regularized space $\cX(\cL)$ such as a $\cC^0$-manifold (with continuous evaluation and forgetful maps).

As a third example, Theorem \ref{thm:A-infinity_2_flow_to_linear} then extracts from the regularized $(A_\infty,2)$-flow category a {\bf linear $\mathbf{(A_\infty,2)}$-category} as in Definition~\ref{def:A_infty-2-category}, roughly as follows:
\begin{itemlist}
\item
$(\Ob, \leftindex^1\Mor)=(\text{symplectic manifolds}, \text{Lagrangian relations})$ with composition \eqref{eq:composition} for appropriate choices of Hamiltonian symplectomorphisms remains the same as above.

\smallskip

\item
For every $L, L' \in  \leftindex^1\Mor$, the intersection space $L\cap L'$ is replaced by an appropriate chain complex
$\leftindex^2\Mor(L,L') \coloneqq C_* (L\cap L')$ with coefficients in some ring (or field) $\Lambda$.

\smallskip

\item
For each collection of 1-morphisms $\cL = \bigl(L_{(i-1)i}^{j,k}\bigr)_{(i,j,k)\in I^{\, r,a}_{\underline\bn}}$ a $\Lambda$-linear map is obtained by a push-pull construction with the evaluation maps from the regularized moduli space $\cX(\cL)$, 
\begin{align*}
\bmu^{r,a}_{\underline\bn}
\,:\: 
\bigotimes_{(i,j,k)}
C_*\bigl(L_{(i-1)i}^{j,k-1} \cap L_{(i-1)i}^{j,k}\bigr)
&\:\to\:
\bigotimes_{j}
C_*\bigl(L_{01}^{j,0}\circ\cdots L_{(r-1)r}^{j,0} \cap L_{01}^{j,n_1^j}\circ\cdots L_{(r-1)r}^{j,n_r^j} \bigr)
\nonumber \\
\bigotimes_{(i,j,k) } c_i^{j,k}  \:  & \:\mapsto\: \: \bigl(\beta^j_{\cL} \bigr)_* \bigl( a_{\cL}^{i,k} \bigr)^*  \bigl( \underset{(i,j,k) }{\times} c_i^{j,k}  \bigr).
\end{align*}
\end{itemlist}
This construction will in particular contain the Fukaya $A_\infty$-categories $\Fuk(M)=\leftindex^1\Mor(\pt,M)$ of all symplectic manifolds $M$ and the predicted bifunctors \eqref{eq:bifunctor}; see Lemma~\ref{lem:mor is A-infty cat} and Remark~\ref{rmk:compatible}.
The following subsections sketch some of the major applications of these structures.

\begin{remark}
In applications to moduli spaces of pseudoholomorphic objects, the notion of regularized space will depend on the choice of regularization theory and will result in fractional counts that require more general coefficient rings than the characteristic 2 case that we use in this paper.
We plan to give a non-partisan construction of $\Symp$ -- for use with any regularization theory -- by axiomatizing these choices of types of regularized spaces and chain complexes.
This, however, does not contribute to the algebraic exposition of the new categorical notions which are the point of this paper.
That is why we avoid the axiomatic language for now by assuming manifolds as regularized spaces, which allows for the simplifying assumption of characteristic 2 that avoids signs.
\null\hfill$\triangle$
\end{remark}

\subsubsection{Topological invariants constructed via Floer field theories}

The symplectic 2-category of \cite{wehrheim_woodward_2010functoriality} was originally developed to explain and add to the wealth of topological invariants -- such as Heegard-Floer theory \cite{ozsvath_szabo_HF_defn} -- of the type ``Floer theory applied to Lagrangians arising from the parts of a topological decomposition'' that originated from the Atiyah-Floer conjecture \cite{Atiyah-1988-new-invariants}.
For example, various dimensional reductions of gauge theories (such as Seiberg-Witten and Donaldson-Yang-Mills) associate to any Heegard decomposition $Y=H_0\cup_\Sigma H_1$ of a 3-manifold along a separating surface $\Sigma$ a pair of Lagrangians $L_{H_0},L_{H_1}\subset M_\Sigma$ in a symplectic manifold.
The conjecture of Atiyah, resp.\ the theorem of Ozsv\'ath-Szab\'o, is that the resulting Floer homology $HF(L_{H_0},L_{H_1})$ is in fact independent of the choice of the decomposition -- hence an invariant of the 3-manifold $Y$.
This construction principle for topological invariants was formalized in \cite{wehrheim_floer_field_philosophy} and explains the invariance by the higher categorical structure: Dimensional reduction of gauge theories yields functors from topological categories such as (2-manifolds, 3-cobordisms) to a symplectic category (symplectic manifolds, {\small suitably generalized} Lagrangians).
Then the 2-categorical structure of the latter induces a Yoneda functor to (categories, functors).
Their composition is a topological field theory functor (2-manifolds, 3-cobordisms) $\to$ (categories, functors) which factors through a symplectic category -- a ``Floer field theory''.
Here the Floer homology ${\rm HF}(L_{H_0},L_{H_1})=\leftindex^2\Mor(  {\scriptstyle M_\emptyset   \xrightarrow{L_{H_0}} M_\Sigma \xrightarrow{L_{H_1}}  M_\emptyset  })$ 
is the 2-morphism space associated to the decomposition of the closed 3-manifold -- which is independent of the choice of decomposition by functoriality of the Floer field theory.
 
This framework -- in monotone settings, based on \cite{wehrheim_woodward_2010functoriality} --  has already been utilized in \cite{abouzaid2019khovanov,manolescu2012floer,reza_khovanov,wehrheim_woodward_floer_field_theory,wehrheim2015floer} 
 to construct or identify invariants.
The analogous application to Heegard-Floer type invariants -- outlined in \cite{Auroux-Heegard,wehrheim_floer_field_philosophy}
based on \cite{perutz_lag_matching_I,perutz_lag_matching_II} -- 
requires an extension of the symplectic 2-category beyond the monotone case.
The resulting symplectic $(A_\infty,2)$-category under construction will moreover allow to construct an abundance of additional topological invariants and field theoretic structures \cite{manion_portwood_HF,wehrheim_floer_field_philosophy} by working at the algebraic and geometric level -- as the pseudoholomorphic curve analysis is packaged in the higher categorical structure.
First examples of such new topological invariants which require correction terms from figure eight bubbles can be found in \cite{cazassus2020correspondence, hedden_herald_kirk}.

\subsubsection{Functors on Fukaya categories}\label{sss:funfuk}

A special case of the $A_\infty$-bifunctor \eqref{eq:bifunctor} associates an $A_\infty$-functor $\Phi_{L_{01}}:\Fuk(M_1)\to\Fuk(M_2)$ to any Lagrangian in $M_1^-\times M_2$ -- and allows to compute the composition of such functors 
$\Phi_{L_{01}}\circ \Phi_{L_{12}} \simeq \Phi_{L_{01}\circ L_{12}}$ by studying the geometric composition $L_{01}\circ L_{12}\subset M_0^-\times M_2$.
For example, any Hamiltonian $G$-action on a symplectic manifold $M$ can be described in terms of a Lagrangian submanifold of  $\rT^*G \times M\times M^-$.
The resulting functor $\Fuk(\rT^*G) \to \Fuk(M^-\times M)$ can be used as a tool for computations in $\Fuk(M)$.
For instance, \cite{evans2019generating} shows that when $M$ is a toric variety, then the level set $\mu^{-1}\{0\}\subset M$ of the moment map split-generates an orthogonal factor of $\Fuk M$.

While these functors have not yet been rigorously available -- except for monotone cases with an artificial extension \cite{mww} -- this vision guided the developments sketched in the following subsections.
Each of these ``applications of the expected higher categorical structure'' had to find work-arounds for the absence or shortcomings of the foundations -- e.g.\ restricting the allowed Lagrangian relations to graphs of maps and split Lagrangians, or working with more easily defined $A_\infty$-bimodules \cite{mau_n-modules}.
This resulted in groundbreaking constructions -- though limited in scope and often based on quite complicated, lengthy, and unnatural constructions.
The benefit of a full construction of the symplectic $(A_\infty,2)$-category $\Symp$ will be a simplification and generalization of each of these developments.
And thus the benefit of the algebraic framework developed in this paper is that it allows for such developments based on an axiomatic description of $\Symp$ -- prior to the completion of its construction, and independent of the particularities of regularization frameworks.

\subsubsection{Generation criteria for Fukaya categories}

The most basic case of the above functoriality is the fact that the diagonal $\Delta_M \subset M^-\times M$ represents the identity functor $\Phi_{\Delta_M}={\rm Id}_{\Fuk(M)}$ on the Fukaya category of $M$.
Now the higher categorical structure -- after extension to formal sums -- allows to deduce that an equivalence $\Delta_M \simeq \sum L'_\alpha \times L_\alpha$ to a sum of split Lagrangians in $\Fuk(M^-\times M)$ -- a ``resolution of the diagonal'' -- directly implies that $\Fuk(M)$ is ``generated'' by the $L_\alpha\subset M$ in the sense that for any object $L$ of $\Fuk(M)$ we have\footnote{
Here only the meaning of $\Phi_{L'}(L)$ as an appropriate coefficient depends on the concrete higher categorical structure.
} 
$$\textstyle
L \;=\; L \circ \Delta_M  \;=\; \Phi_{\Delta_M}(L) \;\simeq\;  \sum \Phi_{L'_\alpha \times L_\alpha}(L)  \;=\; \sum \Phi_{L'_\alpha}(L) \;  L_\alpha .
$$
This line of argument was pioneered by Nadler \cite{nadler2009microlocal} and led to the generation criteria for Fukaya categories formulated by Ganatra \cite{ganatra_thesis} and Abouzaid \cite{abouzaid_generation}.

In \cite[\S5.7]{abouzaid_bottman}, Abouzaid--Bottman proposed applying the Barr--Beck theorem in the context of functors between Fukaya categories.
Specifically, given a Lagrangian correspondence $L_{12} \subset M_1^-\times M_2$ that gives rise to a functor $\Phi_{L_{12}}^\#\colon \Fuk^\#(M_1) \to \Fuk^\#(M_2)$, they sketched the definition of a ``quilted closed-closed map'', then explained how to use this map to formulate a hypothesis for when $\Fuk^\#(M_2)$ can be computed in terms of $\Fuk^\#(M_1)$ and the endofunctor $\Phi_{L_{12}^T}^\# \circ \Phi_{L_{12}}^\#$.
(In the special case $M_1 = \pt$, their conjectured ``symplectic Barr--Beck theorem'' reduces to a close variant of Abouzaid's generation criterion.)
While Abouzaid--Bottman's conjecture was formulated for extended Fukaya categories, we expect that our ongoing work will enable a version of this conjectured theorem for non-extended Fukaya categories.

\subsubsection{Proofs of Mirror Symmetry }

The above generation criteria enable an approach to establishing a Mirror Symmetry pair $M \sim M^\vee$ by exhibiting a generating set $(L_\alpha)$ in $\Fuk(M)$ with the same algebraic properties as a generating collection $(S_\alpha)$ of sheaves on the complex mirror $M^v$.
This approach to proving HMS, or close adaptations thereof, has been carried out in a number of contexts, e.g.:
\begin{itemize}
\item Sheridan, for Fano hypersurfaces in projective space \cite{sheridan_fano};
\item Sheridan--Smith, for Greene--Plesser mirrors \cite{sheridan_smith:hms_for_greene-plesser_mirrors};

\item Li, for open Riemann surfaces \cite{heather:thesis};

\item Cho--Hong--Lau, for the weighted projective line \cite{cho_hong_lau:hms_for_weighted_projective_line};

\item We, for special isogenous tori \cite{wu:hms_for_special_isogenous_tori}; and

\item Lekili--Ueda, for certain Milnor fibers \cite{lekili_ueda:hms_for_milnor_fibers_via_moduli}.
\end{itemize}

Finally, we mention Smith's work \cite{smith:pencils}.
Smith constructs an equivalence
\begin{align}
D^\pi\Fuk(\Sigma_g)
\simeq
D^\pi\Fuk(Q_0\cap Q_1;0),
\end{align}
where $\Sigma_g$ is a Riemann surface of genus $g\geq2$, and $Q_0\cap Q_1$ is the smooth intersection of two smooth quadric hypersurfaces in $\bCP^{2g+1}$.
This has an interpretation in terms of HMS, as explained in Smith's \S1.7.
Quilt theory is an essential ingredient in Smith's work.
In particular, he constructs an Ma'u--Wehrheim--Woodward-type $A_\infty$-functor
\begin{align}
\Fuk(Q_0\cap Q_1) \to \Fuk^\#\bigl(\text{Bl}_{Q_0\cap Q_1}\bigl(\bCP^{2g+1}\bigr)\bigr),
\end{align}
and argues in his Lemma 4.31 that this functor is quasi-isomorphic to one with image in the non-extended Fukaya category $\Fuk\bigl(\text{Bl}_{Q_0\cap Q_1}(\bCP^{2g+1})\bigr)$.
We anticipate that our work will enable similar arguments to be made in a more streamlined way: the functor associated to a correspondence will automatically have codomain a non-extended Fukaya category.

\subsubsection{Monoidal structure on Fukaya categories}
One of the many predictions for Fukaya categories arising from the SYZ Mirror Symmetry conjecture is the existence of a symmetric monoidal structure that is mirror to the natural tensor product on categories of coherent sheaves.
Again, the relation between algebraic and geometric properties in \S\ref{sss:funfuk} provides the crucial guide to identifying such structures: 
A monoidal functor should arise from a Lagrangian $\Gamma \subset (M\times M)^- \times M$.
Its symmetry would be guaranteed by invariance under the exchange of the first two factors $(x,y,o)\mapsto (y,x,o)$.
Associativity would be guaranteed by an identification between the two natural orders of composing $\Gamma$ with itself to obtain a Lagrangian in $(M\times M\times M)^- \times M$,
$$
\{(x,y,z,u) \,|\, \exists o : (x,y,o)\in \Gamma, (o,z,u)\in\Gamma  \} \; \sim \; \{(x,y,z,u) \,|\, \exists o : (x,o,u)\in \Gamma, (y,z,o)\in\Gamma  \}.
$$

A first example of such a symmetric monoidal structure was constructed by Subotic \cite{subotic} on the homology of the Fukaya category of a 2-torus $T^2$ equipped with a fibration over $S^1$ and distinguished section -- by constructing a Lagrangian $\Gamma$ that encodes fiberwise addition in the torus fibers.
The fact that Subotic worked in the SYZ setting is crucial here: SYZ mirror symmetry considers symplectic manifolds presented as Lagrangian torus fibrations with a distinguished Lagrangian section, which is necessary in order to make sense of ``fiberwise addition''.
Subotic's work therefore gave a roadmap toward equipping the homology-level Fukaya category with a symmetric conoidal structure in the SYZ setting.
Towards this, \cite{abouzaid_bottman_niu} constructed an analogue of $\Gamma$ in the context of an elliptically-fibered K3 surface, which is the next case to consider following Subotic's setting.
The associativity and symmetry of Subotic's functor followed immediately from the analogous properties satisfied by $\Gamma$, which are also satisfied by the correspondence constructed by Abouzaid--Bottman--Niu.

Pascaleff \cite{pascaleff_poisson} has proposed a generalization of this construction, and \cite[\S5.4]{abouzaid_bottman} discusses an enhancement of these result to the chain level, which would rely on a new algebraic notion of ``symmetric monoidal $A_\infty$-category''.
We expect that our notions of $(A_\infty,2)$-categories will inform this new algebraic notion, for the same reason that ordinary monoidal categories are related to 2-categories with a single object.
We eventually expect to need the notion of a symmetric monoidal $(A_\infty,2)$-category, in the context of Fukaya categories in the SYZ setting; such an algebraic notion will be related to the notion of an $(A_\infty,3)$-category.

\subsubsection{Bounding cochains, computations of Floer homology, and Seidel's formal group} \label{sss:MC}

The result of the analytic bubbling phenomena for pseudoholomorphic curves is that the $A_\infty$-structures are typically curved, that is the Floer differentials do not generally square to zero.
Thus (quilted) Floer homology generally depends on the choice of bounding cochains $b_L$ for each Lagrangian $L$ -- solutions of a Maurer-Cartan equation which guarantees that the $b_L$-twisted Floer differential squares to zero.
Thus generalizing the strip-shrinking isomorphism of (quilted) Floer homologies 
${\rm HF}(\ldots, L_{01},L_{12},\ldots)\simeq {\rm HF}(\ldots, L_{01}\circ L_{12},\ldots)$ 
under embedded geometric composition \cite{wehrheim_woodward_geometric_composition} to non-monotone settings in particular requires specifying bounding cochains.
While $\Symp$ does not contain quilted Floer complexes, any $A_\infty$-bifunctor $C^2$ in \eqref{eq:bifunctor} will map pairs of bounding cochains $b_{01}$ and $b_{12}$ for Lagrangian relations $L_{01}$ and $L_{12}$ to a bounding cochain for  $L_{12}\circ L_{23}$ given by
\begin{align}
b_{02} := \sum_{k,\ell\geq0} C^2\Bigl(\underbrace{b_{01},\ldots,b_{01}}_\ell \:\big|\: \underbrace{b_{12},\ldots,b_{12}}_k\Bigr).
\end{align}
With that the strip-shrinking isomorphism can be generalized to the \cite[\S4.4]{bottman_wehrheim} conjecture
${\rm HF}(\ldots, (L_{01},b_{01}),(L_{12},b_{12}),\ldots)\simeq {\rm HF}(\ldots, (L_{01}\circ L_{12},b_{02}),\ldots)$ which would allow for much broader computations of Floer homologies than the sample computations \cite{wehrheim_woodward_quilted} in monotone cases.

When we specialize the effect of \eqref{eq:bifunctor} on bounding cochains to $M_1 = M_2 = M_3 \eqqcolon M$ and the diagonal relation $\Delta_M$, then \eqref{eq:bifunctor} induces a product on the set of bounding cochains for $\Delta_M$.
Moreover, Seidel  \cite{seidel_formal} observed that this product should be compatible with an algebraic gauge action on solutions of the Maurer-Cartan equation, and thus induce a group structure on $MC(M)$,  the set of gauge equivalence classes of Maurer-Cartan solutions.
While this new invariant for general symplectic manifolds depends on the construction of \eqref{eq:bifunctor}, Seidel noticed that the product counts \emph{unquilted} pseudoholomorphic spheres with input marked points arranged along two circles, and thus the construction of $MC(M)$ for monotone $M$ was possible with classical methods.

\subsubsection{Categorification of Dehn twists and (symplectic) Khovanov homology}

Arnold's generalization of Dehn twists provide a key example of symplectomorphisms -- which Seidel proved to be non-isotopic to the identity by an exact triangle in Floer homology \cite{seidel_les} shown on the left for any Lagrangian sphere $C\subset M$ and admissible Lagrangians $L,L'\subset M$.
\begin{align*}
\xymatrix{
{\rm HF}(L,L') \ar[rr] &  &  {\rm HF}(L,\tau_C(L')) \ar[ld] \\
& {\rm HF}(L,C) \otimes {\rm HF}(C,L') \ar[lu] & 
}
\qquad
\xymatrix{
\Delta_M \ar[rr] &  &  {\rm gr}(\tau_C) \ar[ld] \\
&  C^T \# C \ar[lu] & 
}
\end{align*}
The technical infrastructure of the monotone symplectic 2-category allowed \cite{wehrheim_woodward_exact_triangle, perutz2008symplecticgysinsequence} to generalize this exact triangle to fibered Dehn twists along spherically fibered coisotropics $C\subset M$; and the conceptual framework allowed for a formulation as an exact triangle in the 2-category -- independent of the ``test branes'' $L,L'$ -- in the extended Fukaya category of $M^-\times M$.
In the case of a Lagrangian sphere $C\subset M$, this implies Seidel's exact triangle and a generalization to an exact triangle between the bimodules $CF^*(-,-)$, $CF^*(-,\tau_C-)$, and $CF^*(-,C)\otimes CF^*(-,C)$.

This exact triangle is central for Abouzaid--Smith's study \cite{abouzaid2019khovanov} of the relationship between the combinatorially defined Khovanov homology and symplectic Khovanov homology -- a knot invariant which \cite{SeidelSmith2006} constructed from the Floer homology of Lagrangians associated to braid representations of knots.
For more details, see \cite{abouzaid_bottman}[\S5.2].

\subsection*{Acknowledgments}

The second author's visits to the Max Planck Institute for Mathematics in Spring 2024 and summer 2025 played a crucial role in the development of this paper.
The authors thank MPIM for enabling this work.
The first author first conceived of associating operations to chains in fiber products of 2-associahedra in Spring 2019, and thanks Paul Seidel and Guillem Cazassus for their interest in this approach.
We are looking forward to the work of Andrew Manion and Chris Portwood building on the notions developed in this paper, which in particular led to several indexing bugs getting fixed in Definition~\ref{def:A_infty-2-category}.

\vfill

\pagebreak

\section{Preliminaries}
\label{sec:prelim}

In symplectic applications, the types of flow categories that we will introduce in section~\ref{sec:flow_cats} result from regularizing compactified moduli spaces of pseudoholomorphic curves and quilts -- which involves the choice of a technical framework to describe the moduli spaces and evaluation maps to base spaces in such a way that they allow for fiber products and push-pull constructions that induce the algebraic structure maps of the linear categories that we construct in section~\ref{sec:extracting}.
To facilitate non-partisan use of the new categorical concepts, we develop them without fixing the choice of a regularization approach.
Instead, we specify the key properties of a regularization framework:

\begin{definition} \label{def:framework}
A \emph{regularization framework} $(*)$ -- or short \emph{framework} -- consists of
\begin{enumerate}
\item[(i)]
a class of $(*)$moduli spaces that comes with a notion of boundary and a notion of a collection of $(*)$embeddings forming a ``system of boundary faces'', 
\item[(ii)]
a notion of $(*)$maps to $(*)$base spaces so that the fiber products in \eqref{eq:A_infty-flow_recursive} resp.\ \eqref{eq:type1_composition} -- \eqref{eq:type3_composition} yield $(*)$moduli spaces, and the notion includes the maps 
$p_{L^0 \ldots L^r}$ in \eqref{eq:a-b-maps} and $p_\cL$ in \eqref{eq:2-a-b-maps},
\item[(iii)]
$(*)$chain complexes over a coefficient ring (or field) $\Lambda$ on the $(*)$base spaces and $(*)$fundamental cycles on the $(*)$moduli spaces that allow for $\Lambda$-linear push-pull constructions under $(*)$maps as in \eqref{eq:push-pull} and \eqref{eq:2-push-pull}.
\end{enumerate}
\end{definition}

This definition has two purposes: First, it serves as a guide for the choice of a regularization framework for specific symplectic applications.
Second, it can be inserted into the definitions of section~\ref{sec:flow_cats} to create consistent definitions of flow categorical structures in any framework.
To maximize accessibility of the exposition of the new categorical concepts, we will mostly work in a simplified regularization framework whose existence in monotone settings and relation to the general case is discussed in Remark~\ref{rmk:cascades}.

\begin{definition} \label{def:C0-framework}
The framework $(*)=(\cC^0,\text{Morse})$ consists of the following notions: 
\begin{enumerate}
\item[(i)]
$(*)$moduli spaces are $\cC^0$-manifolds $X$ with boundary as in Definition~\ref{def:mfds_with_faces}, equipped with a locally constant energy function $\cE:X\to\bR$ so that $\cE^{-1}((-\infty,E])$ is compact for all $E\in\bR$; 
\item[(ii)]
$(*)$maps to $(*)$base spaces are either $\cC^0$-maps to finite sets (of critical points of Morse functions), or $\cC^0$-maps to the (2-)associahedra, viewed as compact $\cC^0$-manifolds with boundary;  
\item[(iii)]
$(*)$chain complexes are Morse chain complexes over the universal Novikov field $\Lambda$ generated by finite sets of critical points, and the relevant fundamental cycles over $\Lambda$ arise from well-defined counts of the $0$-dimensional parts of $\cC^0$-manifolds with fixed energy.
\end{enumerate}
\end{definition}

When working in the $(\cC^0,\text{Morse})$-framework, we will use the universal Novikov field 
\begin{align} \label{eq:novikov}
\Lambda \ \coloneqq \: \left\{\left. {\textstyle  \sum_{l=0}^\infty   a_l  T^{\lambda_l} }  \right| a_l\in\bZ_2, \lambda_l\in\bR , \lim_{l\to\infty} \lambda_l = +\infty   \right\}; 
\end{align}
see e.g.\ \cite[Def.1.3]{auroux_beginners_guide}.
Note in particular that it contains both infinite and finite series, as the latter can be expressed with $a_l=0$ for $l\geq l_0$.
Moreover, we use the following $\cC^0$-differential-geometric conventions.

\begin{definition}
\label{def:mfds_with_faces}
A \emph{$\cC^0$-manifold with boundary of dimension $n\in\bN$}
is a second-countable Hausdorff topological space $X$ 
equipped with an atlas of local homeomorphisms to open subsets of the half-space $[0,\infty)\times\bR^{n-1}$, such that all transition functions are continuous.
Its \emph{boundary} $\partial X$ is the union of preimages of the boundary $\{0\}\times\bR^{n-1}$.

A \emph{$\cC^0$-manifold with boundary} is a finite disjoint union $X = \bigsqcup_{n=0}^N X_n$ of $\cC^0$-manifolds with boundary $X_n$ of dimension $\dim X_n = n$.
Its \emph{$n$-dimensional part} is the component $X_n$ of dimension $n$.
Its \emph{boundary} $\partial X \coloneqq  \bigsqcup_{n=0}^N \partial X_n$ is the union of boundaries.
Its \emph{interior} $X^\circ \coloneqq  X \setminus \partial X$ is the complement of the boundary.

Given a $\cC^0$-manifold $X$ with boundary $\partial X$, a \emph{system of (codim 1)  boundary faces for $X$} is a finite collection of $\cC^0$-embeddings $\phi_i : F_i \to X$ of $\cC^0$-manifolds with boundary $F_i$ called \emph{faces} as follows:
\begin{enumerate}
\item[(i)]
The images of the faces lie in the boundary $\phi_i(F_i)\subset\partial X$, and the interiors of the faces are open subsets of the boundary $\phi_i(F_i^\circ)\subset \partial X$.
(Equivalently, the $n$-dimensional part of $F_i$ is embedded in the boundary of an $n+1$-dimensional part of $X$.) 
\item[(ii)]
The interiors of the faces are pairwise disjoint $\phi_i(F_i^\circ) \cap \phi_j(F_j^\circ)=\emptyset$ for $i\ne j$.
\item[(iii)]
The union of the interiors of the faces is dense in the boundary $\overline{\bigcup_i \phi_i(F_i^\circ)}=\partial X$.
\null\hfill$\triangle$
\end{enumerate}
\end{definition}

\begin{remark} \label{rmk:C0 manifolds}
We work with $\cC^0$-manifolds rather than smooth manifolds for two reasons: 
Firstly -- even for trivial isotropy in symplectic applications -- we cannot expect a smooth structure on the regularized moduli spaces of pseudoholomorphic quilts.
They will at best be smooth manifolds with generalized corners since already the underlying domain moduli spaces have generalized corners as e.g.\ in \cite[Figure 33]{abouzaid_bottman}.

Secondly, they suffice for our present purposes.
In particular note that any $1$-dimensional compact $\cC^0$-manifold with boundary $X$ is homeomorphic to a disjoint union of circles and closed intervals.
Now consider a system of boundary faces for $X$.
Its boundary $\partial X$ is a compact $\cC^0$-manifold of dimension $0$, i.e.\ finite unions of points.
Thus a system of boundary faces amounts to a partition $\partial X = \bigcup_i \phi_i(F_i^\circ)= \bigcup_i \phi_i(F_i)$ of its boundary into $0$-dimensional faces given by embeddings $\phi_i: F_i\to \partial X$, where each $F_i=F_i^\circ$ is a finite union of points.
These have well-defined counts $ \#_{\bZ_2} F_i = \#_{\bZ_2} \phi_i(F_i) \in \bZ_2$ modulo $2$ (i.e.\ odd or even), and since the total number of boundary points of any finite disjoint union of circles and intervals is even, we obtain the identity
$\textstyle \sum_i \#_{\bZ_2} F_i = 0 \in \bZ_2$.
\null\hfill$\triangle$
\end{remark}

The proofs of section \ref{sec:extracting} need a generalization of the identity resulting from systems of boundary faces for compact $\cC^0$-manifolds to noncompact moduli spaces in the regularization framework of Definition~\ref{def:C0-framework}.
This generalization will be formulated in terms of Novikov counts as follows.

\begin{definition}\label{def:Novikov count}
Let $Y$ be a $0$-dimensional $\cC^0$-manifold equipped with an energy function $\cE: Y\to\bR$ such that $\cE( Y ) = \{E_0, E_1, \ldots \}$ is a discrete set with $E_l < E_{l+1}$ and finitely many elements or $\lim_{l\to\infty} E_l = \infty$.
Assume in addition that each $\cE^{-1}(E_l)$ consists of finitely many points.
Then we define the Novikov count as
\begin{equation} \label{eq:Novikov count}
\#_\Lambda Y \:\coloneqq\: \textstyle \sum_{l=0}^\infty  \#_{\bZ_2}\cE^{-1}(E_l) \;  T^{E_l} .
\end{equation} 
\end{definition}

\begin{remark} \label{rmk:energies can be ordered}
The energy condition in Definition~\ref{def:Novikov count} follows if $Y$ is equipped with a locally constant energy function $\cE:Y\to\bR$ so that $\cE^{-1}((-\infty,E])$ is compact for all $E\in\bR$.
Indeed, compactness of $\cE^{-1}((-\infty,E])$ means that it can have only finitely many connected components, and thus $\cE(Y)\cap(-\infty,E]$ is finite for each $E\in\bR$ and thus can be ordered $\cE\bigl(\cE^{-1}((-\infty,E])\bigr) = \{ E_0 < E_1 < \ldots < E_L \}$.
As we increase $E\to\infty$, the ordered list of energy values either ends with finitely many entries or continues with $\lim_{l\to\infty} E_l = \infty$.
Here each $\cE^{-1}(E_l)$ is compact since it is a closed subset of the compact set $\cE^{-1}((-\infty,E])$.
If, moreover, $Y$ is of dimension $0$, then each $\cE^{-1}(E_l)$ is a compact manifold of dimension $0$, that is a finite number of points.
\end{remark}

\begin{lemma}\label{lem:Novikov product}
Let $Y_n$ for $1\leq n \leq N$ be $0$-dimensional $\cC^0$-manifolds equipped with energy functions $\cE_n: Y_n\to\bR$ as in Definition~\ref{def:Novikov count}.
Define an energy function $Y:=Y_1\times \ldots\times Y_N \to \bR$ by addition $(y_1,\ldots,y_N)\mapsto \cE_1(y_1)+\ldots+\cE_N(y_N)$.
Then the Novikov count is well-defined and multiplicative, 
\begin{equation}\label{eq:Novikov product}
\#_\Lambda \bigl( Y_1\times \ldots\times Y_N \bigr) \: = \:  \#_\Lambda Y_1 \: \cdot \: \ldots  \: \cdot \:   \#_\Lambda Y_N.
\end{equation}
\end{lemma}
\begin{proof}
Since the discrete energy values $\cE(Y_n)=\{ E_{n,0}, E_{n,1}, \ldots \}$ are discrete and bounded below for each $1\leq n \leq N$, their sums $\cE(Y)=\{ E_0:=E_{1,0}+\ldots+E_{n,0} , \ldots \}$ are discrete and bounded below as well.
Moreover, for each $E\in\bR$ we have the inclusion $\cE((-\infty, E]) \subset \cE_1^{-1}( (-\infty, E - E_0 + E_{1,0}] )\times \ldots \times \cE_N^{-1}( (-\infty, E - E_0 + E_{N,0}] )$ since $\cE_n(y_n)> E - E_0 + E_{n,0}$ would result in $\cE(y_1,\ldots,y_N) > \ldots+ E_{n-1,0}+E - E_0 + E_{n,0}+ E_{n+1,0}+\ldots = E$.
The right hand side of this inclusion is finite, which implies finiteness of all level sets of $\cE:Y\to\bR$.

To approximate $\#_\Lambda \cE((-\infty, E])$ consider the product 
\begin{align}
& \#_\Lambda  \cE_1^{-1}( (-\infty, E - E_0 + E_{1,0} =: \hat E_1] )\cdot \ldots \cdot \#_\Lambda \cE_N^{-1}( (-\infty, E - E_0 + E_{N,0} =: \hat E_N] ) 
\nonumber \\
& \: = \: 
\left( \textstyle \sum_{E_{1,l} \leq \hat E_1}  \#_{\bZ_2}\cE_1^{-1}(E_{1,l}) \;  T^{E_{1,l}} \right)
\cdot \ldots \cdot
\left( \textstyle \sum_{E_{N,l} \leq \hat E_N}  \#_{\bZ_2}\cE_N^{-1}(E_{N,l}) \;  T^{E_{N,l}} \right)
\\
& \: = \: 
\textstyle \sum_{E_{1,l} \leq \hat E_1} 
\cdot \ldots \cdot
 \textstyle \sum_{E_{N,l} \leq \hat E_N}  
 \#_{\bZ_2} \bigl( \cE_1^{-1}(E_{1,l}) \times \ldots \times \cE_1^{-1}(E_{N,l}) \bigr) \;  T^{E_{1,l}+ \ldots + E_{N,l}}  
 \nonumber \\
 & \: = \: 
\textstyle \sum_{E_l \leq E} 
\left(
 \textstyle \sum_{E_{1,l}+\ldots +E_{N,l}=E_l}  
 \#_{\bZ_2} \bigl( \cE_1^{-1}(E_{1,l}) \times \ldots \times \cE_1^{-1}(E_{N,l}) \bigr) \right) \;  T^{E_l}  
 + \sum_{E_l > E}  \ldots 
 \nonumber \\
  & \: = \: 
\textstyle \sum_{E_l \leq E}  
 \#_{\bZ_2} \cE^{-1}(E_l) \;  T^{E_l}  
 + \sum_{E_l > E}  \ldots  \; .
 \nonumber 
\end{align}
It differs from $\#_\Lambda Y$ only in the terms with $E_l>E$, thus taking the limit $E\to\infty$ proves \eqref{eq:Novikov product}.
\end{proof}

\begin{lemma} \label{lem:system of boundary faces induces identity}
Let  $(\phi_i:F_i \to X)_{i\in I}$ be a system of boundary faces as in Definition~\ref{def:mfds_with_faces} for a $1$-dimensional $\cC^0$-manifold with boundary $X$.
Suppose moreover that $X$ is equipped with a locally constant energy function $\cE:X\to\bR$ so that $\cE^{-1}((-\infty,E])$ is compact for all $E\in\bR$.
Then we have 
\begin{equation}\label{eq:boundary Novikov identity}
\textstyle
 \sum_{i\in I} \#_\Lambda F_i  \: = \:  0  \quad \in \Lambda  ,  
\end{equation}
where each Novikov count is defined as in \eqref{eq:Novikov count} using the induced energy functions $\cE\circ\phi_i \,:\: F_i\to\bR$.
\end{lemma}

\begin{proof}
As in Remark~\ref{rmk:C0 manifolds}, a system of boundary faces for a $1$-dimensional $\cC^0$-manifold is given by a partition of its $0$-dimensional boundary $\partial X = \bigcup_{i\in I} \phi_i(F_i)$ into finitely many images of embeddings $\phi_i:F_i\to \partial X$ of $0$-dimensional $\cC^0$-manifolds $F_i$.
As in Remark~\ref{rmk:energies can be ordered} the energy values on the $1$-dimensional manifold form a discrete set $\cE(X) = \{E_0, E_1, \ldots \}$  which can be ordered $E_l < E_{l+1}$, and has finitely many elements or $\lim_{l\to\infty} E_l; = \infty$.
Moreover, each $\cE^{-1}(E_l)$ is compact.
Thus each $X_l \coloneqq \cE^{-1}(E_l)$ is a compact $1$-dimensional $\cC^0$-manifold with boundary.
It inherits a system of boundary faces $\partial X_l = \bigcup_{i\in I} \phi_i(F_{i,l})$ given by $F_{i,l} \coloneqq \{ f\in F_i \,|\, \cE(\phi_i(f))=E_l \}$ -- some of which may be empty.
As in Remark~\ref{rmk:C0 manifolds}, $X_l$ is homeomorphic to a finite disjoint union of circles and closed intervals, thus has an even number of boundary points, which implies the identity
\begin{equation}
\textstyle
 \sum_{i\in I} \#_{\bZ_2} F_{i,l}  \: = \:  \#_{\bZ_2} \bigcup_{i\in I}\phi_i(F_{i,l})  \: = \:  \#_{\bZ_2} \partial X_l  \: = \:  0  \quad \in \bZ_2  .
\end{equation}
Now summing over any finite subset of energies $E_1<E_2<\ldots<E_N$ implies  
\begin{equation}
\textstyle
 \Lambda \ni \: 0 
\: = \:  
\sum_{l=1}^N \bigl(  \sum_{i\in I}   \#_{\bZ_2} F_{i,l}  \bigr) \; T^{E_l}
 \: = \:
 \sum_{i\in I}  \bigl(   \sum_{l=1}^N \#_{\bZ_2} F_{i,l} \; T^{E_l}  \bigr),
\end{equation}
where we can change the order of summation since both sums are finite.
Now for each fixed $i\in I$ the expression $\sum_{l=1}^L \#_{\bZ_2} F_{i,l} \; T^{E_l}$ converges to $\#_\Lambda F_i$ for $L\to\infty$ since the energy values of $\cE\circ\phi_i:F_i\to\bR$ are a subset of those of $\cE:X\to\bR$.
That results in the overall Novikov identity
\begin{equation}
\textstyle
 \Lambda \ni \: 0 
 \: = \:
  \sum_{i\in I}  \bigl(   \sum_l \#_{\bZ_2} F_{i,l} \; T^{E_l}  \bigr)
 \: = \:
\sum_{i\in I} \#_\Lambda F_i  .
\end{equation}
\vspace{-5mm}
\end{proof}

As a final preliminary, we will use the following notation conventions for categories.

\begin{definition} \label{def:category}
A \emph{category} $\sC$ consists of:
\begin{itemize}
\item
A set $\Ob=\Ob_\sC$ of objects.

\smallskip

\item
A set $\Mor=\Mor_\sC$ of morphisms with source and target maps $\sigma,\tau:\Mor\to\Ob$.
This induces for every $M_0, M_1 \in \Ob$, a set $\Mor_\sC(M_0,M_1) \coloneqq \sigma^{-1}(M_0)\cap\tau^{-1}(M_1)$.

\smallskip

\item
An associative binary operation 
$\Mor \leftindex_\sigma\times_\tau \Mor_\sC  \:\to\: \Mor_\sC$, $(L_{01},L_{12}) \:\mapsto\: L_{01}\circ L_{12}$ 
with $\sigma(L_{01}\circ L_{12})=\sigma(L_{01})$ and $\tau(L_{01}\circ L_{12})=\tau(L_{12})$.
This induces a well-defined composition of tuples of any length $k\geq 1$,  
$$ 
L_{01}\circ L_{12} \circ \ldots \circ L_{(k-1)k} \:\in\: \Mor_\sC(M_0,M_k)
$$
for any $L_{01}\in \Mor_\sC(M_0,M_1), L_{12}\in \Mor_\sC(M_1,M_2), \ldots,  L_{(k-1)k}\in \Mor_\sC(M_{k-1},M_k)$.

\item
Identity morphisms $\id_M\in\Mor_\sC(M,M)$ for all $M\in\Ob$ so that $\id_{M_0}\circ L_{01}=L_{01}$ and $L_{01}\circ \id_{M_1}=L_{01}$ holds for any $L_{01}\in \Mor_\sC(M_0,M_1)$.
\end{itemize}
\end{definition}

\section{$A_\infty$- and $(A_\infty,2)$-flow categories}
\label{sec:flow_cats}

In this section, we define the notions of $A_\infty$- and $(A_\infty,2)$-flow categories.
We formalize the notion of an $A_\infty$-flow category in a general class of spaces and maps as described in  section~\ref{sec:prelim}.
Flow categories were first introduce by Cohen-Jones-Segal \cite{cohen-jones} with an eye towards Floer theory, but no specific applications.
They were formalized in \cite[Def.2.9]{zhou-morse-bott} for application to (equivariant) Morse-Bott theory, and in \cite[Def.3.4]{abouzaid-blumberg} towards Floer homotopy theory.
The following notion of an $A_\infty$-category is a new formalization of existing constructions of Fukaya categories \cite{fooo_12,seidel_picard-lefschetz}.
The subsequent notion of an $(A_\infty,2)$-flow category provides the blueprint for generalizing the monotone symplectic 2-category developed by Wehrheim-Woodward \cite{wehrheim_woodward_2010functoriality} to a construction on chain level that includes all compact symplectic manifolds and Lagrangian relations.

\begin{definition} \label{def:general_A_infty-flow_category}
A \emph{regularized $A_\infty$-flow category $\sC$ in a framework $(*)$} as in Definition~\ref{def:framework} consists of:
\begin{itemize}
\item
A set $\Ob=\Ob_\sC$ of objects.

\smallskip

\item
For every $L, L' \in \Ob$, a $(*)$base space $\Mor(L,L')=\Mor_\sC(L,L')$.

\smallskip

\item
For every $n \geq 0$ and $L^0, \ldots, L^n \in \Ob$, a $(*)$moduli space $\cX(L^0,\ldots,L^n)$ that is equipped with $(*)$maps 
-- called {\it evaluation maps} $\alpha^\cdot_{\cdots}, \beta_{\cdots}$ resp.\ {\it forgetful map} $p_{\cdots}$ -- 
\begin{align} 
\label{eq:a-b-maps}
\xymatrix{
&&& \Mor(L^{k-1},L^k)
\quad\text{for}\quad k = 1, \ldots, n \\
\cX(L^0,\ldots,L^n) \ar@/^1.5pc/[urrr]^(0.4){\alpha^k_{L^0 \ldots L^n}}\ar[rrr]^{\beta_{L^0 \ldots L^n}}\ar@/_1.0pc/[drrr]_{p_{L^0 \ldots L^n}} &&& \Mor(L^0,L^n) \\
&&& K_n
}
\end{align}
where $K_n$ is the $(n-2)$-dimensional associahedron.

\smallskip

\item
For every choice of integers $n,s,t\geq 0$ with $s+t\leq n$ 
and objects $L^0,\ldots,L^n \in \Ob$, a $(*)$embedding -- called a \emph{boundary decomposition map} --
\begin{align}
\label{eq:A_infty-flow_recursive}
\varphi^{s,t}_{L^0 \ldots L^n}
\colon
\cX(L^0,\ldots,L^s,L^{s+t},\ldots,L^n)
\:\leftindex_{\alpha^{s+1}}\times_\beta\:
\cX(L^s,\ldots,L^{s+t})
\to
\cX(L^0,\ldots,L^n).
\end{align}
\end{itemize}

\noindent
We require that these boundary composition maps satisfy the following properties.

\begin{enumerate}
\item[\bf(i)]
Each boundary decomposition map $\varphi_{L^0 \ldots L^n}^{s,t}$ covers the corresponding operadic composition map $\sigma_{s+1}$ on associahedra, in the sense that the following square commutes:
\begin{align}\label{eq:cover operadic composition}
\xymatrix{
\cX(L^0,\ldots,L^s,L^{s+t},\ldots,L^n)
\:\leftindex_{\alpha^{s+1}}\times_\beta\:
\cX(L^s,\ldots,L^{s+t})
\ar[rr]^(0.7){\varphi^{s,t}_{L^0 \ldots L^n}} \ar[d]_{p_{L^0 \ldots L^s,L^{s+t} \ldots L^n }\times p_{L^s\ldots L^{s+t}}} &&
\cX(L^0,\ldots,L^n) \ar[d]^{p_{L^0 \ldots L^n}} \\
K_{n-t+1} \times K_t \ar[rr]^{\sigma_{s+1}} && K_n.
}
\end{align}
\item[\bf(ii)]
Each boundary decomposition map $\varphi_{L^0 \ldots L^n}^{s,t}$ is compatible with the evaluation maps as follows.
Denoting $\cL \coloneqq (L^0,\ldots,L^n)$, $\cL' \coloneqq  (L^0 \ldots L^s,L^{s+t} \ldots L^n)$, and $\cL'' \coloneqq  (L^s ,\ldots, L^{s+t})$ we require for all $(\chi',\chi'')\in \cX(\cL')\leftindex_{\alpha^{s+1}}\times_\beta \cX(\cL'')$ and $\chi \coloneqq \varphi_{L^0 \ldots L^n}^{s,t}(\chi',\chi'')$ 
\begin{equation}
\label{eq:compatibility-with-alpha-beta}
\beta_{\cL}(\chi)= \beta_{\cL'}(\chi')
\qquad\text{and}\qquad
\alpha^k_{\cL}(\chi) \;=\; 
\begin{cases}
 \alpha^k_{\cL'}(\chi') & \quad\text{for} \;  k \leq s    \\
\alpha^{k-s}_{\cL''}(\chi'') & \quad \text{for} \; s+1 \leq k \leq s+t \\
 \alpha^{k-t+1}_{\cL'}(\chi') & \quad\text{for} \;  k \geq s+t + 1 .
\end{cases}
\end{equation}

\item[\bf(iii)]
For any $n \geq 0$ and $L^0,\ldots,L^n \in \Ob$ the collection of boundary decomposition maps $\varphi_{L^0 \ldots L^n}^{s,t}$ 
for $s,t\geq 0$ with $s+t\leq n$ form a system of boundary faces for $\cX(L^0,\ldots,L^n)$.

\smallskip

\item[\bf(iv)]
The boundary decomposition maps are associative in the sense that the
following two diagrams commute for any choice of $n\geq 0$, objects $L^0,\ldots,L^n \in \Ob$, 
and the two combinatorial possibilities for obtaining this string of $n+1$ objects by inserting substrings of length 
$t\geq 0$ and $t'\geq 0$ at the positions $s$ and $s+s'$ into a string of $n-t-t'+1$ objects.
\begin{itemize}
\item[\bf(a)]
The first case -- an insertion within the insertion -- is described by integers $s,s',t,t'\geq 0$ with $s'\leq t$ and $s+t+t' \leq n$ and requires the commuting diagram
\begin{align}
\label{eq:associativity-case-1}
\xymatrix{
\cX(\cL_1)
\leftindex_{\alpha^{s+1}}\times_\beta
\cX(\cL_2)
\leftindex_{\alpha^{s'+1}}\times_\beta
\cX(\cL_3)
\ar[rrr]^{\hspace{2em}\id\times\varphi^{s',t'}_{\cL_{23}}}
\ar[d]_{\varphi^{s,t}_{\cL_{12}}\times\id}
&&&
\cX(\cL_{1})
\leftindex_{\alpha^{s+1}}\times_\beta
\cX(\cL_{23})
\ar[d]^{\varphi^{s,t+t'}_{L^0 \ldots L^n}}
\\
\cX(\cL_{12})
\leftindex_{\alpha^{s+s'+1}}\times_\beta
\cX(\cL_{3})
\ar[rrr]_{\varphi^{s+s',t'}_{L^0 \ldots L^n}}
&&&
\cX(L^0 \ldots L^n)
}
\end{align}
where we write
$\cL_1 \coloneqq (L^0 \ldots L^s,L^{s+t+t'} \ldots L^n)$, 
$\cL_2 \coloneqq (L^s \ldots L^{s+s'},L^{s+s'+t'} \ldots L^{s+t+t'})$,
$\cL_3 \coloneqq (L^{s+s'} \ldots L^{s+s'+t'})$.
Then 
$\cL_{23} \coloneqq (L^s \ldots L^{s+t+t'})$ is obtained by inserting $\cL_3$ into $\cL_2$
and
$\cL_{12} \coloneqq (L^0 \ldots L^{s+s'},L^{s+s'+t'} \ldots L^n)$ is obtained by inserting $\cL_2$ into $\cL_1$.
\item[\bf(b)]
The second case -- two independent insertions -- is described by integers $s,s',t,t'\geq 0$ with $s+t+s'+t' \leq n$ and requires the commuting diagram
\begin{align}
\label{eq:associativity-case-2}
\xymatrix{
\cX(\cL_1)
\leftindex_\beta\times_{\alpha^{s+1}}
\cX(\cL_2)
\leftindex_{\alpha^{s+s'+2}}\times_\beta
\cX(\cL_3)
\ar[rrr]^{\hspace{2em} \id \times \varphi^{s+s',t'}_{\cL_{23}}}
\ar[d]_{\varphi^{s,t}_{\cL_{12}} \times \id}
&&&
\cX(\cL_1)
\leftindex_\beta\times_{\alpha^{s+1}}
\cX(\cL_{23})
\ar[d]^{\varphi^{s,t}_{L^0 \ldots L^n}}
\\
\cX(\cL_{12})
\leftindex_{\alpha^{s+t+s'+1}}\times_\beta
\cX(\cL_{3})
\ar[rrr]_{\varphi^{s+t+s',t'}_{L^0 \ldots L^n}}
&&&
\cX(L^0 \ldots L^n)
}
\end{align}
with 
$\cL_1 \coloneqq (L^s \ldots L^{s+t})$, 
$\cL_2 \coloneqq (L^0 \ldots L^s,L^{s+t} \ldots L^{s+t+s'},L^{s+t+s'+t'} \ldots L^n)$,
$\cL_3 \coloneqq (L^{s+t+s'} \ldots L^{s+t+s'+t'})$.
Then 
$\cL_{12} \coloneqq (L^0 \ldots L^{s+t+s'},L^{s+t+s'+t'} \ldots L^n)$
results by inserting $\cL_1$ into $\cL_2$ 
and
$\cL_{23} \coloneqq (L^0 \ldots L^s,L^{s+t} \ldots L^n)$
results by inserting $\cL_3$ into $\cL_2$.$\triangle$
\end{itemize}
\end{enumerate}
\end{definition}

\begin{remark}
This notion of a regularized $A_\infty$-flow category is a combination of properties of the unregularized moduli spaces -- which allow for their coherent regularization -- and properties of the regularized moduli spaces -- which allow for the construction of a linear $A_\infty$-category in Proposition~\ref{prop:A-infinity_flow_to_linear}.
The latter result requires only properties (ii) and (iii) -- thus could be achieved without reference to the maps $p_{L^0 \ldots L^n}$ to the underlying associahedra $K_n$.
However, these maps and property (i) is what gives this categorical structure its name.
Property (iv) is a natural feature of the unregularized moduli spaces, resultingw from the operadic structure of associahedra (which concerns the domains of the maps in the moduli space) and Gromov compactness (which concerns the possible degenerations of the maps themselves).
It is crucial for achieving property (iii) for the regularized moduli spaces -- by an iterative process which will naturally preserve property (iv).
\null\hfill$\triangle$
\end{remark}

Proposition~\ref{prop:A-infinity_flow_to_linear} outlines how a regularized $A_\infty$-flow category in any regularization framework induces a linear $A_\infty$ category.
To give a precise yet accessible proof, we will work in the simplified 
$(\cC^0,\text{Morse})$-framework of Definition~\ref{def:C0-framework}, as specified in Definition~\ref{def:A_infty-flow_category} below.

\begin{remark} \label{rmk:cascades}
Given a regularized $A_\infty$-flow category $\sC$ in any framework $(*)$, the $(*)$maps to $(*)$base spaces can be replaced by maps to finite sets by an application of the homological perturbation lemma as explained in 
 \cite[2.4]{zhou-morse-bott} -- as long as the $(*)$chain complexes on the $(*)$base spaces are quasi-equivalent to Morse chain complexes via a projection-homotopy relation.
The projections and homotopies in this relation are typically given by half-finite and finite length Morse flow lines, and the homological perturbation lemma then informs the so-called cascades construction, in which these Morse flow lines are coupled with the $(*)$moduli spaces via the evaluation maps $\alpha,\beta$ in \eqref{eq:a-b-maps} to build so-called cascade moduli spaces.
(In the context of moduli spaces of pseudoholomorphic curves and quilts, this cascade construction is outlined in \cite[4.3]{bottman_wehrheim}.) In most regularization frameworks, these will also inherit the structure of a $(*)$moduli space.
Now the simplifying assumption that would yield a regularized $A_\infty$-flow category in the $(\cC^0,\text{Morse})$-framework is that the cascade moduli spaces can be given the structure of $\cC^0$-manifolds with boundary.

In symplectic applications, this is both a gross oversimplification and the appropriate setting to develop new categorical structures.
It can be achieved by geometric assumptions -- such as monotonicity, as shown in \cite{wehrheim_woodward_jhol_quilts} -- which rule out isotropy and guarantee that equivariant transversality of the Cauchy-Riemann operator can be achieved.
However, when dealing with general symplectic manifolds, the issues of isotropy and obstructions to equivariant transversality need to be resolved by a more refined regularization framework that replaces $\cC^0$-manifolds with e.g.\ weighted branched orbifolds or another type of $(*)$moduli spaces with (relative / virtual / ...) fundamental chains.
The overall structure formed by these moduli spaces will remain the same -- just requires working with chain complex with more refined coefficient rings, which would needlessly complicate the exposition of the new type of higher categorical structures that is the purpose of this paper.
\null\hfill$\triangle$
\end{remark}

\begin{definition} \label{def:A_infty-flow_category}
A \emph{regularized $A_\infty$-flow category $\sC$ in the $(\cC^0,\text{Morse})$-framework of Definition~\ref{def:C0-framework}} consists of:
\begin{itemize}
\item
A set $\Ob=\Ob_\sC$.

\smallskip

\item
For every $L,L' \in \Ob$, a finite set $\Mor(L,L')=\Mor_\sC(L,L')$.
\smallskip

\item
For every $n \geq 0$ and $L^0, \ldots, L^n \in \Ob$, a $\cC^0$-manifold with boundary $\cX(L^0 ,\ldots, L^n)$ that is equipped with a locally constant energy function $\cE:\cX(L^0 \ldots L^n)\to\bR$ such that $\cE^{-1}((-\infty,E])$ is compact for all $E\in\bR$, and with $\cC^0$ evaluation and forgetful maps 
\begin{align} 
\xymatrix{
&&& 
\Mor(L^{k-1},L^k)
\quad\text{for}\quad k = 1, \ldots, n \\
\cX(L^0,\ldots,L^n) \ar@/^1.5pc/[urrr]^(0.4){\alpha^k_{L^0 \ldots L^n}}
\ar[rrr]^{\beta_{L^0 \ldots L^n}}\ar@/_1.0pc/[drrr]_{p_{L^0 \ldots L^n}} &&& \Mor(L^0,L^n) \\
&&& K_n
}
\end{align}
where $K_n$ is the $(n-2)$-dimensional associahedron.

\smallskip

\item
For every choice of integers 
$n,s,t\geq 0$ with $s+t\leq n$ 
and objects $L^0,\ldots,L^n \in \Ob$, a $\cC^0$-embedding 
-- called a \emph{boundary decomposition map} -- 
as in \eqref{eq:A_infty-flow_recursive},
\begin{align}
\varphi^{s,t}_{L^0 \ldots L^n}
\colon
\cX(L^0 \ldots L^s,L^{s+t} \ldots L^n)
\:\leftindex_{\alpha^{s+1}}\times_\beta\:
\cX(L^s \ldots L^{s+t})
\to
\cX(L^0 \ldots L^n).
\end{align}
\end{itemize}
We require that these boundary composition maps satisfy the properties (i)--(iv) as in Definition~\ref{def:general_A_infty-flow_category}: 
\begin{enumerate}
\item[\bf(i)]
Each boundary decomposition map $\varphi_{L^0 \ldots L^n}^{s,t}$ covers the corresponding operadic composition map $\sigma_{s+1}$ on associahedra as in \eqref{eq:cover operadic composition}.

\item[\bf(ii)]
Each boundary decomposition map $\varphi_{L^0 \ldots L^n}^{s,t}$ is compatible with the evaluation maps $\alpha^k_{L^0\ldots L^n}$ and $\beta_{L^0\ldots L^n}$ as in \eqref{eq:compatibility-with-alpha-beta} and makes the energy functions additive in the sense that 
\begin{equation}\label{eq:energy additive}
\cE\bigl( \varphi_{L^0 \ldots L^n}^{s,t}(\chi',\chi'') \bigr) \: = \: \cE(\chi') + \cE(\chi'') .
\end{equation}

\item[\bf(iii)]
For any $L^0,\ldots,L^n \in \Ob$ the collection of boundary decomposition maps $\varphi_{L^0 \ldots L^n}^{s,t}$ for $s,t\geq 0$ with $s+t\leq n$ forms a system of boundary faces for $\cX(L^0,\ldots,L^n)$ in the sense of Def.~\ref{def:mfds_with_faces}.

\item[\bf(iv)]
The boundary decomposition maps are associative in the sense that the two types of diagrams \eqref{eq:associativity-case-1} and \eqref{eq:associativity-case-2} commute for any choice of $n\geq 0$, objects $L^0,\ldots,L^n \in \Ob$, 
and the two combinatorial possibilities for obtaining this string of $n+1$ objects by inserting two strings of $\geq 0$ objects into a given string of objects.
\null\hfill$\triangle$
\end{enumerate}
\end{definition}

The definition of an $(A_\infty,2)$-flow category is similar.
It is a higher-categorical structure in which morphisms $\Mor(M_0,M_1)$ between two fixed objects have the structure of an $A_\infty$-flow category, as we will see in Lemma~\ref{lem:mor is cat}.
Its higher categorical structure then results from 2-composition data for any finite tuple $M_0,\ldots,M_r$ of objects and multi-tuples of 1-morphisms between them, which covers the associative composition of 1-morphisms
\begin{equation}
\Mor(M_0,M_1)\times \Mor(M_1,M_2) \times \ldots \times \Mor(M_{r-1},M_r) \quad \longrightarrow\quad  \Mor(M_0,M_r).
\end{equation}

\begin{remark}
In symplectic applications, the $A_\infty$-structure on morphisms with fixed source and target is a flow category version of the Fukaya category of the product symplectic manifold $M_0^-\times M_1$.
Beyond that, the unregularized symplectic $(A_\infty,2)$-flow category as envisioned in \cite[\S4]{bottman_wehrheim} and  \cite[\S4]{abouzaid_bottman} contains 2-composition data -- a moduli space $\cM(\cL)$ with various evaluation maps -- for any finite tuple $M_0,\ldots,M_r$ of objects and a multi-tuple $\cL$ of tuples of 1-morphisms in each of $\leftindex^1\Mor(M_0,M_1), \leftindex^1\Mor(M_1,M_2),\ldots,\leftindex^1\Mor(M_{r-1},M_r)$.
Describing the relative operadic structure of the expected boundary of these moduli spaces requires the consideration of fiber products 
\begin{equation}
\textstyle
\prod^{K_r}_{1\leq j \leq a} \cM(\cL^j) \:  \coloneqq \:  
\bigl\{ (m_1,\ldots, m_a) \in  \cM(\cL^1)\times\ldots\times  \cM(\cL^a) \, \big| \, \pi_{\cL^1}(m_1)=\ldots = \pi_{\cL^a}(m_a)  \bigr\} .
\end{equation}
Here $\cL^1,\ldots,\cL^a$ are $a\geq 1$ multi-tuples as above for the same objects $M_0,\ldots,M_r$, and $\prod^{K_r}_{1\leq j \leq a}$ indicates their iterated fiber product with respect to maps $\pi_{\cL^j}=\pi\circ p_{\cL^j}: \cL^j \to K_r$ for $1\leq j \leq a$ which factor through the forgetful maps $\pi:W_\bn \to K_r$ from 2-associahedra to associahedra.
As discussed in the introduction, these fiber products present a conceptual challenge already at the level of 2-associahedra and the resulting algebraic notion of $(A_\infty,2)$-category.

The resolution presented in this work 
is to expand the collection of 2-associahedra by including their iterated fiber products as in \eqref{eq:O}.
At the level of moduli spaces this means that we consider the fiber products $\prod^{K_r}_{1\leq j \leq a} \cM(\cL^j) =: \cM(\cL)$ as separate moduli spaces and regularize them without necessarily preserving the fiber product structure.
This leads to regularized composition data for tuples $\cL \coloneqq (\cL^j)_{1\leq j\leq a}$ of multi-tuples $\cL^j = (L_{(i-1)i}^{j,k})_{1\leq i \leq r, 1\leq k \leq n^j_i}$ in Definition~\ref{def:general A_infty-2-flow_category} below.

The fact that in certain circumstances -- in particular when $K_r$ is trivial so that the fiber product is actually a Cartesian product -- the moduli spaces could be regularized compatibly with the (fiber) product relations will then be formulated as an additional property of $(A_\infty,2)$-flow categories in Definition~\ref{def:compatible}.
\null\hfill$\triangle$
\end{remark}

\begin{definition}
\label{def:general A_infty-2-flow_category}
A \emph{regularized $(A_\infty,2)$-flow category $2\sC$ in a framework $(*)$} as in Definition~\ref{def:framework} consists of:

\begin{itemize}
\item
A category $(\Ob, \leftindex^1\Mor)=(\Ob_{2\sC}, \leftindex^1\Mor_{2\sC})$.

\smallskip

\item
For every $L, L' \in  \leftindex^1\Mor$, a $(*)$base space 
$\leftindex^2\Mor(L,L')=\leftindex^2\Mor_{2\sC}(L,L')$.

\smallskip

\item
For every choice of integers $r \geq 1, a \geq 1$, tuples of integers 
$\underline\bn=(\bn^1, \ldots, \bn^a) \in \bigl(\bZ_{\geq0}^r\bigr)^a$, 
collections of objects $M_0,\ldots,M_r\in\Ob$ and 
a \emph{collection of 1-morphisms of shape $\ul\bn$}
\begin{align}   \label{eq:curly-L}
\cL & = 
\left(\begin{array}{c}
 \cL^1 \coloneqq \left( L_{(i-1)i}^{1,k} \right)_{ 1\leq i\leq r , 0\leq k\leq n_i^1} 
\\ \vdots \\ 
\cL^a \coloneqq  \left(L_{(i-1)i}^{a,k} \right)_{ 1\leq i\leq r , 0\leq k\leq n_i^a}
\end{array}\right)
\: = \: \Bigl(L_{(i-1)i}^{j,k} \in  \leftindex^1\Mor(M_{i-1},M_i) \Bigr)_{(i,j,k)\in I^{\, r,a}_{\underline\bn} }  \\
& \qquad\qquad\qquad \text{indexed by}  \qquad
 I^{\, r,a}_{\underline\bn}  \coloneqq  \{ (i,j,k) \,|\,  1\leq i\leq r , 1\leq j\leq a, 0\leq k\leq n_i^j  \} 
\nonumber
\end{align}
a $(*)$moduli space $\cX(\cL)$ that is equipped with $(*)$maps
-- called {\it evaluation maps} $\alpha^{\cdots}_\cL, \beta^{\cdots}_\cL$ resp.\ {\it forgetful map} $p_\cL$ -- 
\begin{align}
\label{eq:2-a-b-maps}
\xymatrix{
&&& \leftindex^2\Mor(L_{(i-1)i}^{j,k-1},L_{(i-1)i}^{j,k}) 
\quad\text{for}\quad (i,j,k)\in  {'\!I^{\, r,a}_{\underline\bn}}  \\
\cX(\cL) \ar@/^1.0pc/[urrr]^(0.3){\alpha^{i,j,k}_\cL}\ar[rrr]^(0.3){\beta^j_\cL}\ar@/_1.0pc/[drrr]_(0.3){p_\cL} &&& \leftindex^2\Mor(L_{01}^{j,0}\circ\cdots\circ L_{(r-1)r}^{j,0},L_{01}^{j,n_1^j}\circ\cdots\circ L_{(r-1)r}^{j,n_r^j}) 
\quad\text{for}\quad 1\leq j\leq a 
\\
&&& 
\widetilde W_{\underline\bn} \coloneqq \prod^{K_r}_{1\leq j \leq a}  W_{\bn^j}
}
\end{align}
Here pairs of consecutive 1-morphisms are indexed by 
$$
'\!I^{\, r,a}_{\underline\bn}  \coloneqq  \{ (i,j,k) \,|\,  1\leq i\leq r , 1\leq j\leq a, 1\leq k\leq n_i^j  \}.
$$
$W_\bn$ is the $\bn$-th 2-associahedron as in \cite{bottman_realizations,bottman_oblomkov}, 
and 
$\prod^{K_r}_{1\leq j\leq a} W_{\bn^j}$ indicates the iterated fiber product with respect to the forgetful maps $W_{\bn^j}\to K_r$ on each of the $a$ factors.

\smallskip

\item
For every choice of integers $r \geq 1, a \geq 1$, $\underline\bn\in \bigl(\bZ_{\geq0}^r\bigr)^a$, and collection of 1-morphisms $\cL$ as in \eqref{eq:curly-L}, $(*)$-embeddings -- called \emph{boundary decomposition maps} -- of three types: 

\begin{enumerate}
\item
Insertion of $t\geq 0$ choices of 1-morphisms between objects $M_{i-1}$ and $M_i$ 
for $1\leq i\leq r$ in a block $1\leq j \leq a$ at position $s\geq 0$ for $s+t\leq n^j_i$ 
gives rise to a \emph{type-1 boundary decomposition map}
\begin{equation}
\label{eq:type1_composition}
\varphi^{(1)}_{i,j,s,t} 
\: \colon \quad
\cX(\cL'\coloneqq\cL^{(1) \rm in}_{i,j,s,t}) \:\leftindex_{\underline\alpha'} \times_{\underline\beta''}\: \cX(\cL''\coloneqq\cL^{(1) \rm out}_{i,j,s,t}) \quad \longrightarrow \quad \cX(\cL)
\end{equation}
where we denote $\underline\alpha' \coloneqq \alpha^{i,j,s+1}_{\cL'}$, 
$\underline\beta'' \coloneqq \beta^1_{\cL''}$, 
and
\begin{align}
\cL'\coloneqq\cL^{(1) \rm in}_{i,j,s,t} & \coloneqq 
\left(\begin{array}{c}
\vdots  \\ \cL^{j-1} \\  {\cL'}^j_{i,s,t} \\  \cL^{j+1} \\ \vdots 
\end{array}\right)
\quad\text{with}\quad
{\cL'}^j_{i,s,t} \coloneqq 
\left(\begin{array}{ccccc}
\cdots & L_{(i-2)(i-1)}^{j,0} & L_{(i-1)i}^{j,0} & L_{i(i+1)}^{j,0} & \cdots \\
 && \vdots && \\
& \vdots & L_{(i-1)i}^{j,s} & \vdots &   \\
 & \vdots & L_{(i-1)i}^{j,s+t} & \vdots & \\
 && \vdots && \\
\cdots & L_{(i-2)(i-1)}^{j,n^j_{i-1}} & L_{(i-1)i}^{j,n^j_i} & L_{i(i+1)}^{j,n^j_{i+1}} & \cdots
\end{array}\right), 
 \\
\cL''\coloneqq\cL^{(1) \rm out}_{i,j,s,t} & \coloneqq \left(\begin{array}{c}
L_{(i-1)i}^{j,s} \\
\vdots \\
L_{(i-1)i}^{j,s+t}
\end{array}\right) .
\nonumber
\end{align}
\item
Composition of $2\leq t \leq r-1$ consecutive 1-morphisms between the objects $M_s$ and $M_{s+t}$ for $s\geq 0$ with $s+t\leq r$ along with insertion of 1-morphisms between these objects indexed by partitions for each $1\leq j \leq a$ 
\begin{equation}
(n^j_{s+1},\ldots,n^j_{s+t}) \ = \ \bm^{j,1} + \ldots + \bm^{j,b^j}  
\qquad \text{into} \quad   \bm^{j,1},\ldots,\bm^{j,b^j}  \in \bZ_{\geq0}^t
\quad \text{for some}\   b^j\geq 1
\end{equation}
gives rise to a \emph{type-2 boundary decomposition map}
\begin{equation}
\label{eq:type2_composition}
\varphi^{(2)}_{s,t,\underline\bm}
\: \colon \quad
\cX(\cL'\coloneqq\cL^{(2) \rm in}_{s,t,\underline\bm}) \:\leftindex_{\underline{\alpha}'} \times_{\underline{\beta}''}\: \cX(\cL''\coloneqq\cL^{(2) \rm out}_{s,t,\underline\bm}) \quad \longrightarrow \quad \cX(\cL)
\end{equation}
where we denote $\underline{\alpha}'\coloneqq \bigl(  \alpha^{s+1,j,k}_{\cL'} \bigr)_{1\leq j \leq a, 1\leq k\leq b^j}$, $\underline{\beta}''\coloneqq \bigl(\beta_{\cL''}^j \bigr)_{1\leq j \leq b^1+\ldots b^a}$, and  
\begin{equation}
\cL'\coloneqq\cL^{(2) \rm in}_{s,t,\underline\bm}  \coloneqq \left(\begin{array}{c}
{\cL'}^1 \\  \vdots \\ {\cL'}^{a}
\end{array}\right), \qquad
\cL''\coloneqq\cL^{(2) \rm out}_{s,t,\underline\bm}  \coloneqq 
\left(\begin{array}{c}
{\cL''}^1 \\  \vdots \\  {\cL''}^{b^1+\ldots b^a} 
\end{array}\right)
\end{equation}
with
\begin{align}
& {\cL'}^j  \coloneqq
\left(\begin{array}{ccccc}
\cdots & L_{(s-1)s}^{j,0} & L_{s(s+1)}^{j,0}\circ\cdots\circ L_{(s+t-1)(s+t)}^{j,0} & L_{(s+t)(s+t+1)}^{j,0} & \cdots \\
 & \vdots & L_{s(s+1)}^{j,m^{j,1}_1} \circ\cdots\circ L_{(s+t-1)(s+t)}^{j,m^{j,1}_t}  & \vdots &  \\
  & \vdots & L_{s(s+1)}^{j,m^{j,1}_1+m^{j,2}_1} \circ\cdots\circ L_{(s+t-1)(s+t)}^{j,m^{j,1}_t+m^{j,2}_t}  & \vdots &  \\
 & \vdots & \vdots & \vdots &  \\
\cdots & L_{(s-1)s}^{j,n^j_s} & L_{s(s+1)}^{j,m^{j,1}_1 +\ldots m^{j,b^j}_1= n^j_{s+1}}\circ\cdots\circ L_{(s+t-1)(s+t)}^{j,m^{j,1}_t +\ldots m^{j,b^j}_t = n^j_{s+t}} & L_{(s+t)(s+t+1)}^{j,n^j_{s+t+1}} & \cdots
\end{array}\right), 
\nonumber \\ 
& {\cL''}^{b^1+\ldots b^{j-1}+j'}  \coloneqq 
\left(\begin{array}{ccc}
L_{s(s+1)}^{j,m_1^{j,1}+\cdots+m_1^{j,j'-1}} & \cdots & L_{(s+t-1)(s+t)}^{j,m_t^{j,1}+\cdots+m_t^{j,j'-1}} \\
\vdots &  & \vdots \\
L_{s(s+1)}^{j,m_1^{j,1}+\cdots+m_1^{j,j'}} & \cdots & L_{(s+t-1)(s+t)}^{j,m_t^{j,1}+\cdots+m_t^{j,j'}}
\end{array}\right) 
\qquad\text{for}\; 1\leq j' \leq b^j.
\nonumber
\end{align}

\item
Composition of all $r\geq 2$ 1-morphisms in the $j$-th block along with insertion of 1-morphisms indexed by a partition
$\bn^j = \bm^1+\cdots+\bm^b$ into $\bm^1,\ldots,\bm^b \in \bZ_{\geq0}^r$ for some $b\geq 1$
gives rise to a \emph{type-3 boundary decomposition map}
\begin{equation}
\label{eq:type3_composition}
\varphi^{(3)}_{j,\underline{\bm}} 
\: \colon \quad
\cX(\cL'\coloneqq\cL^{(3) \rm in}_{j,\underline\bm}) \:\leftindex_{\underline\alpha'} \times_{\underline\beta''}\: \cX(\cL''\coloneqq\cL^{(3) \rm out}_{j,\underline\bm}) \quad \longrightarrow \quad \cX(\cL)
\end{equation}
where we denote $\underline{\alpha}'\coloneqq \bigl(  \alpha^{1,1,k}_{\cL'} \bigr)_{1\leq k \leq b}$, $\underline{\beta}''\coloneqq \bigl(\beta_{\cL''}^{j+j''} \bigr)_{0\leq j'' \leq b-1}$, and  
\begin{equation}
\cL'\coloneqq\cL^{(3) \rm in}_{j,\underline\bm}  \coloneqq
\left(\begin{array}{ccccc}
L_{01}^{j,0}\circ\cdots\circ L_{(r-1)r}^{j,0}  \\
L_{01}^{j,m^{1}_1} \circ\cdots\circ L_{(r-1)r}^{j,m^{1}_r}   \\
L_{01}^{j,m^{1}_1+m^{2}_1} \circ\cdots\circ L_{(r-1)r}^{j,m^{1}_t+m^{2}_r}   \\
\vdots \\
 L_{01}^{j,m^{1}_1 +\ldots m^{b}_1= n^j_1}\circ\cdots\circ L_{(r-1)r}^{j,m^{1}_r +\ldots m^{b}_t = n^j_r} 
\end{array}\right)
\nonumber \\ 
, \qquad
\cL''\coloneqq\cL^{(3) \rm out}_{j,\underline\bm}  \coloneqq 
\left(\begin{array}{c}
{\cL''}^1 \coloneqq \cL^1 \\  \vdots \\ {\cL''}^{j-1}\coloneqq\cL^{j-1} \\
{\cL''}^j \\  \vdots \\  {\cL''}^{j+b-1} 
 \\ {\cL''}^{j+b}\coloneqq\cL^{j+1}  \\  \vdots \\  {\cL''}^{a+b-1}\coloneqq\cL^{a}
\end{array}\right)
\end{equation}
with
\begin{align}
& {\cL''}^{j+j''}  \coloneqq 
\left(\begin{array}{ccc}
L_{01}^{j,m_1^{1}+\cdots+m_1^{j''}} & \cdots & L_{(r-1)r}^{j,m_r^{1}+\cdots+m_r^{j''}} \\
\vdots &  & \vdots \\
L_{01}^{j,m_1^{1}+\cdots+m_1^{j''+1}} & \cdots & L_{(r-1)r}^{j,m_r^{1}+\cdots+m_r^{j''+1}}
\end{array}\right) 
\qquad\text{for}\; 0\leq j'' \leq b-1 .
\nonumber
\end{align}
\end{enumerate}
\end{itemize}

We require that these boundary decomposition maps satisfy the following properties.
\begin{enumerate}
\item[\bf(i)]
Each boundary decomposition map $\varphi^{(\tau)}_{\cdots}$ for types $\tau=1,2,3$ covers the corresponding operadic composition map $\wt\Gamma$ -- specified in Remark~\ref{rmk:combinatorics_of_regularized_A-infty_2_categories} -- 
on fiber products of 2-associahedra in the sense that the following square commutes: 
\begin{gather} \label{eq:2 cover}
\xymatrix{
\cX(\cL') \times \cX(\cL'')
\ar[rr]^(0.55){\varphi^{(\tau)}_{\cdots}} \ar[d]_{p_{\cL'}
\times p_{\cL''}} 
&&
\cX(\cL) \ar[d]^{p_\cL} \\
\widetilde W_{\underline\bn'} \times \widetilde W_{\underline\bn''} \ar[rr]_{\wt\Gamma} && \widetilde W_{\underline\bn}
}
\end{gather}

\smallskip

\item[\bf(ii)]
Each boundary decomposition map is compatible with the evaluation maps as follows.
Using the notation of \eqref{eq:type1_composition}, the type-1 map applied to 
$$
(\chi',\chi'')\in \cX(\cL') \leftindex_{\alpha^{i,j,s+1}_{\cL'}}\times_{\beta^1_{\cL''}}\cX(\cL'')
$$ 
yields $\chi \coloneqq \varphi^{(1)}_{i,j,s,t}(\chi',\chi'') \in \cX(\cL)$ with  $\underline\beta_{\cL}(\chi)=\underline\beta_{\cL'}(\chi')$ and
\begin{equation}
\label{eq:type-1-compatibility-with-alpha-beta}
\alpha^{i',j',k}_{\cL}(\chi) \;=\; 
\begin{cases}
 \alpha^{i',j',k}_{\cL'}(\chi') & \quad\text{for} \;   (i',j') \neq (i,j)   \; \text{or} \;  (i',j') = (i,j),  k \leq s    \\
\alpha^{1,j',k-s}_{\cL''}(\chi'') & \quad \text{for} \; (i',j') = (i,j),  s+1 \leq k \leq s+t \\
 \alpha^{i,j,k-t +1}_{\cL'}(\chi') & \quad\text{for} \;  (i',j') = (i,j), k \geq s+t + 1  .
\end{cases}
\end{equation}
Similarly, using the notation of \eqref{eq:type2_composition}, the type-2 map applied to 
$$
\bigl(\chi',\chi''\bigr)\in \cX(\cL') \leftindex_{(\alpha^{s+1,j,k}_{\cL'})_{1\leq j\leq a, 1\leq k \leq b^j}}\times_{(\beta^j_{\cL''})_{1\leq j\leq b^1+\ldots b^a}} \cX(\cL'')
$$ 
yields 
$\chi \coloneqq \varphi^{(2)}_{s,t,\underline\bm}\bigl(\chi',\chi''\bigr) \in \cX(\cL)$ with   
$\underline\beta_{\cL}(\chi)= \underline\beta_{\cL'}(\chi')$ and
\begin{equation}
\label{eq:type-2-compatibility-with-alpha-beta}
\alpha^{i,j,k}_{\cL}(\chi) \;=\; 
\begin{cases}
 \alpha^{i,j,k}_{\cL'}(\chi') & \quad\text{for} \;   i \leq s    \\
\alpha^{i-s,b^1+\ldots b^{j-1}+j',k'}_{\cL''}(\chi'') & \quad \text{for} \;  
s+1 \leq i \leq s+t, 1\leq j' \leq b^j, 1\leq k' \leq m^{j,j'}_{i-s}  \\
& \qquad\qquad\qquad\quad \text{such that}\;  k = \sum_{1\leq \ell \leq j'-1} m^{j,\ell}_{i-s} + k'   \\
  \alpha^{i-t,j,k}_{\cL'}(\chi') & \quad\text{for} \;   i \geq s+t+1. \\
 \end{cases}
\end{equation}
And, using the notation of \eqref{eq:type3_composition}, the type-3 map applied to 
$$
\bigl(\chi',\chi''\bigr)\in \cX(\cL') \leftindex_{(\alpha^{1,1,k}_{\cL'})_{1\leq k \leq b}}\times_{(\beta^{j''}_{\cL''})_{j\leq j''\leq j+b-1}} \cX(\cL'')
$$ 
yields 
$\chi \coloneqq \varphi^{(3)}_{j,\underline\bm}\bigl(\chi',\chi''\bigr) \in \cX(\cL)$ with   
$
\underline\beta_{\cL}(\chi)= \bigl( \ldots,\beta^{j-1}_{\cL''}(\chi''), \beta^1_{\cL'}(\chi'), \beta^{j+b}_{\cL''}(\chi''), \ldots   \bigr)
$ 
and
\begin{equation}
\label{eq:type-3-compatibility-with-alpha-beta}
\alpha^{i,j',k}_{\cL}(\chi) \;=\; 
\begin{cases}
 \alpha^{i,j',k}_{\cL''}(\chi'') & \quad\text{for} \;   j' \neq j    \\
\alpha^{i,j+j'',k'}_{\cL''}(\chi'') & \quad \text{for} \;  j'=j,  0\leq j'' \leq b-1, 1\leq k' \leq m^{j''+1}_{i}  \\
& \qquad\qquad\qquad\quad \text{such that}\;  k = \sum_{1\leq \ell \leq j''} m^{\ell}_i + k'   .
 \end{cases}
\end{equation}
\smallskip

\item[\bf(iii)]
For any collection of 1-morphisms $\cL$ as in \eqref{eq:curly-L}
the collection of boundary decomposition maps $\varphi^{(\tau)}_{\cdots}$ for $\tau=1,2,3$
is a system of boundary faces for $\cX(\cL)$ in the sense of Definition~\ref{def:mfds_with_faces}.

\smallskip

\item[\bf(iv)]
The boundary decomposition maps are associative in the sense of Remark~\ref{rmk:combinatorics_of_regularized_A-infty_2_categories}.
\null\hfill$\triangle$
\end{enumerate}
\end{definition}

\begin{remark}[Combinatorial underpinnings of regularized $(A_\infty,2)$-flow categories]
\label{rmk:combinatorics_of_regularized_A-infty_2_categories}
In this remark, we make some parts of Definition~\ref{def:general A_infty-2-flow_category} explicit -- specifically, those parts with a significant combinatorial component.
We also give some intuition for how to think about the associativity condition in Definition~\ref{def:general A_infty-2-flow_category} in terms of paths in 2-associahedra.

\medskip

\noindent
{\bf The maps $\wt\Gamma$ in Definition~\ref{def:general A_infty-2-flow_category}(i):}
In each diagram \eqref{eq:2 cover},
$\widetilde W_{\underline\bn} \coloneqq \prod^{K_r}_{1\leq j \leq a}  W_{\bn^j}$ denotes the iterated fiber product of 2-associahedra as in \eqref{eq:type1_composition}.
The map $\wt\Gamma$ is the identification of $\widetilde W_{\underline\bn'} \times \widetilde W_{\underline\bn''}$ with a boundary face of $\widetilde W_{\underline\bn}$, defined in terms of the 2-operadic maps $\Gamma_{2T}$ from \cite[Theorem 4.1]{bottman_2-associahedra}.\footnote{\cite[Thm.\ 4.1]{bottman_2-associahedra} concerns the poset instantiation of 2-associahedra, whereas in this paper we only use the topological instantiation of 2-associahedra, as in \cite{bottman_realizations}.
The recursive identifications of faces of 2-associahedra in \cite[Thm.\ 4.1]{bottman_2-associahedra} also hold for the topological instantiation.
It naturally extends to fiber products.
}

For type-$(\tau=1)$ decomposition, the integer multi-tuples are $\ul\bn' \coloneqq (\ldots,\bn^{j-1},{\bn'}^j,\bn^{j+1},\ldots)$ with ${\bn'}^j:=(\ldots,n^j_{i-1},n^j_i-t+1,n^j_{i+1},\ldots)$, $\bn'' \coloneqq (t)$.
We define the map $\wt\Gamma$ like so:
\begin{itemize}
\item
On all the factors of $\wt W_{\ul\bn'}$ besides the $j$-th, $\wt\Gamma$ is the identity.

\item
On $W_{{\bn'}^j} \times W_{(t)}$ (i.e.\ the remaining factor of $\wt W_{\ul\bn'}$, times $\wt W_{\ul\bn''}$), we define $\wt\Gamma$ to be the map $\Gamma_{2T}$ (as defined in \cite[Theorem 4.1]{bottman_2-associahedra}) for the following tree-pair $2T$:
\begin{figure}[H]
\tiny
\centering
\def\svgwidth{0.4\columnwidth}
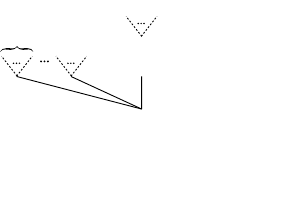
\end{figure}
\end{itemize}
Since both the identity maps and $\Gamma_{2T}$ are the identity on the underlying associahedron $K_r$, the map $\wt\Gamma$ respects the fiber product condition in the target $\wt W_{\ul\bn}$.

For type-$(\tau=2)$ decomposition, the integer multi-tuples are  
$\underline\bn' \coloneqq ({\bn'}^1,\ldots,{\bn'}^a)$ 
with ${\bn'}^j:=(\ldots,n^j_s,b^j,n^j_{s+t+1},\ldots)$, 
$\underline\bn'' \coloneqq ({\bn''}^1,\ldots,{\bn''}^{b^1+\ldots b^a})$ 
with ${\bn''}^{b^1+\ldots b^{j-1}+j'}  \coloneqq (m^{j,j'}_1-m^{j,j'-1}_1 ,\ldots, m^{j,j'}_t-m^{j,j'-1}_t)$.
We define the map $\wt\Gamma$ in the following way:
\begin{itemize}
\item
For any $j$ with $1 \leq j \leq a$, we define $\wt\Gamma$ on the factor $W_{{\bn'}^j}$ of $\wt W_{\ul\bn'}$ and on the factor $\prod_{1\leq j'\leq b^j}^{K_r} W_{{\bn''}^{b^1+\cdots+b^{j-1}+j'}}$ of $\wt W_{\ul\bn''}$ to be $\Gamma_{2T^j}$ for the following tree-pair $2T^j$:
\end{itemize}
\begin{figure}[H]
\tiny
\centering
\def\svgwidth{0.45\columnwidth}
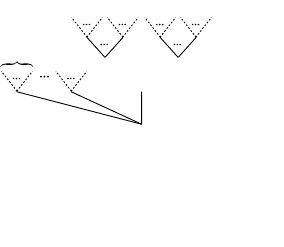
\end{figure}
Since the underlying tree in $2T^j$ is independent of $j$, the map $\wt\Gamma$ respects the fiber product condition in the target.

For type-$(\tau=3)$ decomposition, the integer multi-tuples are 
${\bn'}:=(b)$, 
$\underline\bn'' \coloneqq ({\bn''}^1,\ldots,{\bn''}^{a+b-1})$ 
with ${\bn''}^{j'<j}  \coloneqq \bn^{j'}$, 
${\bn''}^{j+j'<j+b}  \coloneqq (m^{j,j'+1}_1-m^{j,j'}_1 ,\ldots, m^{j,j'+1}_t-m^{j,j'}_r)$,
${\bn''}^{j'\geq j+b}  \coloneqq \bn^{j'-b+1}$.
We define the map $\wt\Gamma$ like so:
\begin{itemize}
\item
On the first $j-1$ and last $a-j$ factors of $\wt W_{\ul\bn''}$, $\wt\Gamma$ is the identity.

\item
On $W_{(b)} \times \prod_{0 \leq j'' \leq b-1}^{K_r} W_{{\bn''}^{j+j''}}$, we define $\wt\Gamma$ as the map $\Gamma_{2T}$ for the following tree-pair $2T$:
\begin{figure}[H]
\tiny
\centering
\def\svgwidth{0.25\columnwidth}
\begingroup%
  \makeatletter%
  \providecommand\color[2][]{%
    \errmessage{(Inkscape) Color is used for the text in Inkscape, but the package 'color.sty' is not loaded}%
    \renewcommand\color[2][]{}%
  }%
  \providecommand\transparent[1]{%
    \errmessage{(Inkscape) Transparency is used (non-zero) for the text in Inkscape, but the package 'transparent.sty' is not loaded}%
    \renewcommand\transparent[1]{}%
  }%
  \providecommand\rotatebox[2]{#2}%
  \newcommand*\fsize{\dimexpr\f@size pt\relax}%
  \newcommand*\lineheight[1]{\fontsize{\fsize}{#1\fsize}\selectfont}%
  \ifx\svgwidth\undefined%
    \setlength{\unitlength}{69.19571691bp}%
    \ifx\svgscale\undefined%
      \relax%
    \else%
      \setlength{\unitlength}{\unitlength * \real{\svgscale}}%
    \fi%
  \else%
    \setlength{\unitlength}{\svgwidth}%
  \fi%
  \global\let\svgwidth\undefined%
  \global\let\svgscale\undefined%
  \makeatother%
  \begin{picture}(1,1.54758046)%
    \lineheight{1}%
    \setlength\tabcolsep{0pt}%
    \put(0,0){\includegraphics[width=\unitlength,page=1]{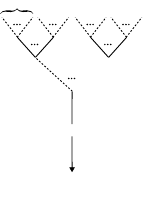}}%
    \put(0.08402015,1.53202081){\color[rgb]{0,0,0}\makebox(0,0)[lt]{\lineheight{0}\smash{\begin{tabular}[t]{l}$m_1^1$\end{tabular}}}}%
    \put(0,0){\includegraphics[width=\unitlength,page=2]{type-3_comp_on_Wn.pdf}}%
    \put(0.33989653,1.53202081){\color[rgb]{0,0,0}\makebox(0,0)[lt]{\lineheight{0}\smash{\begin{tabular}[t]{l}$m_r^1$\end{tabular}}}}%
    \put(0,0){\includegraphics[width=\unitlength,page=3]{type-3_comp_on_Wn.pdf}}%
    \put(0.59788382,1.53202081){\color[rgb]{0,0,0}\makebox(0,0)[lt]{\lineheight{0}\smash{\begin{tabular}[t]{l}$m_1^b$\end{tabular}}}}%
    \put(0,0){\includegraphics[width=\unitlength,page=4]{type-3_comp_on_Wn.pdf}}%
    \put(0.85376019,1.53202081){\color[rgb]{0,0,0}\makebox(0,0)[lt]{\lineheight{0}\smash{\begin{tabular}[t]{l}$m_r^b$\end{tabular}}}}%
    \put(0,0){\includegraphics[width=\unitlength,page=5]{type-3_comp_on_Wn.pdf}}%
    \put(0.4782749,0.24767389){\color[rgb]{0,0,0}\makebox(0,0)[lt]{\lineheight{0}\smash{\begin{tabular}[t]{l}$r$\end{tabular}}}}%
  \end{picture}%
\endgroup%

\end{figure}
\end{itemize}
Since both the identity maps and $\Gamma_{2T}$ are the identity on the underlying associahedron $K_r$, the map $\wt\Gamma$ respects the fiber product condition in the target.

\medskip

\noindent
{\bf The associativity condition in Definition~\ref{def:general A_infty-2-flow_category}(iv):}
Making this associativity condition explicit involves significant input from the theory of 2-associahedra.
We state associativity in terms of new objects called \emph{coppices of tree-pairs}, which label the strata in fiber products of 2-associahedra.
For a more detailed account of the combinatorics of fiber products of 2-associahedra, see \cite{b:homotopies}.
We note that when $a=1$, this associativity condition is a lift of the associativity condition for the relative 2-operad of 2-associahedra, as in \cite[Definition 2.3]{bottman_carmeli}.

Throughout this part of the remark, we will relax the stability condition on tree-pairs, in order to be able to incorporate unstable operations.
Specifically, for any tree-pair $2T = T_b \to T_s$ in this description of the associativity condition, the bubble tree $T_b$ is not required to satisfy \cite[Def.~3.1(stability)]{bottman_2-associahedra}.
The seam tree $T_s$ must still satisfy the stability condition in \cite[Def.~2.2]{bottman_2-associahedra}.
We denote the enlarged 2-associahedra consisting of these tree-pairs by $W_\bn^{\text{us}}$ (the ``$\text{us}$'' stands for ``unstable''); we denote fiber products thereof by $\wt W_{\ul\bn}^{\text{us}}$.

\medskip

\noindent
{\bf (Coppices of tree-pairs and terminology for relative location of vertices in trees)}
For any $r \geq 1$ and $\underline{\mathbf n} = (\mathbf n^1,\ldots,\mathbf n^a) \subset \mathbb{Z}_{\geq0}^r$, a \emph{coppice\footnote{In arboricultural parlance, coppicing a tree results in a collection of trees that share a single root system.} of tree-pairs of shape $\ul\bn$} is an element
\begin{align}
\underline{2T} = (2T^1,\ldots,2T^a) \in \widetilde W_{\underline{\mathbf n}}^{\text{us}} \coloneqq W_{\mathbf n^1}^{\text{us}} \times_{K_r} \cdots \times_{K_r} W_{\mathbf n^a}^{\text{us}}.
\end{align}
The fiber-product condition on $2T^1, \ldots, 2T^a$ exactly says that the seam trees $T_s^1, \ldots, T_s^a$ coincide.
We may therefore denote them all by the single tree $T_s$.

Suppose that $T$ is a rooted ribbon tree.
For distinct $v_1, v_2 \in V(T)$, we say that \emph{$v_1$ is above $v_2$} if the path from $v_1$ to the root passes through $v_2$.
If neither $v_1$ nor $v_2$ is above the other, we say \emph{$v_1$ is to the left of $v_2$} if $v_1$ appears before $v_2$ in the preorder traversal.
(If we depict $T$ with the root at the bottom, as in \cite{bottman_2-associahedra}, $v_1$ literally appears to the left of $v_2$.)

\medskip

\noindent
{\bf (Labelings)}
Suppose that $\underline{2T} \in \widetilde W_{\underline{\mathbf n}}^{\text{us}}$ is a coppice of tree-pairs.
A \emph{labeling} of $\underline{2T}$ is the choice of a sequence $\cL(v) = \bigl(\cL(v)_0,\ldots,\cL(v)_{\#\text{in}(v)}\bigr)$ for every $v \in V_{\text{seam}}(T_b^j)$.
We require these data to satisfy the following conditions:
\begin{itemize}
\item
Fix any $1\leq j \leq a$.
Recall from \cite{bottman_2-associahedra} that $f^j$ is the map $T_b^j \to T_s$, and that $T_{s,f^j(v)}$ is the subtree of $T_s$ with root $f^j(v)$ and with non-root vertices all elements of $V(T_s)$ that are above $v$.
Suppose that the leaves of $T_s$ in $T_{s,f^j(v)}$ are $\lambda_{s+1}^{T_s}, \ldots, \lambda_{s+t}^{T_s}$.
Then each element of $\cL(v)$ is required to be a 1-morphism from $M_s$ to $M_{s+t}$.

\item
Following \cite[Definition-Lemma 3.19]{bottman_2-associahedra}, we equip $V_{\text{seam}}(T_b^j)$ with a partial order in the following way:
\begin{itemize}
\item
Suppose $v_1, v_2 \in V_{\text{seam}}(T_b^j)$ have $f^j(v_1) = f^j(v_2)$.
Suppose that $v_1$ is to the left of $v_2$.
Then we declare $v_1 < v_2$.
\end{itemize}
Suppose $v_1, v_2 \in V_{\text{seam}}(T_b^j)$ satisfy $v_1 \lessdot v_2$, with respect to the partial order just defined.
Then we require $\cL(v_1)_{\#\text{in}(v_1)} = \cL(v_2)_0$.

\item
Fix $v \in V_{\text{seam}}(T_b^j)$ and $v' \in \text{in}(v)$.
Suppose $v'$ is the $m$-th element of $\text{in}(v)$.
Suppose $v' \in V_{\text{comp}(T_b^j)}$, and denote $\text{in}(v') \eqqcolon \{w_1,\ldots,w_\ell\}$.
Then we require:
\begin{align}
&\cL(v)_{m-1}
=
\cL(w_1)_0\circ\cdots\circ\cL(w_\ell)_0,
\qquad
\cL(v)_m
=
\cL(w_1)_{\#\text{in}(w_1)}\circ\cdots\circ\cL(w_\ell)_{\#\text{in}(w_\ell)}.
\end{align}
\end{itemize}
For any $\cL = \Bigl(L_{(i-1)i}^{j,k} \in  \leftindex^1\Mor(M_{i-1},M_i) \Bigr)_{(i,j,k)\in I^{\, r,a}_{\underline\bn} }$ and coppice $\ul{2T} \in \wt W_{\ul\bn}^{\text{us}}$, $\cL$ induces a labeling of $\ul{2T}$ in a canonical way, which we denote $\cL_{\ul{2T}}$ (see \cite{b:homotopies}).\footnote{
Here $\cL$ can be viewed as a labeling of the coppice corresponding to the interior of $\wt W_{\ul\bn}$.
}

\medskip

\noindent
{\bf (Decomposition maps, and the associativity requirement)}
In Definition~\ref{def:general A_infty-2-flow_category} we defined boundary decomposition maps $\varphi^{(\tau=1,2,3)}_{\cdots}$.
By composing these maps, we can produce finer decomposition maps, which we think of as being associated to coppices of tree-pairs.
The associativity requirement (the topic of condition (iv)) is the assertion that for a given coppice, the order in which we iteratively apply the boundary decomposition maps does not affect the final map.

More precisely, fix $\cL = \Bigl(L_{(i-1)i}^{j,k} \in  \leftindex^1\Mor(M_{i-1},M_i) \Bigr)_{(i,j,k)\in I^{\, r,a}_{\underline\bn} }$ and a coppice $\ul{2T} \in \wt W_{\ul\bn}^{\text{us}}$.
We would like to define a \emph{decomposition map}
\begin{align}
\label{eq:decomp_map}
\varphi^{\underline{2T}}_{\cL}
\colon
\prod_{1\leq j \leq a}\prod_{\alpha \in V_{\text{comp}}^1(T_b^j)} \mathcal X(\cL_\alpha) \times \prod_{\rho \in V_{\text{int}}(T_s)}\mathcal X(\cL_\rho)
\: \longrightarrow \: 
\mathcal X(\cL),
\end{align}
where each $\cL_\alpha$ and $\cL_\rho$ are defined as follows:
\begin{itemize}
\item
To define $\cL_\alpha$ denote 
by $\beta \in V_{\text{seam}}(T_b^j)$ the incoming vertex of $\alpha$.
Then $\cL_\alpha \coloneqq \bigl(\cL_{\underline{2T}}(\beta)\bigr)$ is a collection of 1-morphisms of shape $\bigl(\bigl(\#\text{in}(\beta)+1\bigr)\bigr)$.

\item
To define $\cL_\rho$ denote for each $1\leq j \leq a$ 
the elements of $V_{\text{comp}}^{\geq2}(T_b^j) \cap (f^j)^{-1}\{\rho\}$ by $\bigl(\alpha_m^j\bigr)_{1 \leq m \leq \ell_j}$, where the ordering in $m$ corresponds to the left-to-right partial-ordering on $V_{\text{comp}}(T_b^j)$.
For any $j, m$, denote the elements of $\text{in}(\alpha_m^j)$ by $\bigl(\beta^j_{m,o}\bigr)_{1\leq 0 \leq 1,\#\text{in}(\rho)}$.
Then $\cL_\rho$ is a collection of 1-morphisms of shape
\begin{align}
&\bigl(\bigl(\#\text{in}(\beta^1_{1,1}),\ldots,\#\text{in}(\beta^1_{1,\#\text{in}(\rho)})\bigr),\ldots,\bigl(\#\text{in}(\beta^1_{\ell_1,1}),\ldots,\#\text{in}(\beta^1_{\ell_1,\#\text{in}(\rho)})\bigr),\ldots
\\
&\hspace{1in}\ldots,\bigl(\#\text{in}(\beta^a_{1,1}),\ldots,\#\text{in}(\beta^a_{1,\#\text{in}(\rho)})\bigr),\ldots,\bigl(\#\text{in}(\beta^a_{\ell_a,1}),\ldots,\#\text{in}(\beta^a_{\ell_a,\#\text{in}(\rho)})\bigr)\bigr),
\nonumber
\end{align}
defined like so:
\begin{align}
\cL_\rho
\coloneqq
\left(\begin{array}{ccc}
\cL(\beta^1_{1,1})_0 & & \cL(\beta^1_{1,\#\text{in}(\rho)})_0 \\
\vdots & \cdots & \vdots \\
\cL(\beta^1_{1,1})_{\#\text{in}(\beta^1_{1,1})} & & \cL(\beta^1_{1,\#\text{in}(\rho)})_{\#\text{in}(\beta^1_{1,\#\text{in}(\rho)})} \\
& \vdots & \\
\cL(\beta^1_{\ell_1,1})_0 & & \cL(\beta^1_{\ell_1,\#\text{in}(\rho)})_0 \\
\vdots & \cdots & \vdots \\
\cL(\beta^1_{\ell_1,1})_{\#\text{in}(\beta^1_{\ell_1,1})} & & \cL(\beta^1_{\ell_1,\#\text{in}(\rho)})_{\#\text{in}(\beta^1_{\ell_1,\#\text{in}(\rho)})} \\
& \vdots & \\ \cL(\beta^a_{1,1})_0 & & \cL(\beta^a_{1,\#\text{in}(\rho)})_0 \\
\vdots & \cdots & \vdots \\
\cL(\beta^a_{1,1})_{\#\text{in}(\beta^a_{1,1})} & & \cL(\beta^a_{1,\#\text{in}(\rho)})_{\#\text{in}(\beta^a_{1,\#\text{in}(\rho)})} \\
& \vdots & \\ \cL(\beta^a_{\ell_a,1})_0 & & \cL(\beta^a_{\ell_a,\#\text{in}(\rho)})_0 \\
\vdots & \cdots & \vdots \\
\cL(\beta^a_{\ell_a,1})_{\#\text{in}(\beta^a_{\ell_a,1})} & & \cL(\beta^a_{\ell_a,\#\text{in}(\rho)})_{\#\text{in}(\beta^a_{\ell_a,\#\text{in}(\rho)})}
\end{array}\right).
\end{align}
\end{itemize}
The map $\varphi^{\ul{2T}}_\cL$ in \eqref{eq:decomp_map} is now defined by iteratively applying boundary decomposition maps $\varphi^{(\tau=1,2,3)}_{\cdots}$.\footnote{
For instance, for $\rho, \sigma \in V_{\text{int}(T_s)}$ with $\sigma \in \text{in}(\rho)$, we can apply an instance of $\varphi^{(2)}_{\cdots}$ to map the factors $\cX(\cL_\rho) \times \cX(\cL_\sigma)$ to a single $\cX(\cL')$.}
There may be several ways to compose a sequence of $\varphi^{(\tau)}_{\cdots}$'s to obtain a map with the domain and target of $\varphi^{\ul{2T}}_\cL$, and the associativity condition asserts that no matter what boundary decomposition maps we compose in what order, we end up with the same decomposition map $\varphi^{\ul{2T}}_\cL$.
That is, each instance of $\varphi^{\ul{2T}}_\cL$ is well-defined.

\medskip

\noindent
{\bf (Intuition for the associativity condition in terms of paths in 2-associahedra)}
In this final part of the remark, which is solely motivational, we consider only those decomposition maps associated to stable coppices $\ul{2T} \in \wt W_{\ul\bn}$.

Following \cite{bottman_2-associahedra}, one can move from a stratum $F$ of a 2-associahedron to adjacent strata $G$ with $\dim G = \dim F - 1$ via type-1, -2, and -3 moves\footnote{In \cite{bottman_2-associahedra}, the first author dealt with the poset incarnation of 2-associahedra.
The moves in \cite{bottman_2-associahedra} apply just as well to the incarnation of 2-associahedra as topological spaces as in \cite{bottman_realizations}.}.
A generalization of this allows one to define type-1, -2, and -3 moves for fiber products $\wt W_{\ul\bn}$ of 2-associahedra.
For any stratum $F$ of $\wt W_{\ul\bn}$, this allows us to enumerate the strata $G$ with $G \subset \text{cl}(F)$ and $\codim G = \codim F + 1$.
Applying this to the top stratum (corresponding to the interior of $\wt W_{\ul\bn}$), we can enumerate the boundary (i.e.\ codimension-1) strata of $\wt W_{\ul\bn}$; we refer to them as being of type 1, 2, or 3, depending on the type of the move leading us there from the interior.
Just as any strata of $\wt W_{\ul\bn}$, the type-1, -2, and -3 boundary strata of $\wt W_{\ul\bn}$ correspond to coppices $\ul{2T}$.
The corresponding decomposition maps $\varphi^{\ul{2T}}_\cL$ as described above are exactly the boundary decomposition maps $\varphi^{(\tau)}_{\cdots}$, as in Definition~\ref{def:general A_infty-2-flow_category}.

More generally, we defined the decomposition maps $\varphi^{\ul{2T}}_\cL$ by composing boundary decomposition maps $\varphi^{(\tau=1,2,3)}_{\cdots}$.
Each coppice $\ul{2T}$ corresponds to a face $F$ of $\wt W_{\ul\bn}$.
Now composing boundary decomposition maps to define $\varphi^{\ul{2T}}_\cL$ corresponds to choosing a sequence of faces
\begin{align}
\text{int}\:\wt W_{\ul\bn} = F_0 \supset F_1 \cdots \supset F_k = F
\end{align}
such that $F_i \subset \text{cl}(F_{i-1})$ and $\codim F_i = i$.
The different ways to produce $\varphi^{\ul{2T}}_\cL$ as a composition of boundary decomposition maps (which must coincide by the associativity condition in Definition~\ref{def:general A_infty-2-flow_category}(iv)) correspond to the different possible sequences of this form.
\null\hfill$\triangle$
\end{remark}

\begin{remark}\label{rmk:2-a-b-maps}
To simplify notation later on, we summarize the evaluation maps in \eqref{eq:2-a-b-maps} by
\begin{align}
\xymatrix{
&&&  \leftindex^2\Mor^{\rm in}_{2\sC}(\cL) \coloneqq \prod_{ (i,j,k)\in  {'\!I^{\, r,a}_{\underline\bn}}}   \leftindex^2\Mor(L_{(i-1)i}^{j,k-1},L_{(i-1)i}^{j,k})   \\
\cX(\cL) \ar@/^1.0pc/[urrr]^(0.2){\underline\alpha_\cL:=(\alpha^{i,j,k}_\cL)_{ (i,j,k)\in  {'\!I^{\, r,a}_{\underline\bn}}}}\ar[rrr]_(0.2){\underline\beta_\cL:=(\beta^j_\cL)_{1\leq j\leq a}} &&& 
\leftindex^2\Mor^{\rm out}_{2\sC}(\cL) \coloneqq \prod_{1\leq j\leq a} \leftindex^2\Mor(L_{01}^{j,0}\circ\cdots L_{(r-1)r}^{j,0},L_{01}^{j,n_1^j}\circ\cdots L_{(r-1)r}^{j,n_r^j}) .}
\end{align}
With that we can summarize the three types of boundary decomposition maps as
\begin{equation}  \label{eq:2-decomp map}
\varphi^{(\tau)}_{*_\tau}
\: \colon \quad
\cX(\cL'\coloneqq\cL^{(\tau) \rm in}_{*_\tau}) \:\leftindex_{\underline{\alpha}'} \times_{\underline\beta''} \: \cX(\cL''\coloneqq\cL^{(\tau) \rm out}_{*_\tau}) \quad \longrightarrow \quad \cX(\cL) .
\end{equation}
Here the indices $*_\tau$ are elements of the following subsets of integers: 
\begin{equation*} \label{eq:flow123}
\begin{array}{l}
*_1 \; \in \; \bigl\{ (i,j,s,t) \,\big|\, 
1\leq i \leq r, 1 \leq j \leq a, s,t\geq 0 \:
\; \text{so that}\; s+t \leq n_i^j  \bigr\} ;  \\
*_2 \; \in \; \bigl\{ \bigl( s,t,\underline m =( \underline m^{j,j'} )_{1\leq j \leq a, 1\leq j'\leq b^j} \bigr)  \,\big|\, 
s\geq 0, 2\leq t \leq r-1 
\; \text{so that}\;
\ s+t\leq r, \\
\qquad\qquad\qquad\qquad\qquad\qquad\qquad\qquad\quad
\forall 1\leq j \leq a : 
(n^j_{s+1},\ldots,n^j_{s+t})  
= \bm^{j,1} + \ldots + \bm^{j,b^j}   \bigr\} ;
\\
*_3 \; \in \; \bigl \{  \bigl( j, \underline m = ( \underline m^{j'} )_{1\leq j'\leq b} \,\big)
\,\big|\,  r\geq 2,  1 \leq j \leq a ,  \bm^1,\ldots,\bm^b \in \bZ_{\geq0}^r 
\; \text{so that}\;
\bm^1+\cdots+\bm^b  \bigr\}  .
\end{array}
\end{equation*}
For $\tau=1,2$ the fiber product is over 
$\leftindex^2\Mor^{\rm out}_{2\sC}(\cL''=\cL^{(\tau) \rm out}_{*_\tau})$, whereas for $\tau=3$ the fiber product is over $\leftindex^2\Mor^{\rm in}_{2\sC}(\cL'=\cL^{(\tau) \rm in}_{*_\tau})$, via the maps
\begin{equation*}
\begin{array}{ll}
(\tau=1):  \quad 
\underline\alpha' \coloneqq \alpha^{i,j,s+1}_{\cL'} 
&\quad\text{and}\quad 
\underline\beta'' \coloneqq \beta^1_{\cL''} = \underline \beta_{\cL''} ; \\
(\tau=2): \quad 
\underline{\alpha}'\coloneqq \bigl(  \alpha^{s+1,j,k}_{\cL'} \bigr)_{1\leq j \leq a, 1\leq k\leq b^j}
&\quad\text{and}\quad 
\underline{\beta}''\coloneqq \bigl(\beta_{\cL''}^j \bigr)_{1\leq j \leq b^1+\ldots b^a} = \underline \beta_{\cL''}; 
\\
(\tau=3): \quad 
\underline{\alpha}'\coloneqq \bigl(  \alpha^{1,1,k}_{\cL'} \bigr)_{1\leq k \leq b} = \underline\alpha_{\cL'} 
&\quad\text{and}\quad 
\underline{\beta}''\coloneqq \bigl(\beta_{\cL''}^{j+j''} \bigr)_{0\leq j'' \leq b-1}
.
\end{array}
\end{equation*}
The incoming resp.\ outgoing 2-morphism spaces associated to $\cL$ are then identified with the product of the incoming resp.\ outgoing 2-morphism spaces associated to $\cL'=\cL^{(\tau) \rm in}_{*_\tau}$ and $\cL''=\cL^{(\tau) \rm out}_{*_\tau}$ minus those factors that are utilized in the fiber product, that is 
\begin{equation*}
\begin{array}{ll}
(\tau=1,2): 
&\leftindex^2\Mor^{\rm in}_{2\sC}(\cL) = \quotient{\leftindex^2\Mor^{\rm in}_{2\sC}(\cL')}{\leftindex^2\Mor^{\rm out}_{2\sC}(\cL'')} \times \leftindex^2\Mor^{\rm in}_{2\sC}(\cL'') 
\: \:\text{and}\: \:
\leftindex^2\Mor^{\rm out}_{2\sC}(\cL) = \leftindex^2\Mor^{\rm out}_{2\sC}(\cL') ; 
 \\
(\tau=3): \;
&\leftindex^2\Mor^{\rm in}_{2\sC}(\cL) = \leftindex^2\Mor^{\rm in}_{2\sC}(\cL'')
\: \:\text{and}\: \: 
\leftindex^2\Mor^{\rm out}_{2\sC}(\cL) = \quotient{\leftindex^2\Mor^{\rm out}_{2\sC}(\cL'')}{\leftindex^2\Mor^{\rm in}_{2\sC}(\cL')} \times \leftindex^2\Mor^{\rm out}_{2\sC}(\cL')  .
\end{array}
\end{equation*}
Along with this, the compatibilities \eqref{eq:type-1-compatibility-with-alpha-beta}-- \eqref{eq:type-3-compatibility-with-alpha-beta} of the boundary decomposition maps with the evaluation maps in  Definition~\ref{def:A_infty-2-flow_category}~(ii) can be summarized for $(\chi',\chi'')\in \cX(\cL') \leftindex_{\underline\alpha'}\times_{\underline\beta''}\cX(\cL'')$ by
\begin{equation*}
\begin{array}{ll}
(\tau=1,2): 
&
\underline\alpha_\cL\bigl(\phi^{(\tau)}_{*_\tau}(\chi',\chi'')\bigr) =  \; \quotient{ \underline\alpha_{\cL'}(\chi') }{ \underline\beta_{\cL''}(\chi'') } \times \underline\alpha_{\cL''}(\chi'')   
\quad \text{and}\quad
\underline\beta_\cL\bigl(\phi^{(\tau)}_{*_\tau}(\chi',\chi'')\bigr) = \underline\beta_{\cL'}(\chi') ; 
 \\
(\tau=3): \;
&
\underline\alpha_\cL\bigl(\phi^{(\tau)}_{*_\tau}(\chi',\chi'')\bigr) = \underline\alpha_{\cL''}(\chi'') 
\quad \text{and}\quad
\underline\beta_\cL\bigl(\phi^{(\tau)}_{*_\tau}(\chi',\chi'')\bigr) =  \; \quotient{ \underline\beta_{\cL''}(\chi'') }{ \underline\alpha_{\cL'}(\chi') } \times \underline\beta_{\cL'}(\chi')  .
\end{array}
\end{equation*}
\null\hfill$\triangle$
\end{remark}

As a first step to making sense of this lengthy definition we show that, as claimed above, the morphisms between two fixed objects inherit the structure of an $A_\infty$-flow category.

\begin{lemma} \label{lem:mor is cat}
Let $2\sC$ be a regularized $(A_\infty,2)$-flow category as in Definition~\ref{def:general A_infty-2-flow_category}.
Then for any two objects $M_0, M_1 \in \Ob$ the structure of $2\sC$ restricts to a  
regularized $A_\infty$-flow category $\Mor(M_0,M_1)$ as in Definition~\ref{def:general_A_infty-flow_category}
as follows:  
\begin{itemize}
\item
Objects are given by the 1-morphisms $\Ob_{\Mor(M_0,M_1)} \coloneqq \leftindex^1\Mor_{2\sC}(M_0,M_1)$.

\smallskip

\item
For every $L,L' \in  \Ob_{\Mor(M_0,M_1)}$, the associated $(*)$base space is the 2-morphism space $\Mor_{\Mor(M_0,M_1)}(L,L') \coloneqq \leftindex^2\Mor_{2\sC}(L,L')$.

\smallskip

\item
For $L^0, \ldots, L^n \in \Ob_{\Mor(M_0,M_1)}$ the associated $(*)$moduli space is 
$\cX(L^0,\ldots,L^n) \coloneqq \cX(\cL)$ for the tuple $\cL=( L^{1,k}_{01} = L^k )_{0\leq k \leq n}$ with $r=1$, $a=1$, $n^1_1=n$.
The $(*)$maps are 
$\alpha^k_{L^0 \ldots L^n}  \coloneqq {\alpha^{1,1,k}_\cL}: \cX(\cL) \to \leftindex^2\Mor_{2\sC}(L^{k-1},L^k)$ for $1\leq k \leq n$, 
$\beta_{L^0 \ldots L^n}  \coloneqq   \beta^1_\cL: \cX(\cL) \to \leftindex^2\Mor_{2\sC}(L^0,L^k)$, 
and
$p_{L^0 \ldots L^n}  \coloneqq  p_\cL: \cX(\cL) \to W_{\bn=(n)} = K_n$.
\item
For $\cL=(L^0,\ldots,L^n) \subset \Ob_{\Mor(M_0,M_1)}$ and indices $s,t\geq 0$ with $s+t\leq n$ 
the boundary decomposition map is given by 
$\varphi^{s,t}_{L^0 \ldots L^n} \coloneqq  \varphi^{(1)}_{1,1,s,t}: \cX(\cL') \leftindex_{\alpha^{1,1,s+1}}\times_{\beta^1}\: \cX(\cL'') \to \cX(\cL)$
for $\cL'=(L^0,\ldots,L^s,L^{s+t},\ldots,L^n)$ and $\cL''= (L^s,\ldots,L^{s+t})$.
\end{itemize}
\end{lemma}
\begin{proof}
In the statement above, we implicitly use the fact that the 2-associahedra for $r=1$ can be identified with the standard associahedra $W_{\bn=(n)} = K_n$.
This statement is proven for the poset incarnations of the 1- and 2-associahedra in \cite[Lemma 3.9]{bottman_2-associahedra}, by constructing a poset isomorphism $W_{\bn=(n)} \to K_n$.
This isomorphism lifts to a homeomorphism of the topological incarnations (as defined in \cite{bottman_realizations}) of $W_{\bn=(n)}$ and $K_n$.
Indeed, to lift the poset isomorphism to a homeomorphism, we just need to identify the following two objects: (1) a configuration of a single vertical line with $k$ marked points in $\bR^2$, modulo overall translations and positive dilations; and (2) a configuration of $k$ marked points on $\bR$, modulo overall translations and positive dilations.
These can be identified like so: by using horizontal translations, we can arrange for the vertical line in (1) to have $x$-coordinate 0.
The remaining automorphisms are vertical translations and positive dilations, so this configuration can evidently be identified with the configuration in (2).
\noindent
This also identifies the operadic compositions used in (i) below.

It remains to verify the properties required by Definition~\ref{def:general_A_infty-flow_category}.
\begin{enumerate}
\item[\bf(i)]
The boundary decomposition maps cover the operadic composition on associahedra 
since the commuting square for type-1 decomposition over the 2-associahedra in Definition~\ref{def:general A_infty-2-flow_category} restricts in case $r=1$, $a=1$ to  
\begin{gather}
\xymatrix{
\cX(\cL') \times \cX(\cL'')
\ar[rr]^(0.55){\varphi^{(1)}_{1,1,s,t}} \ar[d]_{p_{\cL'}\times p_{\cL''} } &&
\cX(\cL) \ar[d]^{p_\cL} \\
\bigl( K_{n-t+1}=W_{(n-t+1)}\bigr) \times \bigl( W_{(t)}=K_t \bigr) \ar[rr]_{\qquad\qquad\quad\: \Gamma_{2T}=\sigma_{s+1}} && W_{(n)}=K_n
}
\end{gather}

\item[\bf(ii)]
The boundary decomposition maps are compatible with the maps $\alpha^k_{L^0\ldots L^n}$ and $\beta_{L^0\ldots L^n}$ since the special case $r=1$, $a=1$ of compatibility for the type-1 map in \eqref{eq:type-1-compatibility-with-alpha-beta} applied to 
$(\chi',\chi'')\in \cX(\cL')\leftindex_{\alpha^{s+1}}\times_\beta \cX(\cL'')$ 
and $\chi \coloneqq \varphi_{\cL}^{s,t}(\chi',\chi'')$ 
with
$\cL=(L^0 \ldots L^n)$, 
$\cL'=(L^0 \ldots L^s,L^{s+t} \ldots L^n)$, 
$\cL''=(L^s \ldots L^{s+t})$
yields $\beta^1_{\cL}(\chi)= \beta^1_{\cL'}(\chi')$ and
\begin{equation}
\alpha^{1,1,k}_{\cL}(\chi) \;=\; 
\begin{cases}
 \alpha^{1,1,k}_{\cL'}(\chi') & \quad\text{for} \;  k \leq s    \\
\alpha^{1,1,k-s}_{\cL''}(\chi'') & \quad \text{for} \; s+1 \leq k \leq s+t \\
 \alpha^{1,1,k-t+1}_{\cL'}(\chi') & \quad\text{for} \;  k\geq s+t + 1 .
\end{cases}
\end{equation}
This confirms \eqref{eq:compatibility-with-alpha-beta}.

\smallskip

\item[\bf(iii)]
For any $\cL=(L^0,\ldots,L^n) \subset \Ob_{\Mor(M_0,M_1)}$ the 
collection of boundary decomposition maps $\varphi_{L^0 \ldots L^n}^{s,t}$ 
for $s,t\geq 0$ with $s+t\leq n$ form a system of boundary faces for $\cX(L^0,\ldots,L^n)$ 
since this is the special case $r=1$, $a=1$ of Definition~\ref{def:general A_infty-2-flow_category}.
Indeed, the boundary decomposition maps of type-2 require $r\geq 3$ whereas type-3 requires $r\geq 2$.
And for type-1 the collection $\varphi^{(1)}_{1,1,s,t}=\varphi_{L^0 \ldots L^n}^{s,t}$ for $s,t\geq 0$ with $s+t \leq n$ forms a system of boundary faces for $\cX(\cL)$ and agrees with the collection required in Definition~\ref{def:general_A_infty-flow_category}.

\smallskip

\item[\bf(iv)]
The boundary decomposition maps are associative in the sense of \eqref{eq:associativity-case-1} and \eqref{eq:associativity-case-2} by the special case $r=1$, $a=1$ of associativity in Definition~\ref{def:general A_infty-2-flow_category}.
Indeed, \eqref{eq:associativity-case-1} follows by applying the associativity axiom to the decomposition map $\varphi^{2T}_{\ul\bn}$, where $2T$ is the tree-pair with trivial seam tree $T_s = \bullet$ and the following bubble tree:
\begin{figure}[H]
\tiny
\centering
\def\svgwidth{0.175\columnwidth}
\begingroup%
  \makeatletter%
  \providecommand\color[2][]{%
    \errmessage{(Inkscape) Color is used for the text in Inkscape, but the package 'color.sty' is not loaded}%
    \renewcommand\color[2][]{}%
  }%
  \providecommand\transparent[1]{%
    \errmessage{(Inkscape) Transparency is used (non-zero) for the text in Inkscape, but the package 'transparent.sty' is not loaded}%
    \renewcommand\transparent[1]{}%
  }%
  \providecommand\rotatebox[2]{#2}%
  \newcommand*\fsize{\dimexpr\f@size pt\relax}%
  \newcommand*\lineheight[1]{\fontsize{\fsize}{#1\fsize}\selectfont}%
  \ifx\svgwidth\undefined%
    \setlength{\unitlength}{36.16171057bp}%
    \ifx\svgscale\undefined%
      \relax%
    \else%
      \setlength{\unitlength}{\unitlength * \real{\svgscale}}%
    \fi%
  \else%
    \setlength{\unitlength}{\svgwidth}%
  \fi%
  \global\let\svgwidth\undefined%
  \global\let\svgscale\undefined%
  \makeatother%
  \begin{picture}(1,1.9717714)%
    \lineheight{1}%
    \setlength\tabcolsep{0pt}%
    \put(0,0){\includegraphics[width=\unitlength,page=1]{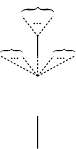}}%
    \put(0.76884465,1.38625847){\color[rgb]{0,0,0}\makebox(0,0)[lt]{\lineheight{0}\smash{\begin{tabular}[t]{l}$t-s'$\end{tabular}}}}%
    \put(0.47492199,1.94199791){\color[rgb]{0,0,0}\makebox(0,0)[lt]{\lineheight{0}\smash{\begin{tabular}[t]{l}$t'$\end{tabular}}}}%
    \put(0,0){\includegraphics[width=\unitlength,page=2]{r=1_a=1_associativity-case-1.pdf}}%
    \put(0.62332801,0.85818925){\color[rgb]{0,0,0}\makebox(0,0)[lt]{\lineheight{0}\smash{\begin{tabular}[t]{l}$n-s-t-t'$\end{tabular}}}}%
    \put(0.13162118,0.85818925){\color[rgb]{0,0,0}\makebox(0,0)[lt]{\lineheight{0}\smash{\begin{tabular}[t]{l}$s$\end{tabular}}}}%
    \put(0.14150516,1.38625847){\color[rgb]{0,0,0}\makebox(0,0)[lt]{\lineheight{0}\smash{\begin{tabular}[t]{l}$s'$\end{tabular}}}}%
  \end{picture}%
\endgroup%

\end{figure}
\noindent
\eqref{eq:associativity-case-2} follows in the same way, but with $2T$ the tree-pair with trivial seam tree and the following bubble tree:
\begin{figure}[H]
\tiny
\centering
\def\svgwidth{0.6\columnwidth}
\begingroup%
  \makeatletter%
  \providecommand\color[2][]{%
    \errmessage{(Inkscape) Color is used for the text in Inkscape, but the package 'color.sty' is not loaded}%
    \renewcommand\color[2][]{}%
  }%
  \providecommand\transparent[1]{%
    \errmessage{(Inkscape) Transparency is used (non-zero) for the text in Inkscape, but the package 'transparent.sty' is not loaded}%
    \renewcommand\transparent[1]{}%
  }%
  \providecommand\rotatebox[2]{#2}%
  \newcommand*\fsize{\dimexpr\f@size pt\relax}%
  \newcommand*\lineheight[1]{\fontsize{\fsize}{#1\fsize}\selectfont}%
  \ifx\svgwidth\undefined%
    \setlength{\unitlength}{151.096845bp}%
    \ifx\svgscale\undefined%
      \relax%
    \else%
      \setlength{\unitlength}{\unitlength * \real{\svgscale}}%
    \fi%
  \else%
    \setlength{\unitlength}{\svgwidth}%
  \fi%
  \global\let\svgwidth\undefined%
  \global\let\svgscale\undefined%
  \makeatother%
  \begin{picture}(1,0.38319227)%
    \lineheight{1}%
    \setlength\tabcolsep{0pt}%
    \put(0,0){\includegraphics[width=\unitlength,page=1]{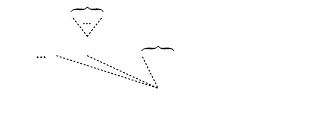}}%
    \put(0.49595314,0.25756154){\color[rgb]{0,0,0}\makebox(0,0)[lt]{\lineheight{0}\smash{\begin{tabular}[t]{l}$s'$\end{tabular}}}}%
    \put(0,0){\includegraphics[width=\unitlength,page=2]{r=1_a=1_associativity-case-2.pdf}}%
    \put(0.08100246,0.25756154){\color[rgb]{0,0,0}\makebox(0,0)[lt]{\lineheight{0}\smash{\begin{tabular}[t]{l}$s$\end{tabular}}}}%
    \put(0.26954459,0.37606664){\color[rgb]{0,0,0}\makebox(0,0)[lt]{\lineheight{0}\smash{\begin{tabular}[t]{l}$t$\end{tabular}}}}%
    \put(0.71555432,0.37606664){\color[rgb]{0,0,0}\makebox(0,0)[lt]{\lineheight{0}\smash{\begin{tabular}[t]{l}$t'$\end{tabular}}}}%
    \put(0.81714423,0.25756154){\color[rgb]{0,0,0}\makebox(0,0)[lt]{\lineheight{0}\smash{\begin{tabular}[t]{l}$n-s-t-s'-t'$\end{tabular}}}}%
  \end{picture}%
\endgroup%

\end{figure}
\end{enumerate}
\end{proof}

Note that this $A_\infty$-structure on the morphisms utilizes only the composition data for $r=1$, $a=1$.
Here $r=1$ reflects the fact that we restricted the structure to two fixed objects.
Then $a=1$ is related to the fact that $(A_\infty,2)$-flow categories will typically be compatible with fiber products for $r\leq 2$ in the sense of the following definition.
This is because the associahedra $K_r$ for $r\leq 2$ are points, so that fiber products over these base spaces $K_r$ are actually Cartesian products.
In such cases the composition data for $a\geq 2$ can be constructed by Cartesian products of the $a=1$ data.
In a general $(A_\infty,2)$-flow category, however, the definition above requires no relationship between the data for $a\geq 2$ and that for $a=1$.

\begin{definition}\label{def:compatible}
A regularized $(A_\infty,2)$-flow category $2\sC$ as in Definition~\ref{def:general A_infty-2-flow_category} is called \emph{compatible with fiber products for $r= r_c$} if 
for each choice of $a \geq 2$, $\underline\bn=(\bn^1, \ldots, \bn^a) \in \bigl(\bZ_{\geq0}^r\bigr)^a$, and collections of objects $M_0,\ldots,M_r\in\Ob$ and 1-morphisms as in \eqref{eq:curly-L}, 
\begin{equation} 
\cL \: = \:
\left(\begin{array}{c}
 \cL^1 = \left( L_{(i-1)i}^{1,k} \right)_{ 1\leq i\leq r , 0\leq k\leq n_i^1} 
\\ \vdots \\ 
\cL^a = \left( L_{(i-1)i}^{a,k} \right)_{ 1\leq i\leq r , 0\leq k\leq n_i^a}
\end{array}\right)
\end{equation}
the associated $(*)$moduli space and its $(*)$maps  are the fiber products 
\begin{equation}\label{eq:fiber-compatible}
\begin{array}{rlrl} 
\textstyle \cX(\cL) & = \: \prod^{K_r}_{1\leq j \leq a}  \cX(\cL^j) ,  
& \qquad \qquad 
\underline\alpha_\cL  &= \: \underline\alpha_{\cL^1}\times \ldots \times \underline\alpha_{\cL^a} \big|_{\cX(\cL)}  \\ 
p_\cL & = \: p_{\cL^1}\times \ldots \times p_{\cL^a} \big|_{\cX(\cL)}  
& \qquad \qquad 
\underline\beta_\cL  &= \: \underline\beta_{\cL^1}\times \ldots \times \underline\beta_{\cL^a} \big|_{\cX(\cL)}  .
\end{array} 
\end{equation}
Here the iterated fiber product $ \prod^{K_r}$ is defined with respect to the maps $p_{\cL^j}: \cX(\cL^j) \to W_{\bn^j}$ composed with the forgetful maps $\pi: W_{\bn^j} \to K_r$.
\end{definition}

The consequences of this compatibility will be discussed in the more algebraic context of \S\ref{ssec:2-extract}.
The next step towards making sense of the notion of $(A_\infty,2)$-flow category is to explain the geometric and combinatorial motivation for the boundary decomposition maps.

\begin{remark}
\label{rmk:three_types_of_comp_maps}
The three types of boundary decomposition maps in Definition~\ref{def:general A_infty-2-flow_category} arise from the codimension-1 degenerations of a witch curve in a 2-associahedron, or more generally a tuple of witch curves in a fiber product of 2-associahedra.
As in \cite{bottman_realizations}, witch curves are ``quilted spheres'' which arise by stereographic projection from a configuration of parallel lines -- called ``seams'' -- in the plane and any number of marked points on the seams.
Another marking -- the ``output marked point'' -- is added to the point on the sphere that represents infinity in the plane.
Thus the seams on the witch curves are circles that intersect tangentially at the output marked point, and each of these circles can carry input marked points.
The relative location of the seams induces maps from each 2-associahedron $W_\bn$ for $\bn \in \bZ_{\geq0}^r$ to the 1-associahedron $K_r$.
Thus the fiber product of 2-associahedra are made up of tuples of witch curves with the same relative location of their seams.

Mildly generalizing the terminology in \cite{bottman_2-associahedra} there are three types of codimension-1 degenerations of such a tuple of witch curves in a fiber product of 2-associahedra, depicted in Figure~\ref{fig:degens of fiber product}:  

\begin{enumerate}
\item
In a type-1 move, consecutive marked points on a single seam on a single sphere collide; 
a single bubble forms, carrying this seam and all the marked points involved in the collision.
This corresponds to the boundary decomposition maps $\varphi^{(1)}_{i,j,s,t} $.

\item
In a type-2 move, a proper subset of consecutive seams on all spheres collides.
Collections of marked points on these seams can collide simultaneously.
Bubbles form whenever marked points are colliding or are involved in the collision of seams.
This corresponds to the boundary decomposition maps $\varphi^{(2)}_{s,t,\underline\bm}$.

\item
In a type-3 move, marked points on a single sphere diverge to infinity -- which is equivalent to all seams on a single sphere colliding without changing the relative location, i.e.\ map to the 1-associahedron.
Collections of marked points on these seams can collide simultaneously.
Bubbles form whenever marked points are colliding or are involved in the collision of seams.
This corresponds to the boundary decomposition maps $\varphi^{(3)}_{j,\underline{\bm}}$.
\null\hfill$\triangle$
\end{enumerate}
\end{remark}

Analogous to the $A_\infty$ case, Theorem~\ref{thm:A-infinity_2_flow_to_linear} outlines how a regularized $(A_\infty,2)$-flow category in any regularization framework induces a linear $(A_\infty,2)$ category.
To give a precise yet accessible proof, we will work in the simplified $(\cC^0,\text{Morse})$-framework of Definition~\ref{def:C0-framework}, as specified in the following definition.
In applications, this framework can be achieved as discussed in Remark~\ref{rmk:cascades}.

\begin{definition}
\label{def:A_infty-2-flow_category}
A \emph{regularized $(A_\infty,2)$-flow category $2\sC$ 
in the $(\cC^0,\text{Morse})$-framework} of Definition~\ref{def:C0-framework} consists of:

\begin{itemize}
\item
A category $(\Ob, \leftindex^1\Mor)=(\Ob_{2\sC}, \leftindex^1\Mor_{2\sC})$.

\smallskip

\item
For every $L, L' \in  \leftindex^1\Mor$, a finite set $\leftindex^2\Mor(L,L')=\leftindex^2\Mor_{2\sC}(L,L')$.

\smallskip

\item
For every choice of integers $r \geq 1, a \geq 1$,
$\underline\bn=(\bn^1, \ldots, \bn^a) \in \bigl(\bZ_{\geq0}^r\bigr)^a$ 
and collections of objects $M_0,\ldots,M_r\in\Ob$ and 1-morphisms
$\cL = \Bigl(L_{(i-1)i}^{j,k} \in  \leftindex^1\Mor(M_{i-1},M_i) \Bigr)_{(i,j,k)\in I^{\, r,a}_{\underline\bn} }$
indexed as in \eqref{eq:curly-L} a $\cC^0$-manifold with boundary $\cX(\cL)$ that is equipped with
a locally constant energy function $\cE:\cX(\cL)\to\bR$ such that $\cE^{-1}((-\infty,E])$ is compact for all $E\in\bR$.
Moreover, $\cX(\cL)$ is equipped with $\cC^0$ evaluation maps
\begin{align}
\xymatrix{
&&& \leftindex^2\Mor(L_{(i-1)i}^{j,k-1},L_{(i-1)i}^{j,k}) 
\quad\text{for}\quad (i,j,k)\in {'\!I^{\,r,a}_{\underline\bn}} \\
\cX(\cL) \ar@/^1.0pc/[urrr]^(0.3){\alpha^{i,j,k}_\cL}\ar[rrr]^(0.3){\beta^j_\cL}\ar@/_1.0pc/[drrr]_(0.3){p_\cL} &&& \leftindex^2\Mor(L_{01}^{j,0}\circ\cdots\circ L_{(r-1)r}^{j,0},L_{01}^{j,n_1^j}\circ\cdots\circ L_{(r-1)r}^{j,n_r^j}) 
\quad\text{for}\quad 1 \leq j \leq a \\
&&& \prod^{K_r}_{1\leq j \leq a}  W_{\bn^j},
}
\end{align}
where $W_\bn$ is the $\bn$-th 2-associahedron as in \cite{bottman_realizations,bottman_oblomkov}, 
and $\prod^{K_r}_{1\leq j\leq a} W_{\bn^j}$ indicates the iterated fiber product with respect to the forgetful maps $W_{\bn^j}\to K_r$ on each of the $a$ factors.
 
\smallskip

\item
For any choice of integers $r \geq 1$, $a\geq 1$, $\bn \in (\bZ_{\geq0}^r)^a$ and collection of 1-morphisms $\cL$ as in \eqref{eq:curly-L}, continuous embeddings -- called \emph{boundary decomposition maps} -- of three types $\tau=1,2,3$ as in \eqref{eq:type1_composition}--\eqref{eq:type3_composition}
\begin{equation}
\varphi^{(\tau)}_{\cdots} \: \colon \quad
\cX(\cL') \:\leftindex_{\underline\alpha'}\times_{\underline\beta''}\: \cX(\cL'') \quad\longrightarrow\quad \cX(\cL)  
\end{equation}

\end{itemize}
We require that these boundary decomposition maps satisfy the following properties.
\begin{enumerate}
\item[\bf(i)]
Each boundary decomposition map covers the corresponding operadic composition map on fiber products of 2-associahedra as in \eqref{eq:2 cover}.

\smallskip

\item[\bf(ii)]
Each boundary decomposition map is compatible with the evaluation maps $\alpha^{i,j,k}_\cL$ and $\beta^j_\cL$ as in \eqref{eq:type-1-compatibility-with-alpha-beta}--\eqref{eq:type-3-compatibility-with-alpha-beta}, and makes the energy functions additive in the sense that 
\begin{equation} \label{eq:2-additive}
\cE\bigl( \varphi^{(\tau)}_{\cdots}(\underline\chi',\underline\chi'') \bigr) 
\ = \: 
\cE(\underline\chi') + \cE(\underline\chi'')   .
\end{equation}

\smallskip

\item[\bf(iii)]
For any collection of 1-morphisms $\cL$ as in \eqref{eq:curly-L}
the collection of boundary decomposition maps $\varphi^{(\tau)}_{\cdots}$ for $\tau=1,2,3$
is a system of boundary faces for $\cX(\cL)$ in the sense of Definition~\ref{def:mfds_with_faces}.

\item[\bf(iv)]
The boundary decomposition maps are associative in the sense that when we form finer decomposition maps $\varphi^{\ul{2T}}_\cL$ as in Remark~\ref{rmk:combinatorics_of_regularized_A-infty_2_categories}, the resulting map is independent of the order in which we composed boundary decomposition maps.

\end{enumerate}

Moreover, a regularized $(A_\infty,2)$-flow category $2\sC$ 
in the $(\cC^0,\text{Morse})$-framework is called \emph{compatible with fiber products for $r= r_c$} if it satisfies the conditions of Definition~\ref{def:compatible} and the energy is additive in the sense that for $(\chi^1,\ldots,\chi^a) \in \cX(\cL) =  \prod^{K_r}_{1\leq j \leq a}  \cX(\cL^j)$ we have
\begin{equation}\label{eq:fiber compatible energy}
\cE(\chi^1,\ldots,\chi^a)=\cE(\chi^1)+\ldots+\cE(\chi^a) .
\end{equation}  
\null\hfill$\triangle$
\end{definition}

\section{Extracting linear categorical structures from $A_\infty$- and $(A_\infty,2)$-flow categories}
\label{sec:extracting}

This section shows that $A_\infty$- and $(A_\infty,2)$-flow categories give rise to linear $A_\infty$- and $(A_\infty,2)$-categories by replacing base spaces with chain complexes and replacing moduli spaces with maps to base spaces with push-pull maps between the chain complexes.

\subsection{From $A_\infty$-flow categories to linear $A_\infty$-categories}
\label{sec:flow to cat}

We will start by reviewing the notion of a curved $\Lambda$-linear $A_\infty$-category $\sC_\lin$ for any coefficient ring (or field) $\Lambda$.
As explained in \cite[Rmk.\ 2.12]{auroux_beginners_guide}, this is a generalization of an ordinary (``flat'') $A_\infty$-category where we associate to every object $L$ an element $\mu_0=\mu_0(L)\in \Mor(L,L)$, and modify the $A_\infty$-relations by incorporating these elements into the sum.
\noindent
Here and throughout, we will avoid specifying signs by working with a coefficient ring (or field) $\Lambda$ of characteristic 2.

\begin{definition} \label{def:A_infty-category}
A \emph{$\Lambda$-linear $A_\infty$-category $\sC$}  consists of:
\begin{itemize}
\item
A set $\Ob=\Ob_{2\sC}$.

\smallskip

\item
For every $L, L' \in \Ob$, a $\Lambda$-module (or vector space over $\Lambda$) $\Mor(L,L')=\Mor_{2\sC}(L,L')$.

\smallskip

\item
For every $n \geq 0$ and $L^0, \ldots, L^n \in \Ob$, a  $\Lambda$-linear map
\begin{align} 
\mu_n \:: \:  \Mor(L^0,L^1)\otimes\cdots\otimes\Mor(L^{n-1},L^n) \;\to\; \Mor(L^0,L^n).
\end{align}
This includes algebraic curvature from $n=0$ for each $L^0 \in \Ob$ given by $\mu_0 \in \Mor(L^0,L^0)$.
\end{itemize}

\noindent
We require that these $\Lambda$-linear maps satisfy the (curved) $A_\infty$ equations for each $n\geq 0$, 
\begin{align}
\label{eq:A_infty-equations}
\textstyle
0 \:=\: \sum_{s+t\leq n}
\mu_{n-t+1}(x_1\smotimes\ldots\smotimes x_s \smotimes \mu_s(x_{s+1}\smotimes \ldots \smotimes x_{s+t}) \smotimes x_{s+t+1}\smotimes \ldots \smotimes x_n) 
\end{align}
for any choice of $x_k\in \Mor(L^{k-1},L^k)$ for $k=1,\ldots,n$.
This sum is over $s,t\geq 0$ with $s+t \leq n$, so that the terms with $t=0$ contribute
$\sum_{s=0}^n \mu_{n+1}(\ldots\smotimes x_s \smotimes\mu_0 \smotimes x_{s+1} \smotimes\ldots)$.
\end{definition}

The following proposition states how an $A_\infty$-flow category in any regularization framework gives rise to a linear $A_\infty$-category -- assuming the framework satisfies the requirements of Definition~\ref{def:framework}.
We then specify to the $(\cC^0,\text{Morse})$-framework of Definition~\ref{def:C0-framework} to prove that this construction satisfies the $A_\infty$-equations.
We expect this proof to generalize to other regularization frameworks -- using framework-specific analogues of the fact in Lemma~\ref{lem:system of boundary faces induces identity} that a system of boundary faces induces an algebraic identity between push-pull maps.

\begin{proposition}
\label{prop:A-infinity_flow_to_linear}
Suppose that $\sC$ is a regularized $A_\infty$-flow category in a framework $(*)$ 
as in Definition~\ref{def:general_A_infty-flow_category}, for example in the $(\cC^0,\text{Morse})$-framework as in Definition~\ref{def:A_infty-flow_category}.
Then it gives rise to a $\Lambda$-linear $A_\infty$-category $\sC_\lin$ as follows:
\begin{itemize}
\item
The set of objects is the same $\Ob_{\sC_\lin} \coloneqq \Ob_\sC$.

\smallskip

\item
For every $L,L'\in \Ob_{\sC_\lin}$, the morphism $\Lambda$-module 
$\Mor_{\sC_\lin}(L,L') \,\coloneqq\: C_* Mor_\sC(L,L')$
is the chain complex associated by the $(*)$-framework to the $(*)$base space $Mor_\sC(L,L')$.

In the $(\cC^0,\text{Morse})$-framework the morphism $\Lambda$-module
$\Mor_{\sC_\lin}(L,L') \,\coloneqq\: \bigoplus_{p\in\Mor_\sC(L,L')} \Lambda \, p $
is generated by the finite set $\Mor_\sC(L,L')$.

\smallskip

\item
For $n \geq 0$ and $\cL \coloneqq (L^0,\ldots,L^n)$ we define the  $n$-ary composition map by push-pull via the $(*)$moduli space $\cX(L^0,\ldots,L^n)$: 
\begin{align} \label{eq:push-pull}
\mu_n
\,\colon \:
\Mor_{\sC_\lin}(L^0,L^1)
\otimes
\cdots
\otimes
\Mor_{\sC_\lin}(L^{n-1},L^n)
& \:\to\:
\Mor_{\sC_\lin}(L^0,L^n)  \\
c_1 \smotimes \ldots \smotimes c_n & \:\mapsto\: {\beta_{\cL}}_* \bigl( ( \alpha^1_{\cL} \times \ldots \times \alpha^n_{\cL} ) ^* (c_1 \times \ldots \times c_n ) \bigr). \nonumber
\end{align}
In the $(\cC^0,\text{Morse})$-framework this push-pull map is given by $\Lambda$-linear extension of the map 
$\mu_n(p_1\smotimes\cdots\smotimes p_n) \coloneqq \sum_{q\in \Mor_\sC(L^0,L^n)} \langle p_1\smotimes\cdots\smotimes p_n,q\rangle \, q$ given on generators by counting the $0$-dimensional part $\cX(L^0 \ldots L^n)_0\subset \cX(L^0 \ldots L^n)$ with Novikov coefficients.
That is, we denote
\begin{equation}
\cX(p_1,\ldots,p_n,q; E)_0 
 \coloneqq  \cX(L^0 \ldots L^n)_0\cap (\alpha^1_{\cL},\ldots,\alpha^n_{\cL},\beta_{\cL})^{-1}\{(p_1,\ldots,p_n,q)\} \cap \cE^{-1}(E)
\end{equation}
for any $p_1\in \Mor_{\sC}(L^0,L^1) \ldots p_n\in \Mor_{\sC}(L^{n-1},L^n), q\in \Mor_{\sC}(L^0,L^n)$, $E\in\bR$, and set 
\begin{align} \label{eq:count}
\langle p_1\smotimes\cdots\smotimes p_n,q\rangle
& \: \coloneqq \: 
\#_\Lambda \bigl( \cX(L^0,\ldots,L^n)_0\cap (\alpha^1_{\cL},\ldots,\alpha^n_{\cL},\beta_{\cL})^{-1}\{(p_1,\ldots,p_n,q)\} \bigr) \\
& \: \coloneqq \: \textstyle \sum_{l=0}^\infty  \#_{\bZ_2} \cX(p_1,\ldots,p_n,q; E_l )_0 \:  T^{E_l} ,  
\nonumber
\end{align}
where $\cE( \cX(L^0,\ldots,L^n)_0\cap (\alpha^1_{\cL},\ldots,\alpha^n_{\cL},\beta_{\cL})^{-1}\{(p_1,\ldots,p_n,q)\} ) = \{E_0, E_1, \ldots \}$ is a discrete set with $E_l < E_{l+1}$ and finitely many elements or $\lim_{l\to\infty} E_l = \infty$ as in Remark~\ref{rmk:energies can be ordered}.
\end{itemize}
\end{proposition}

\begin{proof}
The proof will be given in the $(\cC^0,\text{Morse})$-framework, relying on the properties in Definitions~\ref{def:C0-framework} and \ref{def:A_infty-flow_category} of a regularized $A_\infty$-flow category in this framework.
We obtain well-defined  $\Lambda$-modules $\bigoplus_{p\in\Mor_\sC(L,L')} \Lambda \, p $ since each $\Mor_\sC(L,L')$ is assumed to be a finite set.
To verify that $\langle p_1\smotimes\cdots\smotimes p_n,q\rangle\in\Lambda$ are well-defined coefficients in the universal Novikov field \eqref{eq:novikov}, we use the fact that by Definition~\ref{def:C0-framework}~(i) -- after restriction to the $0$-dimensional part of the $(\cC^0,\text{Morse})$-moduli space -- the energy function $\cE:\cX(L^0,\ldots,L^n)_0\to\bR$ is locally constant, and $\cE^{-1}((-\infty,E])$ is compact for all $E\in\bR$.
Thus each $\cE^{-1}((-\infty,E])$ can contain only finitely many points, which have finitely many energy values $\cE\bigl(\cE^{-1}((-\infty,E])\bigr) = \{ E_0 < E_1 < \ldots < E_L \}$.
As we increase $E\to\infty$, the ordered list of energy values either ends with finitely many entries or continues with $\lim_{l\to\infty} E_l = \infty$.
Moreover, each $\cX(p_1,\ldots,p_n,q; E_l )_0=\cE^{-1}(E_l)$ is a finite number of points, so has a well-defined count $\#_{\bZ_2} \cX(p_1,\ldots,p_n,q; E_l )_0$ modulo 2.
Thus each coefficient $\langle p_1\smotimes\cdots\smotimes p_n,q\rangle \coloneqq \sum_{l=0}^\infty  \#_{\bZ_2} \cX(p_1,\ldots,p_n,q; E_l )_0 \:  T^{E_l}\in\Lambda$ in the $n$-ary composition map
is a well-defined expression in \eqref{eq:novikov}.

To verify the $A_\infty$ equations \eqref{eq:A_infty-equations} we will utilize Definition~\ref{def:A_infty-flow_category}~(iii), which ensures that the collection of boundary decomposition maps $\varphi^{s,t}_{L^0 \ldots L^n}$ for $s,t\geq 0$ with $s+t\leq n$ in \eqref{eq:A_infty-flow_recursive} forms a system of boundary faces for the $\cC^0$-manifold $\cX(L^0,\ldots,L^n)$ with boundary.
Spelling out Definition~\ref{def:mfds_with_faces}: 
\begin{enumerate}
\item[(a)]
Each boundary decomposition map is a $\cC^0$-embedding
$$
\varphi^{s,t}_{L^0 \ldots L^n}\colon
\cX(L^0,\ldots,L^s,L^{s+t},\ldots,L^n)
\:\leftindex_{\alpha^{s+1}}\times_\beta\:
\cX(L^s,\ldots,L^{s+t})
\to
\cX(L^0,\ldots,L^n)
$$
whose image lies in the boundary $\partial\cX(L^0,\ldots,L^n)$, and the image of the interior is an open subset of the boundary.
\item[(b)]
There is an open dense subset $O\subset\partial\cX(L^0,\ldots,L^n)$ so that each point of $O$ lies in the image of exactly one boundary decomposition map $\varphi^{s,t}_{L^0 \ldots L^n}$ (restricted to its interior).
\end{enumerate}
When restricting to components of fixed dimension as in Definition~\ref{def:mfds_with_faces}, property (a) implies that each boundary decomposition map restricts to $\cC^0$-embeddings for $m,m'\geq 0$ 
$$
\cX(L^0,\ldots,L^s,L^{s+t},\ldots,L^n)_m
\:\leftindex_{\alpha^{s+1}}\times_\beta\:
\cX(L^s,\ldots,L^{s+t})_{m'}
\to
\partial \cX(L^0,\ldots,L^n)_{m+m'+1}.
$$
Indeed, since the boundary decomposition maps are embeddings, the image of the interior under such a map is a submanifold of dimension $m+m'$.
By property (a) this image is an open subset of the boundary, hence it lies in the boundary of a component of dimension $m+m'+1$.
Now consider the 1-dimensional part $\cX(L^0,\ldots,L^n)_1\subset \cX(L^0,\ldots,L^n)$.
By definition, this is a 1-dimensional $\cC^0$-manifold with boundary, whose boundary is a $\cC^0$-manifold of dimension $0$, that is a disjoint union of points.
Now by property (b) each such boundary point $\chi\in\partial\cX(L^0,\ldots,L^n)_1$ lies in the image of a unique boundary decomposition map -- in fact, we have $\chi=\varphi^{s,t}_{L^0 \ldots L^n}(\chi',\chi'')$ for some 
$(\chi',\chi'')\in \cX(L^0,\ldots,L^s,L^{s+t},\ldots,L^n)_0 \:\leftindex_{\alpha^{s+1}}\times_\beta\: \cX(L^s,\ldots,L^{s+t})_0$.
Thus the collection of maps
$
\varphi^{s,t}_{L^0 \ldots L^n}|_{\cX(L^0 \ldots L^s,L^{s+t} \ldots L^n)_0
\:\leftindex_{\alpha^{s+1}}\times_\beta\: \cX(L^s,\ldots,L^{s+t})_0}
$
induces a bijection
\begin{align}
\label{eq:bijection}
\bigsqcup_{s+t\leq n}
\cX(L^0,\ldots,L^s,L^{s+t},\ldots,L^n)_0
\leftindex_{\alpha^{s+1}}\times_\beta
\cX(L^s,\ldots,L^{s+t})_0
\quad\overset{\sim}{\longrightarrow}\quad
\partial\cX(L^0,\ldots,L^n)_1.
\end{align}
Moreover, the energy is additive $\cE(\varphi^{s,t}_{L^0 \ldots L^n}(\chi',\chi''))=\cE(\chi')+\cE(\chi'')$ with respect to this bijection by Definition~\ref{def:A_infty-flow_category}~(ii).

Since the composition operations $\mu_n$ are $\Lambda$-linear, it suffices to check the $A_\infty$-relations  \eqref{eq:A_infty-equations} on generators.
For fixed $n\geq0$, $L^0,\ldots,L^n\in \Ob_{\sC_\lin}$ and $p_k \in \Mor_\sC(L^{k-1},L^k)$ for $1 \leq k \leq n$ we have
\begin{align*}
& \sum_{s+t\leq n}
\mu_{n-s+1}(p_1\smotimes \ldots\smotimes p_s\smotimes \mu_s(p_{s+1}\smotimes \ldots\smotimes p_{s+t})\smotimes p_{s+t+1}\smotimes \ldots,p_n) \\
&\;=\; 
\sum_{s+t\leq n} \mu_{n-s+1}\Bigl(p_1\smotimes \ldots,p_s\smotimes 
\Bigl( \sum_{p\in \Mor_\sC(L^s,L^{s+t})} 
\langle p_{s+1}\smotimes\ldots\smotimes p_{s+t} , p \rangle \,
 p \Bigr)\:  \smotimes p_{s+t+1}\smotimes \ldots\smotimes p_n\Bigr) \\
&\;=\; 
\sum_{s+t\leq n}
\sum_{q\in \Mor_\sC(L^0,L^n)}
\sum_{p\in \Mor_\sC(L^s,L^{s+t})} 
\langle p_{s+1}\smotimes\ldots\smotimes p_{s+t} , p \rangle \, 
\langle p_1\smotimes\cdots p_s \smotimes p \smotimes p_{s+t+1} \cdots \smotimes p_n,q\rangle  \, q .
\end{align*}
So \eqref{eq:A_infty-equations} is equivalent to the following identity for all $p_k \in \Mor_\sC(L^{k-1},L^k)$ and $q\in \Mor_\sC(L^0,L^n)$ with 
$\cL' \coloneqq  (L^0,\ldots L^s,L^{s+t},\ldots,L^n)$, $\cL'' \coloneqq  (L^s,\ldots,L^{s+t})$, and $\cL \coloneqq (L^0,\ldots,L^n)$ 
\begin{align*}
0
&\;=\; 
\sum_{s+t\leq n}
\sum_{p\in \Mor_\sC(L^s,L^{s+t})} 
\langle p_{s+1}\smotimes\ldots\smotimes p_{s+t} , p \rangle \, 
\langle p_1\smotimes\cdots p_s \smotimes p \smotimes p_{s+t+1} \cdots \smotimes p_n,q\rangle  \\
&\;=\; 
\sum_{s+t\leq n}
\sum_{p\in \Mor_\sC(L^s,L^{s+t})} 
\left(
\begin{aligned}
& \#_\Lambda \bigl(\cX(\cL'')_0\cap (\alpha^1_{\cL''} \ldots \alpha^s_{\cL''},\beta_{\cL''})^{-1}\{(p_{s+1},\ldots,p_{s+t},p)\} \bigr) \\
& \cdot \#_\Lambda  \bigl(\cX(\cL')_0\cap (\alpha^1_{\cL'} \ldots \alpha^{n-s+1}_{\cL'},\beta_{\cL'})^{-1}\{(p_1,\ldots,p_s,p,p_{s+t+1},\ldots,p_n,q)\} \bigr)  
\end{aligned}
\right)
\\
&\;=\; 
\sum_{s+t\leq n}
\sum_{p\in \Mor_\sC(L^s,L^{s+t})} 
\#_\Lambda  \left\{  (\chi',\chi'')
\,\left| \;
\begin{aligned}
&  \chi' \in \cX(\cL')_0,  \chi'' \in \cX(\cL'')_0, \\
& \alpha^1_{\cL''}(\chi'')=p_{s+1},\ldots, \alpha^s_{\cL''}(\chi'')=p_{s+t}, \beta_{\cL''}(\chi'')=p  \\
& \alpha^1_{\cL'}(\chi')=p_1,\ldots, \alpha^s_{\cL'}(\chi')=p_s, \alpha^{s+1}_{\cL'}(\chi')=p,  \\
& \alpha^{s+2}_{\cL'}(\chi')=p_{s+t+1}, \ldots,
\alpha^{n-s+1}_{\cL'}(\chi')=p_n, \beta_{\cL'}(\chi')=q 
\end{aligned}
\right.
\right\}  \\
&\;=\; 
\sum_{s+t\leq n}
\#_\Lambda  \left\{  (\chi',\chi'') 
\,\left| \;
\begin{aligned}
&  (\chi',\chi'')\in \cX(\cL')_0\leftindex_{\alpha^{s+1}}\times_\beta \cX(\cL'')_0  \\
& \alpha^1_{\cL'}(\chi')=p_1,\ldots, \alpha^s_{\cL'}(\chi')=p_s, \\
& \alpha^1_{\cL''}(\chi'')=p_{s+1},\ldots, \alpha^s_{\cL''}(\chi'')=p_{s+t}, \\
& \alpha^{s+2}_{\cL'}(\chi')=p_{s+t+1}, \ldots,
\alpha^{n-s+1}_{\cL'}(\chi')=p_n, \beta_{\cL'}(\chi')=q 
\end{aligned}
\right.
\right\}  \\
&\;=\; 
\sum_{s+t\leq n}
\#_\Lambda  \left\{  \chi \in  \im \varphi^{s,t}_{\cL} \cap \partial \cX(\cL)_1 
\,\left| \; 
 \alpha^1_{\cL}(\chi)=p_1,\ldots, \alpha^{n}_{\cL}(\chi)=p_n, \beta_{\cL}(\chi)=q 
\right.
\right\} \\
&\;=\; 
 \#_\Lambda  \; \partial \bigl(\cX(\cL)_1\cap (\alpha^1_{\cL},\ldots,\alpha^n_{\cL},\beta_{\cL})^{-1}\{(p_1,\ldots,p_n,q)\} \bigr).
\end{align*}
Here the second step (for fixed $s,t$ and $p$) is of the form 
\begin{equation}\label{eq:Novikov sum}
 \#_\Lambda \bigl\{ \chi' \in \cX(\cL'')_0 \,\big|\, \ldots \bigr\} 
 \cdot \#_\Lambda  \bigl\{\chi''\in\cX(\cL')_0  \,\big|\, \ldots \bigr\} 
\;=\; 
\#_\Lambda  \bigl\{  (\chi',\chi'') \in \cX(\cL'')_0\times \cX(\cL')_0 \,\big|\, \ldots \bigr\}  ,
\end{equation}
where we define the Novikov count on the right hand side by setting $\cE(\chi',\chi'') \coloneqq \cE(\chi')+\cE(\chi'')$.
Note that we sum over the same set of pairs $(\chi',\chi'')$ on both sides.
If this set is finite, then the identity holds by multiplying out finite sums 
$$
\textstyle \bigl( \sum_{\chi'}  T^{\cE(\chi')} \bigr) \cdot \bigl( \sum_{\chi''}  T^{\cE(\chi'')} \bigr) 
=
\sum_{\chi'}  \sum_{\chi''}  T^{\cE(\chi')+\cE(\chi'')} 
=
\sum_{(\chi',\chi'')}  T^{\cE(\chi',\chi'')} .
$$
If the set is infinite, then the same identity holds when we restrict both sides to the set of pairs $(\chi',\chi'')$ with $\cE(\chi')+\cE(\chi'')\leq E$ -- which we claim to be finite for any $E\in\bR$.
Then any pair that contributes to the infinite sum does so for some finite $E$, so the limit proves the identity in general, if we can confirm that such sets of pairs with bounded energy $\leq E$ are indeed finite.
The latter will be deduced from the properties of the energy function in Definition~\ref{def:C0-framework}: 
The energy on each moduli space $\cX(\cL')$ and $\cX(\cL'')$ is locally constant -- thus constant on connected components.
Moreover, each sublevel set is compact, so in particular has finitely many components, which ensures that the energy functions are bounded below, $\cE|_{\cX(\cL')}\geq - E'$ and $\cE|_{\cX(\cL'')}\geq - E''$.
Now the set in question is finite since it is a subset of the Cartesian product of finite sets
$$
\bigl\{ (\chi',\chi'') \,\big|\, \cE(\chi')+\cE(\chi'')\leq E \bigr\}   \subset 
\bigl\{ \chi' \,\big|\, \cE(\chi')\leq E+E'' \bigr\} \times \bigl\{ \chi'' \,\big|\, \cE(\chi'')\leq E+E' \bigr\}  .
$$
In the third step we include a sum over the finite set $\Mor_\sC(L^s,L^{s+t})$ in the Novikov counts $\#_\Lambda$, and the fourth step uses the energy additivity $\cE(\varphi^{s,t}_{\cL}(\chi',\chi''))=\cE(\chi')+\cE(\chi'')=\cE(\chi',\chi'')$.
The last two steps also use the bijection \eqref{eq:bijection} and its compatibility with the maps in 
\eqref{eq:compatibility-with-alpha-beta}.
Thus the $A_\infty$-relations \eqref{eq:A_infty-equations} are equivalent to
$ \#_\Lambda  \; \partial \bigl(\cX(\cL)_1\cap (\alpha^1_{\cL},\ldots,\alpha^n_{\cL},\beta_{\cL})^{-1}\{(p_1,\ldots,p_n,q)\} \bigr) = 0$.
This identity holds by Lemma~\ref{lem:system of boundary faces induces identity} applied to the system of boundary faces in \eqref{eq:bijection} as follows: 
The target spaces of the $\cC^0$-maps $\alpha^k_\cL$ and $\beta_\cL$ in \eqref{eq:a-b-maps} are discrete.
Hence each space $\cX(\cL)_1\cap (\alpha^1_{\cL},\ldots,\alpha^n_{\cL},\beta_{\cL})^{-1}\{(p_1,\ldots,p_n,q)\}$
is a union of connected components of the $1$-dimensional $\cC^0$-manifold $\cX(\cL)_1$.
It has a system of boundary faces given by restricting the bijection \eqref{eq:compatibility-with-alpha-beta} to the subset of pairs in the fiber product
$(\chi',\chi'')\in \cX(\cL')_0\leftindex_{\alpha^{s+1}}\times_\beta \cX(\cL'')_0$ which map to $p_1,\ldots,p_n,q$ by the evaluation maps as specified above.
Thus Lemma~\ref{lem:system of boundary faces induces identity} applies, and the Novikov identity \eqref{eq:boundary Novikov identity} amounts to
$$
\sum_{s+t\leq n}
\#_\Lambda  \left\{  (\chi',\chi'') 
\,\left| \;
\begin{aligned}
&  (\chi',\chi'')\in \cX(\cL')_0\leftindex_{\alpha^{s+1}}\times_\beta \cX(\cL'')_0  \\
& \alpha^1_{\cL'}(\chi')=p_1,\ldots, \beta_{\cL'}(\chi')=q 
\end{aligned}
\right.
\right\}
\: = \: 0 , 
$$
which proves the $A_\infty$-relations \eqref{eq:A_infty-equations}.
(The final two steps in the long computation above spell out the steps of the proof of Lemma~\ref{lem:system of boundary faces induces identity} in this setting.) 
\end{proof}

\begin{remark} \label{rmk:characteristic and regularization}
\begin{enumerate}
\item[(i)]
Lemma~\ref{lem:system of boundary faces induces identity} relies on the coefficient ring (or field) $\Lambda$ having characteristic 2.
We could drop this hypothesis at the expense of introducing signs into the $A_\infty$-equations, and would need to assume that the morphism spaces in the $A_\infty$-flow category carry orientations that are taken into account in the counting \eqref{eq:count}
and are compatible with the bijection \eqref{eq:bijection}.
\item[(ii)]
When constructing algebraic structures by counting pseudoholomorphic objects, a key step is to regularize the relevant moduli spaces.
This step is invisible in the proof of Proposition~\ref{prop:A-infinity_flow_to_linear}.
Indeed, when constructing e.g.\ a Fukaya $A_\infty$-category, the regularization step is part of constructing a regularized $A_\infty$-flow category in an appropriate regularization framework.
Once this is done, extracting a linear $A_\infty$-category is essentially a formal procedure.
\null\hfill$\triangle$
\end{enumerate}
\end{remark}

\subsection{From $(A_\infty,2)$-flow categories to linear $(A_\infty,2)$-categories}
\label{ssec:2-extract}

As discussed in the introduction, the reason that defining a linear version of $(A_\infty,2)$-categories is difficult is that the boundary faces of 2-associahedra involve fiber products of 2-associahedra.
This is in contrast to the situation with $A_\infty$-categories: 
The relatively straight-forward notion of linear $A_\infty$-categories arises naturally from the fact that
the boundary faces of associahedra are Cartesian products of associahedra.
In this subsection, we develop an alternative definition of a linear $(A_\infty,2)$-category.
The key idea is to replace ``2-associahedra'' with ``fiber products of 2-associahedra'' 
and associate separate operations to each fiber product of 2-associahedra.
Then we expect these operations to satisfy an $(A_\infty,2)$-equation that mirrors the combinatorics of the boundary decomposition maps of an $(A_\infty,2)$-flow category summarized in Remark~\ref{rmk:2-a-b-maps}.
Here we again avoid specifying signs by working with a coefficient ring (or field) $\Lambda$ of characteristic 2.

\begin{definition} \label{def:A_infty-2-category}
A \emph{$\Lambda$-linear $(A_\infty,2)$ category $2\sC$} consists of:
\begin{itemize}
\item
A category $(\Ob, \leftindex^1\Mor)=(\Ob_{2\sC}, \leftindex^1\Mor_{2\sC})$.

\smallskip

\item
For every $M_0, M_1 \in \Ob$ and $L_{01}, L_{01}' \in \leftindex^1\Mor(M_0,M_1)$, a $\Lambda$-module (or vector space over~$\Lambda$) $\leftindex^2\Mor_{2\sC}(L_{01},L_{01}')$.

\smallskip

\item
For any choice of integers $r \geq 1, a \geq 1$, tuples of integers 
$\underline\bn={\left( \small \begin{array}{c} \bn^1\\ \vdots \\ \bn^a \end{array} \right)} \in \bigl(\bZ_{\geq0}^r\bigr)^a$, 
and collections of objects $M_0,\ldots,M_r\in\Ob$ and 1-morphisms
\begin{align}  \label{eq:2-collection}
\cL & = 
\left(\begin{array}{c}
 \cL^1 \coloneqq  \left(  L_{(i-1)i}^{1,k} \right)_{ 1\leq i\leq r , 0\leq k\leq n_i^1} 
\\ \vdots \\ 
 \cL^a \coloneqq  \left( L_{(i-1)i}^{a,k} \right)_{ 1\leq i\leq r , 0\leq k\leq n_i^a}
\end{array}\right)
\: = \: \Bigl(L_{(i-1)i}^{j,k} \in  \leftindex^1\Mor(M_{i-1},M_i) \Bigr)_{(i,j,k)\in I^{\, r,a}_{\underline\bn} }  \\
& \qquad\qquad\qquad \text{indexed by}  \qquad
 I^{\, r,a}_{\underline\bn}  \coloneqq  \{ (i,j,k) \,|\,  1\leq i\leq r , 1\leq j\leq a, 0\leq k\leq n_i^j  \} 
\nonumber
\end{align}
a $\Lambda$-linear map
\begin{align} \label{eq:2-mu-r}
\bmu^{r,a}_{\underline\bn} \: : \: \bigotimes_{(i,j,k)\in {'\!I^{\, r,a}_{\underline\bn}}}
\leftindex^2\Mor(L_{(i-1)i}^{j,k-1},L_{(i-1)i}^{j,k})
\:\to \:
\bigotimes_{1\leq j \leq a} 
\leftindex^2\Mor(L_{01}^{j,0}\circ\cdots  L_{(r-1)r}^{j,0},L_{01}^{j,n_1^j}\circ\cdots L_{(r-1)r}^{j,n_r^j}).
\end{align}
Here pairs of consecutive 1-morphisms are indexed by 
$$
'\!I^{\, r,a}_{\underline\bn}  \coloneqq  \{ (i,j,k) \,|\,  1\leq i\leq r , 1\leq j\leq a, 1\leq k\leq n_i^j  \}.
$$
\end{itemize}

\noindent
We require that these $\Lambda$-linear maps satisfy the following \emph{$(A_\infty,2)$-equations}: 
For any $r \geq 1, a\geq 1$, $\underline\bn\in \bigl(\bZ_{\geq0}^r\bigr)^a$, 
collections of objects and 1-morphisms $\cL$ as in \eqref{eq:2-collection}, 
and every tensor product of 2-morphisms $x_i^{j,k} \in \leftindex^2\Mor\bigl(L_{(i-1)i}^{j,k-1},L_{(i-1)i}^{j,k}\bigr)$
\begin{align}
\label{eq:tuple of 2-morphisms}
\bX \, 
\: = \:
\bigotimes
\left(\begin{array}{c}
 \bX^1 \coloneqq \bigotimes_{ 1\leq i\leq r , 1\leq k\leq n_i^1} \; x_{i}^{1,k}  
\\ \vdots \\ 
 \bX^a \coloneqq \bigotimes_{ 1\leq i\leq r , 1\leq k\leq n_i^a} \;  x_{i}^{a,k} 
\end{array}\right)
\: = \:
\bigotimes_{(i,j,k)\in {'\!I^{\, r,a}_{\underline\bn}}}  x_i^{j,k} 
\end{align}
utilizing the notation of Remark~\ref{rmk:2-arrangements and notation}, 
we require that the corresponding \emph{$(A_\infty,2)$-equation} is satisfied: 
\begin{align}
\nonumber
0\: = \:
&
\sum_{*_1} \bmu_{\underline\bn_{*_1}}^{r,a}
\bigotimes \left(
\begin{array}{c}
\bX^{[1,j-1]}  \\
\left( \bX^j_{[1,i-1]} \otimes \bigotimes 
\left(
\begin{array}{c}
 \bX_i^{j,[1,s]} \\
\bmu_{(t)}^{1,1}
\left(\bX_i^{j,[s+1,s+t]} \right) \\
\bX_i^{j,[s+t+1,n_i^j]} \\
\end{array}
\right) 
\otimes \bX^j_{[i+1,r]} \right) \\
\bX^{[j+1,a]}  
\end{array}\right)
\\
&
+\sum_{*_2}
\bmu_{\underline\bn_{*_2} }^{r-t+1,a}
\left(
\bX_{[1,s]} \otimes 
\bmu_{\underline\bm_{*_2}}^{t,b^1+\ldots b^a}
\bigl(\bX_{[s+1,s+t]} \bigr)
\otimes \bX_{[s+t+1,r]} 
\right)
\label{eq:A_infty-2-equations}
\: + \: 
\sum_{*_3} \left(
\begin{array}{c} \tiny
\left.
\begin{array}{c} 
\vdots \\
\id   
\end{array}
\right\} \scriptstyle  j-1  \\
\bmu^{1,1}_{(b)} \phantom{\int_{bot}^{top}}   \\ 
\tiny 
\scriptstyle \left.
\begin{array}{c} 
\id  \\ 
\vdots
\end{array}
\right\}  \scriptstyle a-j \\
\end{array}
\right)
\left( 
\bmu^{r,a+b-1}_{\underline\bn_{*_3}} 
\left(
\bX
\right)
\right) 
\end{align}

Here the symbols $*_1, *_2, *_3$ indicate that we sum over the following:
\begin{equation} \label{eq:123}
\begin{array}{l}
*_1: \: \text{integers} \; 
1\leq i \leq r, 
1 \leq j \leq a, \, \text{and}\; 
s,t\geq 0 \:
\, \text{so that}\; s+t \leq n_i^j ;  \\
*_2: \: 
\text{in case} \; r\geq 3 \;  \text{we sum over integers} \; 
s\geq 0 \; \text{and}\;  2\leq t \leq r-1 \; \text{with}\; s+t\leq r, \\
\qquad\qquad 
\;\text{and for each}\; 1\leq j \leq a \; \text{partitions} \\
\qquad\qquad 
(n^j_{s+1},\ldots,n^j_{s+t})  
= \bm^{j,1} + \ldots + \bm^{j,b^j}  \;\text{into}\;  \bm^{j,1},\ldots,\bm^{j,b^j}  \in
\bZ_{\geq0}^t
\;\text{for some}\;  b^j\geq 0 ; 
\\
*_3: \: 
\text{in case} \; r\geq 2 \;  \text{we sum over integers} \; 1 \leq j \leq a  \; \text{and partitions} \\ 
\qquad\qquad \bn^j = \bm^1+\cdots+\bm^b \;\text{into}\; \bm^1,\ldots,\bm^b \in \bZ_{\geq0}^r \; \text{for some} \; b\geq 0 .
\end{array}
\end{equation}
Moreover, we denote
\begin{equation}  \label{eq:n13} \tiny  
\underline\bn_{*_1} \: \coloneqq \: 
\left( \begin{array}{c} 
\vdots \\  \bn^{j-1} \\ (\ldots,n_{i-1}^j, n_i^j-t+1,n_{i+1}^j,\ldots) \\ \bn^{j+1} \\ \vdots 
\end{array} \right), 
\qquad
\underline\bn_{*_3} \: \coloneqq \: 
\left( \begin{array}{c} 
\vdots \\ \bn^{j-1} \\ \bm^1 \\ \vdots \\ \bm^b \\ \bn^{j+1} \\ \vdots 
\end{array} \right), 
\end{equation}
\begin{equation} \label{eq:nm2} \tiny 
\underline\bn_{*_2} \: \coloneqq \: 
\left( \begin{array}{c} 
(\ldots,n_s^1,b^1,n_{s+t+1}^1,\ldots) \\ \vdots \\ (\ldots,n_s^a,b^a,n_{s+t+1}^a,\ldots)
\end{array} \right), 
\qquad
\underline\bm_{*_2} \: \coloneqq \: 
\left( \begin{array}{c} 
 \bm^{1,1} \\ \vdots \\ \bm^{1,b^1} \\ \vdots \\ \bm^{a,1} \\ \vdots \\ \bm^{a,b^a} 
\end{array} \right).
\end{equation}
\null\hfill$\triangle$
\end{definition}

\begin{remark} \label{rmk:2-arrangements and notation}
The $(A_\infty,2)$-structure in the above definition is -- analogously to the $A_\infty$-structure in Definition~\ref{def:A_infty-category} -- made up of a countable collection of $\Lambda$-linear maps \eqref{eq:2-mu-r} that satisfies a countable collection of $(A_\infty,2)$-equations.
Its higher categorical complexity is better understood by arranging the inputs of each map $\bmu^{r,a}_{\underline\bn}$ in two dimensions: 
\begin{align} 
{\tiny
\begin{array}{c}
\leftindex^2\Mor(L_{01}^{1,0},L_{01}^{1,1}) \\
\otimes \\
\vdots \\
\otimes \\
\leftindex^2\Mor(L_{01}^{1,n_1^1-1},L_{01}^{1,n_1^1}) \\
\otimes \\
\vdots \\
\otimes \\
\leftindex^2\Mor(L_{01}^{a,0},L_{01}^{a,1}) \\
\otimes \\
\vdots \\
\otimes \\
\leftindex^2\Mor(L_{01}^{a,n_1^a-1},L_{01}^{a,n_1^a})
\end{array}
\otimes
\cdots
\otimes
\begin{array}{c}
\leftindex^2\Mor(L_{(r-1)r}^{1,0},L_{(r-1)r}^{1,1}) \\
\otimes \\
\vdots \\
\otimes \\
\leftindex^2\Mor(L_{(r-1)r}^{1,n_r^1-1},L_{(r-1)r}^{1,n_r^1}) \\
\otimes \\
\vdots \\
\otimes \\
\leftindex^2\Mor(L_{(r-1)r}^{a,0},L_{(r-1)r}^{a,1}) \\
\otimes \\
\vdots \\
\otimes \\
\leftindex^2\Mor(L_{(r-1)r}^{a,n_r^a-1},L_{(r-1)r}^{a,n_r^a})
\end{array}
}
\sr{\bmu^{r,a}_{\underline\bn}}{\to}
{\tiny \begin{array}{c}
\leftindex^2\Mor(L_{01}^{1,0}\circ\cdots\circ L_{(r-1)r}^{1,0},L_{01}^{1,n_1^1}\circ\cdots\circ L_{(r-1)r}^{1,n_r^1}) \\
\otimes \\
\vdots \\
\otimes \\
\leftindex^2\Mor(L_{01}^{a,0}\circ\cdots\circ L_{(r-1)r}^{a,0},L_{01}^{a,n_1^a}\circ\cdots\circ L_{(r-1)r}^{a,n_r^a}).
\end{array}}
\end{align}
\noindent
The input for this linear map is a tensor product of 2-morphisms 
$x_i^{j,k} \in \leftindex^2\Mor\bigl(L_{(i-1)i}^{j,k-1},L_{(i-1)i}^{j,k}\bigr)$
which we abbreviate and group into various forms as follows: 
\begin{align}
\bX \, \coloneqq  \;  \bigotimes_{\substack{1\leq i\leq r,1\leq j\leq a \\1\leq k\leq n_i^j}}
x_i^{j,k} 
\; = \; 
\bigotimes
\left(
\begin{array}{c}
\bX^1  \\
\vdots \\
\bX^a  \\
\end{array}\right)
\end{align}
can be viewed as a tensor product of $a$ blocks\footnote{Each block represents the marked points on one of the $a$ quilted spheres in the fiber product.}
\begin{align}
\bX^j
\: \coloneqq  \:
\bigotimes_{1\leq i\leq r, \, 1\leq k\leq n_i^j}  x_i^{j,k} 
\: =  \:
\bigotimes
\left(
\begin{array}{ccc}
x_1^{j,1} & & x_r^{j,1} \\
\vdots & \cdots & \vdots \\
x_1^{j,n_1^j} & & x_r^{j,n_r^j}
\end{array}\right).
\end{align}
To compactify the notation in the $(A_\infty,2)$-equations \eqref{eq:A_infty-2-equations}, we denote the tensor product of blocks $j_1\leq j \leq j_2$ by 
\begin{align}
\bX^{[j_1,j_2]} \, \coloneqq  \; \bigotimes_{\substack{1\leq i\leq r,j_1\leq j\leq j_2 \\1\leq k\leq n_i^j}} 
x_i^{j,k} 
\; = \; 
\bigotimes
\left(
\begin{array}{c}
\bX^{j_1}  \\
\vdots \\
\bX^{j_2}  \\
\end{array}\right).
\end{align}
We can further view each block as the tensor product of column blocks, given for $1\leq j \leq a$ and $1\leq i \leq r$  by 
\begin{align}
\bX^j_i
\, \coloneqq  \;
\bigotimes_{1\leq k\leq n_i^j}  x_i^{j,k} 
\; = \; 
\bigotimes
\left(
\begin{array}{c}
x_i^{j,1}  \\
\vdots  \\
x_i^{j,n_i^j} \\
\end{array} \right)
\end{align}
so that we can view each block and as well as tensor products of column blocks $i_1\leq i \leq i_2$ as
\begin{align}
\bX^j
\:=\: 
\bX^j_1 \otimes \ldots  \otimes \bX^j_r 
\qquad
\text{and}
\qquad
\bX^j_{[i_1,i_2]}
\: \coloneqq \:
\bigotimes_{\substack{i_1\leq i\leq i_2 ,1\leq k\leq n_i^j}}  x_i^{j,k}
\: = \:
\bX^j_{i_1} \otimes \ldots \otimes \bX^j_{i_2} .
\end{align}
Alternatively, the tuple $\bigl(x_i^{j,k}\bigr)_{(i,j,k)\in {'\!I^{\, r,a}_{\underline\bn}}} $ decomposes naturally into $r$ columns\footnote{Each column represents the marked points on one of the $r$ seams.
Thus the columns may have different heights.} given by fixing $1\leq i \leq r$.
We denote their tensor products by 
\begin{align}
\bX_i
\, \coloneqq  \;
\bigotimes_{1\leq j\leq a, \, 1\leq k\leq n_i^j}  x_i^{j,k} 
\; = \; 
\bigotimes
\left(
\begin{array}{c}
\bX^1_i \\
 \vdots  \\
\bX^a_i 
\end{array}\right), 
\end{align}
so that we can express the full tensor product and the tensor product of columns $i_1\leq i \leq i_2$ as
\begin{equation}
\bX = \bX_1 \otimes \ldots \otimes\bX_r  , 
\qquad
\text{and}
\qquad
\bX_{[i_1,i_2]}
\: \coloneqq \:
\bX_{i_1} \otimes \ldots \otimes \bX_{i_2} 
\: = \:
\bigotimes_{\substack{i_1\leq i\leq i_2 , 1\leq j \leq a \\ 1\leq k\leq n_i^j}}
x_i^{j,k}  .
\end{equation}
Finally, we denote tensor products of 2-morphisms for $k_1\leq k \leq k_2$ within a column block by
\begin{align}
\bX_i^{j,[k_1,k_2]}
\, & \coloneqq  \;
\bigotimes_{k_1\leq k\leq k_2 } x_i^{j,k}
\; = \; 
\bigotimes
\left(
\begin{array}{c}
x^{j,k_1}_i \\
 \vdots  \\
x^{j,k_2}_i
\end{array}\right)  .
\nonumber
\end{align}
With this notation in place, the partitions of integer tuples in \eqref{eq:n13} and \eqref{eq:nm2} uniquely determine the partitions of the inputs in the following experessions in \eqref{eq:A_infty-2-equations}:  
\begin{align}
\bmu^{r,a+b-1}_{\underline\bn_{*_3}} 
\left( \bX \right)
\:=\: 
\bmu^{r,a+b-1}_{\underline\bn_{*_3}} 
\bigotimes
\left(
\begin{array}{c}
\bX^{[1,j-1]}  \\
\bX_1^{j,[1,m^{j,1}_1]} \otimes\ldots\otimes \bX_r^{j,[1,m^{j,1}_r]}  \\
\bX_1^{j,[m^{j,1}_1+1,m^{j,2}_1]} \otimes\ldots\otimes\bX_r^{j,[m^{j,1}_r+1,m^{j,2}_r]} \\
\vdots \\
\bX_1^{r,[m^{r,1}_1+\ldots m^{r,b-1}_1 +1 , n^j_1]} \otimes\ldots\otimes \bX_r^{j,[m^{j,1}_r+\ldots m^{1,b-1}_r +1,n^j_r]}  \\
\bX^{[j+1,a]} 
\end{array}
\right) , 
\end{align}
and similarly -- with a partition in each block -- 
\begin{align}
\bmu_{\underline\bm_{*_2}}^{t,b^1+\ldots b^a}
\bigl(\bX_{[s+1,s+t]} \bigr)
\:=\: 
\bmu_{\underline\bm_{*_2}}^{t,b^1+\ldots b^a}
\bigotimes
\left(
\begin{array}{c}
\bX_{s+1}^{1,[1,m^{1,1}_{s+1}]} \otimes\ldots\otimes \bX_{s+t}^{1,[1,m^{1,1}_{s+t}]}  \\
\vdots \\
\bX_{s+1}^{1,[m^{1,1}_{s+1}+\ldots m^{1,b^1-1}_{s+1} +1,n^1_{s+1}]} \otimes\ldots\otimes\bX_{s+t}^{1,[m^{1,1}_{s+t}+\ldots m^{1,b^1-1}_{s+t} +1,n^1_{s+t}]} \\
\vdots \\
\bX_{s+1}^{a,[1,m^{a,1}_{s+1}]} \otimes\ldots\otimes \bX_{s+t}^{a,[1,m^{a,1}_{s+t}]}  \\
\vdots \\
\bX_{s+1}^{a,[m^{a,1}_{s+1}+\ldots m^{1,b^a-1}_{s+1} +1,n^a_{s+1}]} \otimes\ldots\otimes\bX_{s+t}^{a,[m^{a,1}_{s+t}+\ldots m^{1,b^a-1}_{s+t} +1,n^a_{s+t}]} 
\end{array}
\right).
\end{align}
\null\hfill$\triangle$
\end{remark}

Next we explain the indexing conventions \eqref{eq:123} for the $(A_\infty,2)$-equations by the geometry and combinatorics of the 2-associahedra.

\begin{remark}
The three sums in the $(A_\infty,2)$-equation \eqref{eq:A_infty-2-equations} arise from the three types of codimension-1 degenerations in fiber products of 2-associahedra as in Remark~\ref{rmk:three_types_of_comp_maps} and depicted in Figure~\ref{fig:degens of fiber product}: 
\begin{enumerate}
\item
The first sum in \eqref{eq:A_infty-2-equations} corresponds to a type-1 move in which 
a consecutive sequence of $t$ points (the $(s+1)$-st through the $(s+t)$-th) on the $i$-th seam of the $j$-th witch ball collide.

\medskip

\item
The second sum corresponds to a type-2 move in which a consecutive sequence of $t$ seams (the 
$(s+1)$-st through $(s+t)$-th) on each witch ball collide.
On the $j$-th witch ball, $b^j\geq0$ bubbles form on the resulting fused seam.
These bubbles each carry $t$ seams, and the vector $\bm^{j,\ell}$ records the number of marked points on the $\ell$-th bubble.
The inequalities $2\leq t$ and $t\leq r-1$
correspond to the requirements that at least two seams but not not all of the seams collide.

\medskip

\item
The third sum corresponds to a type-3 move in which all the seams on the $j$-th witch ball collide.
On this witch ball, $b\geq0$ bubbles form on the fused seam.
These bubbles carry $r$ seams, and the vector $\bm^\ell$ records the number of marked points on the $\ell$-th bubble.
\null\hfill$\triangle$
\end{enumerate}
\end{remark}

As in Lemma~\ref{lem:mor is cat} we note that the morphisms $\Mor(M_0,M_1)$ between two fixed objects of a linear $(A_\infty,2)$-category have the structure of a linear $A_\infty$-category.

\begin{lemma} \label{lem:mor is A-infty cat}
Let $2\sC$ be a linear $(A_\infty,2)$-category as in Definition~\ref{def:general A_infty-2-flow_category}.
Then for any two objects $M_0, M_1 \in \Ob$ the structure of $2\sC$ restricts to a  
linear $A_\infty$-category $\Mor(M_0,M_1)$ as in Definition~\ref{def:general_A_infty-flow_category} as follows:  
\begin{itemize}
\item
Objects are given by the 1-morphisms $\Ob_{\Mor(M_0,M_1)} \coloneqq \leftindex^1\Mor_{2\sC}(M_0,M_1)$.

\smallskip

\item
$\Mor(L, L') \coloneqq \Mor_{2\sC}(L, L')$ forms a $\Lambda$-module (or vector space over $\Lambda$) for each pair of objects $L, L' \in \Ob_{\Mor(M_0,M_1)}$.

\smallskip

\item
For every $n \geq 0$ and $L^0, \ldots, L^n \in \Ob_{\Mor(M_0,M_1)}$, a $\Lambda$-linear map 
\begin{align} 
\mu_n  \coloneqq \bmu^{1,1}_{(n)}  \:: \: \bigotimes_{1\leq k \leq n} 
\leftindex^2\Mor(L^{k-1}=L_{01}^{1,k-1},L_{01}^{1,k}=L^k)  \;\to\; \Mor(L^0=L_{01}^{1,0},L_{01}^{1,n}=L^n)
\end{align}
 is given by choosing $r=1$, $a=1$, $\bn^1=(n)$ in \eqref{eq:2-mu-r}.
\end{itemize}

\end{lemma}
\begin{proof}
The structures are identified in the statement, so it remains to verify the $A_\infty$-equations \eqref{eq:A_infty-equations} for any choice of $\bigl( x_k\in \Mor(L^{k-1},L^k)\bigr)_{1\leq k\leq n}$.
These will follow from the $(A_\infty,2)$-equations \eqref{eq:A_infty-2-equations} for $r=1$, $a=1$, $\bn^1=(n)$, which reduces the indices in \eqref{eq:123} to case $*_1$ with $i=1$, $j=1$, and $s,t\geq 0$ with $s+t \leq n$.
Thus only the first sum is present in this $(A_\infty,2)$-equation \eqref{eq:A_infty-2-equations}, which confirms 
\begin{align*}
\sum_{s+t\leq n}
\mu_{n-t+1}(\ldots x_s\smotimes \mu_s(x_{s+1}\smotimes \ldots x_{s+t})\smotimes x_{s+t+1} \ldots ) 
 \: = \sum_{s+t\leq n}
\bmu^{1,1}_{(n-t+1)}\left(   
 \bigotimes
\left(
\begin{array}{c}
\vdots  \\
 x_s \\
\bmu_{(t)}^{1,1}
\left(\begin{array}{c}x_{s+1} \\ \vdots \\ x_{s+t}
\end{array}\right) \\
x_{s+t+1} \\
\vdots \\
\end{array}
\right) 
\right)
= 0 .
\end{align*}
\end{proof}

This $A_\infty$-structure on the morphisms results from the $\Lambda$-linear maps $\bmu^{r,a}_{(n)}$ for $r=1$ and $a=1$.
To describe the expected algebraic meaning of these maps for $r\geq 2$ or $a\geq 2$ we need the linear algebraic analog of the notion of compatibility with fiber products in Definition~\ref{def:compatible} for flow categories.
Here the algebraic compatibility notion is restricted to $r\leq 2$ as the fiber products for $r\geq 3$ are nontrivial, so that no direct algebraic relationship between the associated $\Lambda$-linear maps can be expected.

\begin{definition}\label{def:linear compatible}
For $r_c \in \{1,2\}$, a linear $(A_\infty,2)$-category $2\sC$ as in Definition~\ref{def:A_infty-2-flow_category} is called \emph{compatible with fiber products for $r = r_c$} if for $a \geq 2$, $\underline\bn=(\bn^1, \ldots, \bn^a) \in \bigl(\bZ_{\geq0}^{r_c}\bigr)^a$, and collections of objects $M_0,\ldots,M_{r_c}\in\Ob$ and 1-morphisms as in \eqref{eq:curly-L}, 
\begin{equation} 
\cL 
\: = \:
\left(\begin{array}{c}
\cL^1 \coloneqq \left( L_{(i-1)i}^{1,k} \right)_{ 1\leq i\leq r_c , 0\leq k\leq n_i^1} 
\\ \vdots \\ 
\cL^a \coloneqq  \left( L_{(i-1)i}^{a,k} \right)_{ 1\leq i\leq r_c , 0\leq k\leq n_i^a}
\end{array}\right),
\end{equation}
the associated $\Lambda$-linear maps $\bmu^{r_c,a}_{\underline\bn}$ are tensor products of the corresponding $\Lambda$-linear maps for $a=1$, 
\begin{equation}
\bmu^{r_c,a}_{\underline\bn} 
\left( \bX = 
\bigotimes \left(
\begin{array}{c}
\bX^1  \\
\vdots \\
\bX^a  \\
\end{array}
\right)
\right)
\: =  \: 
\bigotimes
\left(
\begin{array}{c}
\bmu^{r_c,1}_{\bn^1}\bigl( \bX^1\bigr)  \\
\vdots \\
\bmu^{r_c,1}_{\bn^a}\bigl( \bX^a \bigr)  \\
\end{array}
\right).
\end{equation} 

\vspace{-8mm}
\null\hfill$\triangle$
\end{definition}

\begin{remark} \label{rmk:compatible}
The compatibility of Definition~\ref{def:linear compatible} for $r=1$ means that each $\Lambda$-linear map $\bmu^{1,a}_{\underline\bn}$ for $a\geq 2$ can be interpreted as the $a$-fold tensor product of $a$-many $A_\infty$-composition operations, where each of these composition operations is defined as in Lemma~\ref{lem:mor is A-infty cat}.
Assuming compatibility for $r=2$, the maps for $r=2,a=1$ induce a lift of the composition of morphisms to a curved $A_\infty$-bifunctor 
$\bigl( \Mor(M_0,M_1) , \Mor (M_1,M_2) \bigr) \to \Mor (M_0,M_2)$,  
as was conjectured for the symplectic applications in \cite[Conjecture~4.11]{bottman_wehrheim}.\footnote{
Indeed, the $(A_\infty,2)$-equations for $r=2$ in \eqref{eq:A_infty-2-equations} simplify as the second sum contributes no terms.
The first resp.\ third sum then correspond to the left- resp.\ right-hand-sides in \cite[Equation (40)]{bottman_wehrheim}.
} 
In the absence of curvature terms -- which are discussed below -- this associates to every 1-morphism $L_{12}\in\Mor (M_1,M_2)$ an $A_\infty$-functor $\Phi_{L_{01}}: \Mor(M_0,M_1) \to \Mor(M_0,M_2)$.

The relation between the composition of these functors $\Phi_{L_{12}}\circ \Phi_{L_{23}}$ and the functor associated to the composition $\Phi_{L_{12}\circ L_{23}}$ is then captured by the $(A_\infty,2)$-equations for $r=3$.
These provide a generalized homotopy\footnote{
The $(A_\infty,2)$-operations for $r=3$ do not provide an $A_\infty$-homotopy in the typical sense of \cite[\S1h]{seidel_picard-lefschetz} since the $(A_\infty,2)$-equations \eqref{eq:A_infty-2-equations} for $r=3$ have the third summand, indexed by $*_3$, whereas the $A_\infty$-homotopy equations have a different sum, namely the first sum on the right-hand side of \cite[(1.8)]{seidel_picard-lefschetz}.
However, the ``generalized $A_\infty$-homotopy'' here is homotopic to the usual definition by \cite{b:homotopies}
because the 2-associahedra with $r=3$ can be augmented with a union of new cells such that the augmented 2-associahedra exactly encode the $A_\infty$-homotopy equations.
} between the $A_\infty$-functors given by the maps $\bmu^{3,a}_{\underline\bn}$ -- and up to new curvature terms.
The algebraic structure of the $\bmu^{r,a}_{\underline\bn}$ for $r\geq4$ can then be viewed as iterated algebraic homotopies -- each up to yet new curvature terms, which we discuss in Remark~\ref{rmk:curvature}.
\end{remark}

In the symplectic application -- once the moduli spaces are constructed and regularized -- this realizes Weinstein's vision of a symplectic category \cite{weinstein_symplectic_category} and extends its Floer-theoretic functoriality properties from the Ma'u-Wehrheim-Woodward constructions \cite{mww} -- which were somewhat artificial and limited to monotone symplectic manifolds -- to a natural structure including all compact (or geometrically bounded) symplectic manifolds.
However, the algebraic structure of this theory involves an infinite hierarchy of new algebraic curvature terms as follows.

\begin{remark} \label{rmk:curvature}
As in the $A_\infty$-case of Definition~\ref{def:A_infty-category} we include unstable configurations of witch curves in Definition~\ref{def:A_infty-2-category}, as these arise in symplectic applications from energy concentration in pseudoholomorphic quilts without (or in addition to) degeneration of the underlying witch curve, and have important algebraic consequences.

\begin{enumerate}
\item[(i)]
The unique unstable configuration with $r=1$, $a=1$ and $\bn^1=(1)$ is a witch curve that -- after removing the incoming and outgoing marked point -- is biholomorphic to a cylinder with two seams dividing the cylinder into strips of equal width.
In symplectic applications, the corresponding pseudoholomorphic quilts can be combined into a single strip mapping to a product of symplectic manifolds with Lagrangian boundary conditions.
These moduli spaces give rise to the differential in Lagrangian Floer theory developed in \cite{floer_lag_int}.

In general, the unique unstable configuration in the trivial (2-)associahedron $W_{(1)}=K_1=\{\pt\}$ gives rise to a $\Lambda$-linear map $\mu_1: \Mor(L^0,L^1)\to  \Mor(L^0,L^1)$ in Definition~\ref{def:A_infty-category} and a $\Lambda$-linear map $\bmu^{1,1}_{(1)}:\leftindex^2\Mor(L^0_{01},L^1_{01})\to \leftindex^2\Mor(L^0_{01},L^1_{01})$ in Definition~\ref{def:A_infty-2-category}.
In the absence of algebraic curvature -- see (ii) -- this can be understood as giving each morphism space $\Mor(L^0,L^1)$ resp.\ each 2-morphism space $\leftindex^2\Mor(L^0_{01},L^1_{01})$ the structure of a chain complex.

\item[(ii)]
The unique unstable configuration with $r=1$, $a=1$ and $\bn^1=(0)$ is a witch curve with a single seam and no incoming marked points.
In symplectic applications, the corresponding pseudoholomorphic quilts can be combined into a single disk mapping to a product of symplectic manifolds with Lagrangian boundary conditions and only an outgoing marked point.
These moduli spaces give rise to the algebraic curvature in the $A_\infty$-algebra of Lagrangian Floer theory developed in \cite{fooo_12}.

In general, the unique unstable configuration in the trivial (2-)associahedron $W_{(0)}=K_0=\{\pt\}$ gives rise to algebraic curvature of the ``chain complexes'' in (i).
In Definition~\ref{def:A_infty-category} it appears as $\mu_0 \in \Mor(L^0,L^0)$ for each object $L^0$, which satisfies the zero-th $A_\infty$-equation $\mu_1(\mu_0)=0$.
The first $A_\infty$-equation 
$\mu_1(\mu_1(x)) + \mu_2(\mu_0\smotimes x) + \mu_2(x\smotimes \mu_0)$, or more precisely\footnote{Here we indicate by superscripts which object each curvature contribution is associated to.}
\begin{equation}
\mu_1(\mu_1(x)) + \mu_2(\mu_0^{L^0}\smotimes x) + \mu_2(x\smotimes \mu_0^{L^1})=0
\qquad\forall x\in \Mor(L^0,L^1) 
\end{equation}
quantifies the failure of differentials $\mu_1$ to square to zero.
In Definition~\ref{def:A_infty-2-category}, this curvature appears in the terms $\mu_0^{L_{01}}\coloneqq\bmu^{1,1}_{(0)}\in\leftindex^2\Mor(L_{01},L_{01})$ for each 1-morphism $L_{01}=L_{01}^{1,0}\in\leftindex^1\Mor(M_0,M_1)$ and satisfies the analogous relations with the differential $\bmu^{1,1}_{(1)}$ from (i).

\item[(iii)]
For $r=1$ and $a\geq 2$ the fiber products of (2-)associahedra of unstable configurations are Cartesian products of the (2-)associahedra, so that it makes sense to impose the compatibility assumption of Definition~\ref{def:linear compatible}.
The resulting algebraic contributions from unstable configurations are then tensor products of the above curvature terms, 

\begin{equation}
\bmu^{1,a}_{\tiny \left( \substack{(0) \\  \vdots  \\  (0)  } \right)} 
\: = \: 
\left( \begin{array}{c}
\bmu^{1,1}_{\tiny (0)} = \mu_0^{L^1_{01}}  \\
\otimes \\
\vdots \\
\bmu^{1,1}_{\tiny (0)} = \mu_0^{L^a_{01}}  .
\end{array} \right)
\in 
 \begin{array}{c}
\leftindex^2\Mor(L^1_{01} , L^1_{01} ) \\
\otimes \\
\vdots \\
\leftindex^2\Mor(L^a_{01},L^a_{01})  .
\end{array}
\end{equation}
However, the curvature also appears in mixed terms
\begin{equation}
\bmu^{1,2}_{\tiny \left( \substack{(1) \\  (0) } \right)} =
\otimes \left( \begin{array}{c}
\bmu^{1,1}_{(1)} = \mu_1 \\ 
\bmu^{1,1}_{(0)} = \mu_0 
\end{array} \right) , 
\quad
\bmu^{1,2}_{\tiny \left( \substack{(0) \\  (1) } \right)} =
\otimes \left( \begin{array}{c}
\bmu^{1,1}_{(0)} = \mu_0  \\ 
\bmu^{1,1}_{(1)}= \mu_1
\end{array} \right) , 
\quad
\bmu^{1,3}_{\tiny \left( \substack{(1) \\  (0)  \\  (0)  } \right)} =
\otimes \left( \begin{array}{c}
\bmu^{1,1}_{(1)} = \mu_1 \\ 
\bmu^{1,1}_{(0)} = \mu_0  \\ 
\bmu^{1,1}_{(0)} = \mu_0 
\end{array} \right)  ,  \quad \ldots 
\end{equation}
\end{enumerate}

New algebraic obstruction terms in the $(A_\infty,2)$-algebra that go beyond the differential and curvature terms in the $A_\infty$-algebra arise from the unstable configurations with $r\geq 2$ and $\bn^j=(0,\ldots,0)$, indicating no incoming marked points:
\begin{enumerate}
\item[(iv)]
For $r=2$, $a=1$ the unique underlying domain curve is the figure eight configuration introduced in \cite[Figure~2]{wehrheim_woodward_geometric_composition} -- two circular seams on the sphere that intersect tangentially at the outgoing marked point.
In symplectic applications, the corresponding pseudoholomorphic quilts arise as bubbles during strip-shrinking as described in \cite[Theorem~3.5]{bottman_wehrheim}, which do not add to the boundary codimension.
These moduli spaces capture obstructions to the ``categorification commutes with composition'' and ``Floer homology is well-defined under geometric composition'' results proven in \cite{wehrheim_woodward_2010functoriality} and \cite{wehrheim_woodward_geometric_composition} for monotone Lagrangian relations.
However, thanks to their Fredholm description \cite{bottman_figure_eight_singularity}, a regularization of these moduli spaces can now be understood as giving rise to a bounding cochain $\bmu^{2,1}_{(0,0)}\in \leftindex^2\Mor(L_{01}\circ L_{12},L_{01}\circ L_{12})$
which adjusts the Floer differential for composed Lagrangian relations $L_{01}\circ L_{12}$ such that the new adiabatic Fredholm theory for strip-shrinking \cite{bottman_wehrheim:adiabatic} can generalize the above algebraic results to all compact Lagrangian relations as outlined in \cite[\S4.4]{bottman_wehrheim}.

In general, the unique unstable configuration in the trivial 2-associahedron $W_{(0,0)}=\{\pt\}$ gives rise to ``figure eight'' terms $\bmu^{2,1}_{(0,0)}\in \leftindex^2\Mor(L_{01}\circ L_{12},L_{01}\circ L_{12})$ in Definition~\ref{def:A_infty-2-category} for each composable pair of 1-morphisms $L_{01}=L_{01}^{1,0}\in\leftindex^1\Mor(M_0,M_1)$, $L_{12}=L_{12}^{1,0}\in\leftindex^1\Mor(M_1,M_2)$.
Assuming compatibility with fiber products for $r=2$ as in Definition~\ref{def:linear compatible}, these figure eight terms determine the maps for $r=2$, $a\geq 2$, and $\bn^j=(0,0)$  
\begin{equation}
\bmu^{2,2}_{\tiny \left( \substack{(0,0) \\  (0,0) } \right)} =
\otimes \left( \begin{array}{c}
\bmu^{2,1}_{(0,0)}  \\ 
\bmu^{2,1}_{(0,0)}
\end{array} \right) , 
\qquad
\bmu^{2,3}_{\tiny \left( \substack{(0,0) \\  (0,0)  \\  (0,0)  } \right)} =
\otimes \left( \begin{array}{c}
\bmu^{2,1}_{(0,0)}  \\ 
\bmu^{2,1}_{(0,0)}  \\ 
\bmu^{2,1}_{(0,0)}
\end{array} \right)  ,  \quad \ldots 
\end{equation}
Compatibility also implies that the figure eight terms appear in the linear maps for $r=2$, $a\geq 2$, and $\underline\bn\not\equiv 0$, for example we will use below
\begin{equation}
\bmu^{2,2}_{\tiny \left( \substack{(1,0) \\  (0,0) } \right)} =
\otimes \left( \begin{array}{c}
\bmu^{2,1}_{(1,0)}  \\ 
\bmu^{2,1}_{(0,0)}
\end{array} \right) , 
\quad
\bmu^{2,2}_{\tiny \left( \substack{(0,0) \\  (1,0) } \right)} =
\otimes \left( \begin{array}{c}
\bmu^{2,1}_{(0,0)}  \\ 
\bmu^{2,1}_{(1,0)}
\end{array} \right) , 
\quad
\bmu^{2,3}_{\tiny \left( \substack{(1,0) \\  (0,0)  \\  (0,0)  } \right)} =
\otimes \left( \begin{array}{c}
\bmu^{2,1}_{(1,0)}  \\ 
\bmu^{2,1}_{(0,0)}  \\ 
\bmu^{2,1}_{(0,0)}
\end{array} \right)  ,  \quad \ldots 
\end{equation}
Still assuming compatibility with fiber products for $r=2$, we can now make algebraic sense of the figure eight terms: 
They obstruct -- or add higher curvature to -- the $A_\infty$-bifunctor relations arising from the 
$(A_\infty,2)$-equations \eqref{eq:A_infty-2-equations} for $r=2$, $a=1$.
For example, if we assume vanishing of the disk curvature $\bmu_{(0)}^{1,1}$ in (ii), then the $(A_\infty,2)$-equation for $\underline\bn=\bn^1=\bigl(n^1_1=1, n^1_2=0)$ applied to $\bX=\otimes ( x , \ )$ with $x=x_1^{1,1} \in \leftindex^2\Mor\bigl(L_{01}^{1,0},L_{01}^{1,1}\bigl) $ yields the first $A_\infty$-bifunctor relation -- making $\bmu^{2,1}_{(1,0)}$ a chain map -- up to an infinite sum arising from figure eight curvature terms, 
\begin{align}
0 &= 
 \bmu_{(1,0)}^{2,1}
 \left(
\bmu_{(1)}^{1,1} (x)
\right)
+
\bmu^{1,1}_{(1)}
\left( \bmu^{2,1}_{(1,0)} (x)  \right) 
+
\bmu^{1,1}_{(2)}
\otimes\!\! \left( 
\begin{array}{c}
\bmu^{2,1}_{(1,0) } ( x )  \\
\bmu^{2,1}_{(0,0) }  \\
\end{array}
\right)
+
\bmu^{1,1}_{(2)}
\otimes\!\! \left( 
\begin{array}{c}
\bmu^{2,1}_{(0,0)}   \\
 \bmu^{2,1}_{(1,0)}( x ) \\
\end{array}
\right) 
\nonumber \\
&\quad 
+
\bmu^{1,1}_{(3)}
\otimes\!\! \left( \begin{array}{c}
\bmu^{2,1}_{(1,0)}  ( x )  \\
\bmu^{2,1}_{(0,0)}   \\
 \bmu^{2,1}_{(0,0)}  \\
\end{array} \right) 
+
\bmu^{1,1}_{(3)}
\otimes\!\! \left( \begin{array}{c}
\bmu^{2,1}_{(0,0)}  \\
 \bmu^{2,1}_{(1,0)}  (x )  \\
 \bmu^{2,1}_{(0,0)}  \\
\end{array} \right) 
+
\bmu^{1,1}_{(3)}
\otimes\!\! \left( \begin{array}{c}
\bmu^{2,1}_{(0,0)}  \\
\bmu^{2,1}_{(0,0)} \\
  \bmu^{2,1}_{(1,0)} ( x )  \\
\end{array} \right) 
\nonumber \\
&\quad 
+
\bmu^{1,1}_{(4)}
\otimes\!\! \left( \begin{array}{c}
 \bmu^{2,1}_{(1,0)} (x )  \\
\bmu^{2,1}_{(0,0)}   \\
 \bmu^{2,1}_{(0,0)}  \\
 \bmu^{2,1}_{(0,0)}  \\
\end{array} \right) 
+
\bmu^{1,1}_{(4)}
\otimes\!\! \left( \begin{array}{c}
\bmu^{2,1}_{(0,0)}  \\
 \bmu^{2,1}_{(1,0)} ( x )  \\
\bmu^{2,1}_{(0,0)}   \\
\bmu^{2,1}_{(0,0)}  \\
\end{array} \right) 
+
\bmu^{1,1}_{(4)}
\otimes\!\! \left( \begin{array}{c}
\bmu^{2,1}_{(0,0)}  \\
\bmu^{2,1}_{(0,0)}   \\
 \bmu^{2,1}_{(0,0)} ( x ) \\
\bmu^{2,1}_{(0,0)}  \\
\end{array} \right) 
+
\bmu^{1,1}_{(4)}
\otimes\!\! \left( \begin{array}{c}
\bmu^{2,1}_{(0,0)}   \\
\bmu^{2,1}_{(0,0)}   \\
\bmu^{2,1}_{(0,0)}  \\
 \bmu^{2,1}_{(1,0)}  ( x ) \\
\end{array} \right) 
\nonumber \\
&\quad 
+ \ 
\textstyle \sum_{b\geq 5}
\sum_{\bm^1+ \ldots + \bm^b = (1,0)}  \: 
\bmu^{1,1}_{(b)}
\bigl( \bmu^{2,1}_{\bm^1 \ldots \bm^b} (x)  \bigr) .
\nonumber 
\end{align}
Convergence of this infinite sum follows in symplectic applications from the fact that figure eight bubbles have a positive minimal energy, so that contributions for $b\to\infty$ appear only in the tail of the Novikov coefficients.
For the general algebraic framework see (vi) below.

\item[(v)]
For $r \geq 3, a = 1$ the underlying domain curves are the first new examples of configurations introduced in \cite[Definition 2.4]{bottman_realizations}: $r$ circular seams on the sphere that intersect tangentially at the outgoing marked point.
In symplectic applications, the corresponding pseudoholomorphic quilts are expected to arise by the analysis of \cite{bottman_wehrheim} as bubbles during the simultaneous shrinking of several adjacent seams.
This bubbling phenomenon is also expected not to increase the boundary codimension.
The moduli spaces can be given a Fredholm description by combining \cite{bottman_figure_eight_singularity,bottman_wehrheim:adiabatic}.

In general, the 2-associahedra $W_{(0,\ldots,0)}$ for $r\geq 3$ consist of unstable configurations and can be identified with nontrivial associahedra via the forgetful maps $\pi: W_{(0,\ldots,0)}\to K_r$.
They give rise to terms $\bmu^{r,1}_{(0,\ldots,0)}\in \leftindex^2\Mor(L_{01}\circ L_{12}\circ\cdots\circ L_{(r-1)r},L_{01}\circ L_{12}\circ\cdots\circ L_{(r-1)r})$ in Definition~\ref{def:A_infty-2-category} for each composable chain of 1-morphisms $L_{01}=L_{01}^{1,0}\in\leftindex^1\Mor(M_0,M_1), \ldots, $$L_{(r-1)r}=L_{(r-1)r}^{1,0}\in\leftindex^1\Mor(M_{r-1},M_r)$.
Then a term $\bmu^{r',1}_{(0,\ldots,0)}$ for fixed $r'$ appears as algebraic obstruction / curvature in all $(A_\infty,2)$-equations for $r\geq r'$.

\item[(vi)]
For $a\geq 2$ and $r\geq 3$ the moduli spaces of unstable configurations are fiber products of the moduli spaces for $a=1$ over the nontrivial associahedron $K_r$.
Thus the algebra will generally contain further obstruction terms 
\begin{equation}
\bmu^{r,a}_{\tiny \left( \substack{(0,\ldots,0) \\  \vdots  \\  (0,\ldots,0)  } \right)} 
\in 
 \begin{array}{c}
\leftindex^2\Mor(L^1_{01}\circ L^1_{12}\circ\cdots\circ L^1_{(r-1)r} , L^1_{01}\circ L^1_{12}\circ\cdots\circ L^1_{(r-1)r}) \\
\otimes \\
\vdots \\
\leftindex^2\Mor(L^a_{01}\circ L^a_{12}\circ\cdots\circ L^a_{(r-1)r},L^a_{01}\circ L^a_{12}\circ\cdots\circ L^a_{(r-1)r})  .
\end{array}
\end{equation}
Finally, fiber products of unstable and stable configurations for $r\geq 3$ yield a further infinite hierarchy of obstructed linear maps such as 
\begin{equation}
\bmu^{3,2}_{\tiny \left( \substack{(1,0,0) \\  (0,0,0) } \right)} \: : \: 
\leftindex^2\Mor(L^{1,0}_{01},L^{1,1}_{01})
\quad\longrightarrow\quad
\begin{array}{c}
\leftindex^2\Mor(L^{1,0}_{01}\circ L^{1,0}_{12}\circ L^{1,0}_{23},L^{1,1}_{01}\circ L^{1,0}_{12}\circ L^{1,0}_{23}) \\
\otimes \\
\leftindex^2\Mor(L^{2,0}_{01}\circ L^{2,0}_{12}\circ L^{2,0}_{23},L^{2,0}_{01}\circ L^{2,0}_{12}\circ L^{2,0}_{23})  .
\end{array}
\end{equation}

\item[(vii)]
Throughout, the infinite sums resulting from insertions of the curvature and obstruction terms can be made sense of by working in a Novikov ring (or field) and using the additional property (or algebraic assumption) that the terms $\bmu^{r,a}_{(0,\ldots,0)}$ only contribute in positive energy.
\null\hfill$\triangle$
\end{enumerate}
\end{remark}

Finally, we can state and prove our main result -- a generalization of Proposition~\ref{prop:A-infinity_flow_to_linear} to $(A_\infty,2)$-categories.
That is, we explain how an $(A_\infty,2)$-flow category in any regularization framework gives rise to a linear $(A_\infty,2)$-category -- assuming the framework satisfies the requirements of Definition~\ref{def:framework}.
We then specify to the $(\cC^0,\text{Morse})$-framework of Definition~\ref{def:C0-framework} to prove that this construction satisfies the $(A_\infty,2)$-equations.
Again, we expect this proof to generalize to other regularization frameworks -- using framework-specific analogues of the fact in Lemma~\ref{lem:system of boundary faces induces identity} that a system of boundary faces induces an algebraic identity between push-pull maps.

\begin{theorem}
\label{thm:A-infinity_2_flow_to_linear}
Suppose that $2\sC$ is an $(A_\infty,2)$-flow category
as in Definition~\ref{def:A_infty-2-flow_category}.
Then it gives rise to a $\Lambda$-linear $(A_\infty,2)$-category $2\sC_\lin$ as follows:
\begin{itemize}
\item
The categories of objects and 1-morphisms are the same for $2\sC$ and $2\sC_\lin$.

\smallskip

\item
For $L_{01}, L_{01}' \in \leftindex^1\Mor(M_0, M_1)$, the 2-morphism $\Lambda$-module 
\begin{equation}
\leftindex^2\Mor_{2\sC_\lin}(L_{01},L_{01}') \,\coloneqq\: C_* \leftindex^2\Mor_{2\sC}(L_{01},L_{01}')
\end{equation}
is the chain complex associated by the $(*)$-framework to the $(*)$base space $\leftindex^2\Mor_{2\sC}(L_{01},L_{01}')$.

In the $(\cC^0,\text{Morse})$-framework the morphism $\Lambda$-module  is generated by the finite set $\leftindex^2\Mor_{2\sC}(L_{01},L_{01}')$, 
\begin{equation}
\leftindex^2\Mor_{2\sC_\lin}(L_{01},L_{01}') \,\coloneqq\: \bigoplus_{p\in\leftindex^2\Mor_{2\sC}(L_{01},L_{01}')} \Lambda \, p  .
\end{equation}

\smallskip

\item
For any integers $r \geq 1$, $a \geq 1$, tuples of integers 
$\underline \bn = (\bn^1, \ldots, \bn^a) \in \bigl(\bZ_{\geq0}^r \bigr)^a$, 
and a collection of 1-morphisms $\cL = \bigl(L_{(i-1)i}^{j,k}\bigr)_{(i,j,k)\in I^{\, r,a}_{\underline\bn}}$ as in \eqref{eq:2-collection} we define the $\Lambda$-linear map in \eqref{eq:2-mu-r} by 
\begin{align}
\bmu^{r,a}_{\underline\bn}
\,:\: 
\bigotimes_{(i,j,k) \in {'\!I^{\, r,a}_{\underline\bn}}}
\leftindex^2\Mor_{2\sC_\lin}\bigl(L_{(i-1)i}^{j,k-1},L_{(i-1)i}^{j,k}\bigr)
&\:\to\:
\bigotimes_{\scriptscriptstyle 1\leq j\leq a}
\leftindex^2\Mor_{2\sC_\lin}(L_{01}^{j,0}\circ\cdots L_{(r-1)r}^{j,0},L_{01}^{j,n_1^j}\circ\cdots L_{(r-1)r}^{j,n_r^j})
\nonumber \\
\bigotimes_{(i,j,k) \in {'\!I^{\, r,a}_{\underline\bn}} } c_i^{j,k}  \:  & \:\mapsto\: \:\bigl((\beta^j_{\cL})_{\scriptscriptstyle 1\leq j\leq a}\bigr)_* \bigl( ( a_{\cL}^{i,k})_{(i,j,k) \in {'\!I^{\, r,a}_{\underline\bn}} }\bigr)^*  \bigl( \underset{(i,j,k) \in {'\!I^{\, r,a}_{\underline\bn}} }{\times} c_i^{j,k}  \bigr).
\label{eq:2-push-pull}
\end{align}
In the $(\cC^0,\text{Morse})$-framework, this construction is given by $\Lambda$-linear extension of the map 
$\bmu^{r,a}_{\underline\bn}(\otimes \bP) \coloneqq \sum_{\underline q=(q^1,\ldots,q^a)\in\leftindex^2\Mor^{\rm out}_{2\sC}(\cL)}  \langle \bP , \underline q \rangle \, q^1\otimes \ldots\otimes q^a$ given for $\bP\in\leftindex^2\Mor^{\rm in}_{2\sC}(\cL)$
by counting the $0$-dimensional part $\cX(\cL)_0$ of $\cX(\cL)$ with Novikov coefficients.
That is, we denote 
\begin{align}
\cX(\bP,\underline q)_0 
&  \coloneqq  \cX(\cL)_0\cap \underline\alpha_\cL^{-1}(\bP) \cap \underline\beta_\cL^{-1}(\underline q)  , \qquad
\cX(\bP,\underline q; E)_0 
\coloneqq  \cX(\bP,\underline q)_0  \cap \cE^{-1}(E)
\end{align}
for any $\bP=\bigl( p_i^{j,k}\bigr)_{(i,j,k) \in {'\!I^{\, r,a}_{\underline\bn}} }\in \leftindex^2\Mor^{\rm in}_{2\sC}(\cL)$, $\underline q=(q^j)_{1\leq j \leq a} \in \leftindex^2\Mor^{\rm out}_{2\sC}(\cL)$, $E\in\bR$, and set 
\begin{align} \label{eq:2-count}
\langle \bP , \underline q\rangle
 \: \coloneqq \: 
\#_\Lambda \cX(\bP,\underline q)_0  
 \: \coloneqq \: \textstyle \sum_{l=0}^\infty  \#_{\bZ_2} \cX(\bP,\underline q, E_l)_0 \:  T^{E_l} .
\end{align}
Here $\cE( \cX(\bP,\underline q)_0  ) = \{E_0, E_1, \ldots \}$ is a discrete set with $E_l < E_{l+1}$ and finitely many elements or $\lim_{l\to\infty} E_l = \infty$ as in Remark~\ref{rmk:energies can be ordered}.

\end{itemize}
Moreover, if $2\sC$ is compatible with fiber products for $r = r_c \in \{1,2\}$
in the sense of Definition~\ref{def:compatible}, then $2\sC_\lin$ is compatible with fiber products for
$r = r_c$ in the sense of Definition~\ref{def:linear compatible}.
\end{theorem}

\begin{proof}
The proof will be given in the $(\cC^0,\text{Morse})$-framework, relying on the properties in Definitions~\ref{def:C0-framework} and \ref{def:A_infty-2-flow_category} of a regularized $(A_\infty,2)$-flow category in this framework.
We obtain well-defined  $\Lambda$-modules $\bigoplus_{p\in \leftindex^2\Mor_{2\sC}(L_{01},L'_{01})} \Lambda \, p $ since each $\leftindex^2\Mor_{2\sC}(L_{01},L'_{01})$ is a finite set.

To verify that $\langle \bP , \underline q\rangle\in\Lambda$ are well-defined coefficients in the universal Novikov field \eqref{eq:novikov}, we use the fact that by Definition~\ref{def:C0-framework}~(i) -- after restriction to the $0$-dimensional part of the $(\cC^0,\text{Morse})$-moduli space -- the energy function $\cE:\cX(\cL)_0\to\bR$ is locally constant, and $\cE^{-1}((-\infty,E])$ is compact for all $E\in\bR$.
Thus each $\cE^{-1}((-\infty,E])$ can contain only finitely many points, which have finitely many energy values $\cE\bigl(\cE^{-1}((-\infty,E])\bigr) = \{ E_0 < E_1 < \ldots < E_L \}$.
As we increase $E\to\infty$, the ordered list of energy values either ends with finitely many entries or continues with $\lim_{l\to\infty} E_l = \infty$.
Moreover, each $\cX(\bP , \underline q; E_l )_0=\cE^{-1}(E_l)$ is a finite number of points, so has a well-defined count $\#_{\bZ_2} \cX(\bP , \underline q; E_l )_0$ modulo 2.
Thus each coefficient $\langle \bP , \underline q \rangle \coloneqq \sum_{l=0}^\infty  \#_{\bZ_2} \cX(\bP , \underline q; E_l )_0 \:  T^{E_l}\in\Lambda$ in the $n$-ary composition map
is a well-defined expression in \eqref{eq:novikov}.

To verify the $(A_\infty,2)$-equations \eqref{eq:A_infty-2-equations} we first use the fact that the composition operations $\bmu^{r,a}_{\underline\bn}$ are $\Lambda$-linear, so it suffices to check the relations on generators -- for any fixed choice of integers
$r \geq 1, a\geq 1$, $\underline\bn\in \bigl(\bZ_{\geq0}^r\bigr)^a$, 
collections of objects $M_0,\ldots,M_r\in\Ob_{2\sC}$, 
1-morphisms $\cL= \bigl(L_{(i-1)i}^{j,k} \in  \leftindex^1\Mor_{2\sC}(M_{i-1},M_i) \bigr)_{(i,j,k)\in I^{\, r,a}_{\underline\bn} }$ as in \eqref{eq:2-collection}, 
and every tuple of 2-morphisms 
\begin{align*}
\bP \, 
\: = \:
\left(\begin{array}{c}
 \bP^1 \coloneqq \left( p_{i}^{1,k} \right)_{ 1\leq i\leq r , 1\leq k\leq n_i^1} 
\\ \vdots \\ 
 \bP^a \coloneqq  \left( p_{i}^{a,k} \right)_{ 1\leq i\leq r , 1\leq k\leq n_i^a}
\end{array}\right)
\;=\; \left( p^{j,k}_i \in 
\leftindex^2\Mor_{2\sC}\bigl(L_{(i-1)i}^{j,k-1},L_{(i-1)i}^{j,k}\bigr) \right)_{(i,j,k)\in {'\!I^{\, r,a}_{\underline\bn}}} 
\: \in \:
\leftindex^2\Mor^{\rm in}_{2\sC}(\cL)   .
\end{align*}
Given such choices, the $(A_\infty,2)$-equation \eqref{eq:A_infty-2-equations} is -- in the notation of \eqref{eq:123} and Remarks~\ref{rmk:2-a-b-maps} and \ref{rmk:2-arrangements and notation} and denoting $\otimes \bigl( \underline q=(q^j)_{1\leq j \leq a} \bigr) \coloneqq q^1\otimes\ldots\otimes q^a$ -- equivalent to 
\begin{align*}
0\: = \:
&
\sum_{*_1} \bmu_{\underline\bn_{*_1}}^{r,a}
\bigotimes \left(
\begin{array}{c}
\scriptstyle \bP^{[1,j-1]}  \\
\left(   \bP^j_{[1,i-1]} \otimes \bigotimes
\left({ \tiny
\begin{array}{c}
\scriptstyle  \bP_i^{j,[1,s]} \\
\bmu_{(t)}^{1,1}
\left(\scriptstyle  \bP_i^{j,[s+1,s+t]} \right) \\
\scriptstyle  \bP_i^{j,[s+t+1,n_i^j]} \\
\end{array}
}\right) 
\otimes \bP^j_{[i+1,r]} \right) \\
\scriptstyle \bP^{[j+1,a]}  
\end{array}\right)
 \\
&
+\sum_{*_2}
\bmu_{\underline\bn_{*_2} }^{r-t,a}
\left(
 \bP_{[1,s]} \otimes  
\bmu_{\underline\bm_{*_2}}^{t,b^1+\ldots b^a}
\bigl( \bP_{[s+1,s+t]} \bigr)
\otimes   \bP_{[s+t+1,r]} 
\right)
\: + \: 
\sum_{*_3} \left(
\begin{array}{c}
\substack{\scriptscriptstyle \vdots \\ \scriptscriptstyle \id } \Bigr\} \scriptscriptstyle  j-1  
\\
\scriptstyle \bmu^{1,1}_{(b)} \phantom{\int_{bot}^{top}}   \\ 
\substack{\scriptscriptstyle \id \\ \scriptscriptstyle \vdots } \Bigr\} \scriptscriptstyle a-j
\end{array}
\right)
\left( 
\bmu^{r,a+b-1}_{\underline\bn_{*_3}} 
\left(
\bP
\right)
\right) 
 \\
\: = \: &  
\sum_{*_1} \bmu_{\underline\bn_{*_1}}^{r,a}
\bigotimes\left(
\begin{array}{c}
\scriptstyle \bP^{[1,j-1]}  \\
\left( \bP^j_{[1,i-1]} \otimes \bigotimes
\left({\tiny 
\begin{array}{c}
 \scriptstyle \bP_i^{j,[1,s]} \\
{\textstyle \sum_{q\in  \leftindex^2\Mor^{\rm out}( \cL^{(1) \rm out}_{*_1} ) } \#_\Lambda 
\cX\left(\scriptstyle \bP_i^{j,[s+1,s+t]}, q \right) q } \\
\scriptstyle \bP_i^{j,[s+t+1,n_i^j]} \\
\end{array}
}\right) 
\otimes  \bP^j_{[i+1,r]} \right) \\
\scriptstyle \bP^{[j+1,a]}  
\end{array}\right)
 \\
&
+\sum_{*_2}
\bmu_{\underline\bn_{*_2} }^{r-t,a}
\left(
 \bP_{[1,s]} \otimes  
\Bigl( {\textstyle \sum_{\underline q \in  \leftindex^2\Mor^{\rm out}( \cL^{(2) \rm out}_{*_2} ) } \#_\Lambda 
\cX\bigl( \bP_{[s+1,s+t]}  , \underline q \bigr) \otimes\underline q } \Bigr)
\otimes  \bP_{[s+t+1,r]} 
\right)
\\
&
+
\sum_{*_3} \left(
\begin{array}{c} \tiny
\substack{\scriptscriptstyle \vdots \\ \scriptscriptstyle \id } \Bigr\} \scriptscriptstyle j-1 \\
\scriptstyle \bmu^{1,1}_{(b)} \phantom{\int_{bot}^{top}}   \\ 
\substack{\scriptscriptstyle \id \\ \scriptscriptstyle \vdots } \Bigr\} \scriptscriptstyle a-j \\
\end{array}
\right)
\left( 
\sum_{\underline{\tilde q} \in  \leftindex^2\Mor^{\rm out}( \cL^{(3) \rm out}_{*_3} )  }
 \#_\Lambda 
\cX\left( 
\bP
, 
\underline{\tilde q} 
 \right)
\otimes \underline{\tilde q} 
\right)
\end{align*}
Here the second equality holds by construction and $\Lambda$-linearity of the composition operations.
Applying the construction and $\Lambda$-linearity again, this is equivalent to 

\begin{align*}
0 \: = \: & 
\sum_{\substack{*_1 \\ \scriptscriptstyle  q\in \leftindex^2\Mor^{\rm out}( \cL^{(1) \rm out}_{*_1} ) \\ \scriptscriptstyle \underline {\hat q} \in \leftindex^2\Mor^{\rm out}( \cL^{(1) \rm in}_{*_1} ) }   } 
\#_\Lambda \cX\left( 
\left(
\begin{array}{c}
\scriptstyle \bP^{[1,j-1]}  \\
\left( \bP^j_{[1,i-1]}, 
\left({\tiny 
\begin{array}{c}
 \scriptstyle \bP_i^{j,[1,s]} \\
 q  \\
\scriptstyle \bP_i^{j,[s+t+1,n_i^j]} \\
\end{array}
}\right) 
,  \bP^j_{[i+1,r]} \right) \\
\scriptstyle \bP^{[j+1,a]}  
\end{array}\right) 
, 
\underline{\hat q}
\right) 
\cdot
\#_\Lambda  \cX\left(\scriptstyle \bP_i^{j,[s+1,s+t]}, q \right) 
\otimes \underline{\hat q}
 \\
&
+\sum_{\substack{*_2 \\ \scriptscriptstyle   \underline q \in \leftindex^2\Mor^{\rm out}( \cL^{(2) \rm out}_{*_2} ) \\ \scriptscriptstyle \underline {\hat q} \in  \leftindex^2\Mor^{\rm out}( \cL^{(2) \rm in}_{*_2} ) }}
\#_\Lambda \cX\left( 
\left(
\bP_{[1,s]}, 
 \underline q 
,  \bP_{[s+t+1,r]} 
\right)
, 
\underline{\hat q}
 \right)
 \cdot
 \#_\Lambda 
\cX\left( \bP_{[s+1,s+t]}  , \underline q \right)
\otimes \underline{\hat q}
\\
&
+
\sum_{\substack{*_3 \\ \scriptscriptstyle  \underline{\tilde q}  \in \leftindex^2\Mor^{\rm out}( \cL^{(3) \rm out}_{*_3} ) \\  \scriptscriptstyle \hat q \in  \leftindex^2\Mor^{\rm out}( \cL^{(3) \rm in}_{*_3} ) } }
\#_\Lambda 
\cX\left(
( \tilde q^{j+j''})_{0\leq j'' \leq b-1}
, \hat q   
\right) 
\cdot
 \#_\Lambda 
\cX\left( 
 \bP
, 
\underline{\tilde q} 
 \right)
\; 
 \ldots \tilde q^{j-1} \otimes \hat q 
\otimes \tilde q^{j+b}  \ldots 
 \\ 
\: = \: & 
\hspace{-4mm}
\sum_{\substack{*_1 \\  \scriptscriptstyle  \underline {\hat q} \in \leftindex^2\Mor^{\rm out}( \cL) }} 
\hspace{-4mm}
\#_\Lambda  \left[
\cX\left( 
\left(
\begin{array}{c}
\scriptstyle \bP^{[1,j-1]}  \\
\left(  \scriptstyle \bP^j_{[1,i-1]}, 
\left({
\begin{array}{c}
\scriptstyle  \bP_i^{j,[1,s]} \\
\scriptstyle \leftindex^2\Mor^{\rm out}( \cL^{(1) \rm out}_{*_1} )  \\
\scriptstyle \bP_i^{j,[s+t+1,n_i^j]} \\
\end{array}
}\right) 
, \scriptstyle \bP^j_{[i+1,r]} \right) \\
\scriptstyle \bP^{[j+1,a]}  
\end{array}\right) 
, 
\underline{\hat q}
\right) 
\leftindex_{\underline\alpha'}\times_{\underline\beta''}
 \cX\left(\scriptstyle \bP_i^{j,[s+1,s+t]},
 \scriptstyle\leftindex^2\Mor^{\rm out}( \cL^{(1) \rm out}_{*_1} ) \right) 
 \right]
\otimes \underline{\hat q}
 \\
&
+
\hspace{-4mm}
\sum_{\substack{*_2 \\ \scriptscriptstyle  \underline {\hat q} \in  \leftindex^2\Mor^{\rm out}( \cL) }}
\hspace{-4mm}
\#_\Lambda 
\left[
\cX\left( 
\left(
\scriptstyle \bP_{[1,s]}, 
\leftindex^2\Mor^{\rm out}( \cL^{(2) \rm out}_{*_2} )
, \scriptstyle \bP_{[s+t+1,r]} 
\right)
, 
\underline{\hat q}
 \right) 
\leftindex_{\underline\alpha'}\times_{\underline\beta''}
\cX\left(\scriptstyle \bP_{[s+1,s+t]}  , \leftindex^2\Mor^{\rm out}( \cL^{(2) \rm out}_{*_2} )  \right)
\right]
\otimes \underline{\hat q}
\\
&
+
\hspace{-8mm}
\sum_{\substack{*_3 \\  \scriptscriptstyle \underline{\tilde q} \in \leftindex^2\Mor^{\rm out}( \cL^{(3) \rm out}_{*_3} ) / \scriptscriptstyle \leftindex^2\Mor^{\rm in}( \cL^{(3) \rm in}_{*_3} )  \\ \hat q \in  \leftindex^2\Mor^{\rm out}( \cL^{(3) \rm in}_{*_3} ) } }
\hspace{-18mm}
 \#_\Lambda
 \left[
\cX\left(
{\scriptstyle \leftindex^2\Mor^{\rm in}( \cL^{(3) \rm in}_{*_3} )}
, \hat q   
\right) 
\leftindex_{\underline\alpha'}\times_{\underline\beta''}
\cX\left( 
 \bP
, 
\bigl( \ldots ,\tilde q^{j-1}, {\scriptstyle \leftindex^2\Mor^{\rm in}( \cL^{(3) \rm in}_{*_3} ) }, 
 \tilde q^{j+b} , \ldots \bigr)
 \right)
 \right]
\; 
 \ldots \tilde q^{j-1} \otimes \hat q 
\otimes \tilde q^{j+b}  \ldots 
\end{align*}
Here the last equality is of the form 
$$
\textstyle
\sum_{Q\in \leftindex^2\Mor^{\cdots}(\ldots)}
 \#_\Lambda \bigl\{ \chi' \,\big|\, \underline\alpha'(\chi')=Q \bigr\} 
 \cdot \#_\Lambda  \bigl\{\chi''  \,\big|\,  \underline\beta''(\chi'')=Q \bigr\} 
\;=\; 
\#_\Lambda  \left[ \{ \chi' \} \leftindex_{\underline\alpha'}\times_{\underline\beta''}  \{ \chi'' \}  \right] ,
$$
where we define the Novikov count of the fiber product on the right hand side by setting $\cE(\chi',\chi'') \coloneqq \cE(\chi')+\cE(\chi'')$.
This is a general feature of Novikov counts as explained in the proof of Proposition~\ref{prop:A-infinity_flow_to_linear} following \eqref{eq:Novikov sum}.
Next, the last sum over $\underline{\hat q} := \ldots \tilde q^{j-1} \otimes \hat q \otimes \tilde q^{j+b}  \ldots $ for $\underline{\tilde q} \in \leftindex^2\Mor^{\rm out}( \cL^{(3) \rm out}_{*_3} ) / \scriptscriptstyle \leftindex^2\Mor^{\rm in}( \cL^{(3) \rm in}_{*_3} )$ and $\hat q \in  \leftindex^2\Mor^{\rm out}( \cL^{(3) \rm in}_{*_3} )$ is in fact a sum over $\underline{\hat q} \in  \leftindex^2\Mor^{\rm out}( \cL)$ by Remark~\ref{rmk:2-a-b-maps}.
So \eqref{eq:A_infty-2-equations} is equivalent to the following identity for all tuples $\bP\in \leftindex^2\Mor^{\rm in}_{2\sC}(\cL)$ and $\underline{\hat q} \in  \leftindex^2\Mor^{\rm out}_{2\sC}( \cL)$
\begin{align*}
0 \: = \: & 
\sum_{*_1} 
\#_\Lambda  
\left[
\cX\left( 
\left(
\begin{array}{c}
\scriptstyle \bP^{[1,j-1]}  \\
\left(  \scriptstyle \bP^j_{[1,i-1]}, 
\left({
\begin{array}{c}
\scriptstyle  \bP_i^{j,[1,s]} \\
\scriptstyle \leftindex^2\Mor^{\rm out}( \cL^{(1) \rm out}_{*_1} )  \\
\scriptstyle \bP_i^{j,[s+t+1,n_i^j]} \\
\end{array}
}\right) 
, \scriptstyle \bP^j_{[i+1,r]} \right) \\
\scriptstyle \bP^{[j+1,a]}  
\end{array}\right) 
, 
\underline{\hat q}
\right) 
\leftindex_{\underline\alpha'}\times_{\underline\beta''}
 \cX\left(\scriptstyle \bP_i^{j,[s+1,s+t]},
 \scriptstyle\leftindex^2\Mor^{\rm out}( \cL^{(1) \rm out}_{*_1} ) \right) 
\right]
 \\
&
+
\sum_{*_2 }
\#_\Lambda 
\left[
\cX\left( 
\left(
\scriptstyle \bP_{[1,s]}, 
\leftindex^2\Mor^{\rm out}( \cL^{(2) \rm out}_{*_2} )
, \scriptstyle \bP_{[s+t+1,r]} 
\right)
, 
\underline{\hat q}
 \right) 
\leftindex_{\underline\alpha'}\times_{\underline\beta''}
\cX\left(\scriptstyle \bP_{[s+1,s+t]}  , \leftindex^2\Mor^{\rm out}( \cL^{(2) \rm out}_{*_2} )  \right)
\right]
\\
&
+
\sum_{*_3}
 \#_\Lambda
 \left[
\cX\left(
{\scriptstyle \leftindex^2\Mor^{\rm in}( \cL^{(3) \rm in}_{*_3} )}
, \hat q   
\right) 
\leftindex_{\underline\alpha'}\times_{\underline\beta''}
\cX\left( 
 \bP
, 
\bigl( \ldots ,\tilde q^{j-1}, {\scriptstyle \leftindex^2\Mor^{\rm in}( \cL^{(3) \rm in}_{*_3} ) }, 
 \tilde q^{j+b} , \ldots \bigr)
 \right) 
 \right]
\end{align*}
Now we use the facts that each boundary decomposition map intertwines the evaluation maps as specified in \eqref{eq:type-1-compatibility-with-alpha-beta}-- \eqref{eq:type-3-compatibility-with-alpha-beta}, and that the energy is additive $\cE(\varphi^{(\tau)}_{\cdots}(\chi',\chi''))=\cE(\chi')+\cE(\chi'')$ under boundary decomposition maps by Definition~\ref{def:A_infty-2-flow_category}~(ii).
This shows that \eqref{eq:A_infty-2-equations} is equivalent to
\begin{align}
\: = \: & 
\sum_{*_1} 
\#_\Lambda \; 
\varphi^{(1)}_{*_1}\left( 
\cX(\cL^{(1) \rm in}_{*_1})_0 \:\leftindex_{\underline{\alpha}'} \times_{\underline\beta''} \: \cX(\cL^{(1) \rm out}_{*_1})_0 \right)
\cap \underline\alpha_\cL^{-1}(\bP) \cap \underline\beta_\cL^{-1}(\underline{\hat q})
\label{eq:2-equations-to-prove} \\
&
+
\sum_{*_2} 
\#_\Lambda \; 
\varphi^{(2)}_{*_2}\left( 
\cX(\cL^{(2) \rm in}_{*_2})_0 \:\leftindex_{\underline{\alpha}'} \times_{\underline\beta''} \: \cX(\cL^{(2) \rm out}_{*_2})_0 \right)
\cap \underline\alpha_\cL^{-1}(\bP) \cap \underline\beta_\cL^{-1}(\underline{\hat q})
\nonumber \\
&
+
\sum_{*_3} 
\#_\Lambda \; 
\varphi^{(3)}_{*_3}\left( 
\cX(\cL^{(3) \rm in}_{*_3})_0 \:\leftindex_{\underline{\alpha}'} \times_{\underline\beta''} \: \cX(\cL^{(3) \rm out}_{*_3})_0 \right)
\cap \underline\alpha_\cL^{-1}(\bP) \cap \underline\beta_\cL^{-1}(\underline{\hat q}) .
  \nonumber \\
\: = \: & 
\: \#_\Lambda \; \partial \cX(\cL)_1 \cap \underline\alpha_\cL^{-1}(\bP) \cap \underline\beta_\cL^{-1}(\underline{\hat q})
\: = \: 
\#_\Lambda \; \partial \bigl( 
\cX(\cL)_1 \cap \underline\alpha_\cL^{-1}(\bP) \cap \underline\beta_\cL^{-1}(\underline{\hat q}) \bigr).
\nonumber
\end{align}
Here the last step follows from the bijection \eqref{eq:2-bijection}, which we will establish below.
Now the $(A_\infty,2)$-equations -- which we showed to be equivalent to \eqref{eq:2-equations-to-prove} -- hold by Lemma~\ref{lem:system of boundary faces induces identity} applied to the system of boundary faces in \eqref{eq:2-bijection}.
To establish the last step in \eqref{eq:2-equations-to-prove} we will utilize Definition~\ref{def:A_infty-2-flow_category}~(iii), which ensures that the collection of boundary decomposition maps $\varphi^{(\tau)}_{\cdots}$ for $\tau=1,2,3$ forms a system of boundary faces for the $\cC^0$-manifold $\cX(\cL)$ with boundary.
Spelling out Definition~\ref{def:mfds_with_faces} and specifying to the component $\cX(\cL)_1$ of dimension $1$ as in Definition~\ref{def:mfds_with_faces} and the proof of Proposition~\ref{prop:A-infinity_flow_to_linear}, this system of boundary faces implies the following: 
\begin{enumerate}
\item[(a)]
Each boundary decomposition map is a $\cC^0$-embedding 
\begin{equation} 
\varphi^{(\tau)}_{*_\tau}
\: \colon \quad
\cX(\cL'\coloneqq\cL^{(\tau) \rm in}_{*_\tau})_0 \:\leftindex_{\underline{\alpha}'} \times_{\underline\beta''} \: \cX(\cL''\coloneqq\cL^{(\tau) \rm out}_{*_\tau})_0 \quad \longrightarrow \quad \partial \cX(\cL)_1 .
\end{equation}
\item[(b)]
Each point of $\partial\cX(\cL)_1$ lies in the image of exactly one boundary decomposition map $\varphi^{(\tau)}_{*_\tau}$.
\end{enumerate}
Thus the collection of maps $\varphi^{(\tau)}_{*_\tau}$ restricted to the $0$-dimensional components of their domains  induces a bijection -- which can again be viewed as a system of boundary faces for $\partial\cX(\cL)_1$, 
\begin{align}
\label{eq:2-bijection}
\bigsqcup_{\tau=1,2,3} \; \bigsqcup_{*_\tau} \; 
\cX(\cL^{(\tau) \rm in}_{*_\tau})_0 \:\leftindex_{\underline{\alpha}'} \times_{\underline\beta''} \: \cX(\cL^{(\tau) \rm out}_{*_\tau})_0
\quad\overset{\sim}{\longrightarrow}\quad
\partial\cX(\cL)_1.
\end{align}
This proves the $(A_\infty,2)$-equations -- equivalent to \eqref{eq:2-equations-to-prove} -- via Lemma~\ref{lem:system of boundary faces induces identity} applied to this bijection.

Finally, assume that the $(A_\infty,2)$-flow category $2\sC$ is compatible with fiber products for
$r = r_c \in \{1,2\}$ in the sense of Definition~\ref{def:compatible} and \eqref{eq:fiber compatible energy} for the $(\cC^0,\text{Morse})$-framework.
In case $r_c=1$ this means that 
for any $M_0,M_1\in\Ob$, 
integers $a \geq 2$, $\underline\bn=\bigl(\bn^1=(n^1_1), \ldots, \bn^a=(n^a_1) \bigr) \in \bigl(\bZ_{\geq0}\bigr)^a$, 
and 1-morphisms 
\begin{equation} 
\cL \: = \:
\left(\begin{array}{c}
 \cL^1 = \left( L_{01}^{1,k} \right)_{  0\leq k\leq n_1^1} 
\\ \vdots \\ 
\cL^a = \left( L_{01}^{a,k} \right)_{ 0\leq k\leq n_1^a}
\end{array}\right)
\qquad
\text{we have}
\qquad
\begin{array}{rl} 
\textstyle \cX(\cL) & = \: \cX(\cL^1)\times \ldots\times \cX(\cL^a) ,  \\
\underline\alpha_\cL  &= \: \underline\alpha_{\cL^1}\times \ldots \times \underline\alpha_{\cL^a}  \\ 
\underline\beta_\cL  &= \: \underline\beta_{\cL^1}\times \ldots \times \underline\beta_{\cL^a} \\
p_\cL & = \: p_{\cL^1}\times \ldots \times p_{\cL^a} \\
\cE(\chi^1 \ldots \chi^a)& =\: \cE(\chi^1)+\ldots+\cE(\chi^a)  .
\end{array} 
\end{equation}
In case $r_c=2$ the compatibility ensures that the same holds for any 
$M_0,M_1,M_2 \in\Ob$, integers $\underline\bn=\bigl(\bn^1=(n^1_1,n^1_2), \ldots, \bn^a=(n^a_1,n^a_2) \bigr) \in \bigl(\bZ_{\geq0}^2\bigr)^a$ and 1-morphisms 
$\cL = \bigl( \cL^1 , \ldots, \cL^a \bigr)$ consisting of 
$\cL^j = \bigl( L_{(i-1)i}^{j,k} \bigr)_{  1\leq i \leq 2, 0\leq k\leq n_i^j}$.
This compatibility implies that for any generators $\bP=(\bP^1,\ldots,\bP^a)\in \leftindex^2\Mor^{\rm in}_{2\sC}(\cL)$ and $\underline q=(q^1,\ldots,q^a)\in \leftindex^2\Mor^{\rm out}_{2\sC}(\cL)$ we have  
$\cX(\bP,\underline q)  =  \cX(\bP^1, q^1) \times \ldots\times \cX(\bP^a, q^a)$.
Since dimensions in Cartesian products are additive and negative dimensional components are empty, we also obtain a product identity for the $0$-dimensional component
$\cX(\bP,\underline q)_0  =  \cX(\bP^1, q^1)_0 \times \ldots\times \cX(\bP^a, q^a)_0$.
Furthermore, since the energy is additive, the Novikov count of this Cartesian product is the product of Novikov counts by Lemma~\ref{lem:Novikov product}, 
\begin{equation}
\#_\Lambda \bigl( \cX(\bP^1, q^1)_0 \times \ldots\times \cX(\bP^a, q^a)_0 \bigr) 
\: = \: 
\#_\Lambda \cX(\bP^1, q^1)_0 \cdot \ldots\cdot \#_\Lambda \cX(\bP^a, q^a)_0  .
\end{equation}
With that we can establish the algebraic compatibility of Definition~\ref{def:linear compatible} on any tuple of generators $\bP=(\bP^1 , \ldots, \bP^a)$ 
\begin{align}
& \bmu^{r,a}_{\underline\bn} 
\left( \otimes \bP = 
\bigotimes \left(
\begin{array}{c}
\otimes \bP^1  \\
\vdots \\
\otimes \bP^a  \\
\end{array}
\right)
\right) 
\nonumber\\
\: & =  \: 
\sum_{\underline q=(q^1,\ldots,q^a)\in\leftindex^2\Mor^{\rm out}_{2\sC}(\cL)}  \#_\Lambda \cX(\bP,\underline q)_0    \; q^1\otimes \ldots\otimes q^a 
 \\ 
 \: & =  \: 
\sum_{q^1\in\leftindex^2\Mor^{\rm out}_{2\sC}(\cL^1)} \ldots \sum_{q^a\in\leftindex^2\Mor^{\rm out}_{2\sC}(\cL^a)}  \#_\Lambda \cX(\bP^1,q^1)_0 \cdot \ldots \cdot  \#_\Lambda \cX(\bP^a, q^a)_0   \; q^1\otimes \ldots\otimes q^a 
\nonumber \\ 
 \: & =  \: 
\left(\sum_{q^1\in\leftindex^2\Mor^{\rm out}_{2\sC}(\cL^1)} \#_\Lambda \cX(\bP^1,q^1)_0 \; q^1 \right) \otimes \ldots \otimes  \left( 
\sum_{q^a\in\leftindex^2\Mor^{\rm out}_{2\sC}(\cL^a)} 
\#_\Lambda \cX(\bP^a, q^a)_0   \; q^a  \right)
\nonumber \\ 
\: & =  \: 
\bmu^{r,1}_{\bn^1}\bigl( \bP^1\bigr)  \otimes
\ldots \otimes
\bmu^{r,1}_{\bn^a}\bigl( \bP^a \bigr) 
\: = \: 
\bigotimes
\left(
\begin{array}{c}
\bmu^{r,1}_{\bn^1}\bigl( \bP^1\bigr)  \\
\vdots \\
\bmu^{r,1}_{\bn^a}\bigl( \bP^a \bigr)  \\
\end{array}
\right).
\nonumber
\end{align} 
\end{proof}

\begin{remark}
\begin{enumerate}
\item[(i)]
As in Remark~\ref{rmk:characteristic and regularization}, 
Lemma~\ref{lem:system of boundary faces induces identity} relies on the coefficient ring (or field) $\Lambda$ having characteristic 2.
We could drop this hypothesis at the expense of introducing signs into the $(A_\infty,2)$-equations, and would need to assume that the morphism spaces in the $(A_\infty,2)$-flow category carry orientations that are taken into account in the counting \eqref{eq:2-count}
and are compatible with the bijection \eqref{eq:2-bijection}.
\item[(ii)]
When constructing algebraic structures by counting pseudoholomorphic quilts, a key step is to regularize the relevant moduli spaces.
This step is invisible in the proof of Theorem~\ref{thm:A-infinity_2_flow_to_linear} as the regularization step is part of constructing a regularized $(A_\infty,2)$-flow category in an appropriate regularization framework.
Once this is done, extracting a linear $(A_\infty,2)$-category is essentially a formal procedure.\null\hfill$\triangle$
\end{enumerate}
\end{remark}

\bibliographystyle{alpha}
\small
\bibliography{biblio}

\end{document}